\documentclass{amsart}

\usepackage{enumerate}
\usepackage{amssymb, amsmath}
\usepackage{mathrsfs}
\usepackage[active]{srcltx}
\usepackage{verbatim}

\usepackage{newsymbol}

\let\mathcal \undefined
\def\mathcal{\mathscr}

\let\emptyset \undefined
\let\ge       \undefined
\let\le       \undefined
\newsymbol\le          1336  \let\leq\le
\newsymbol\ge          133E  \let\geq\ge
\newsymbol\emptyset    203F
\newsymbol\notle       230A
\newsymbol\notge       230B

\theoremstyle{plain}
\newtheorem{theorem}{Theorem}[section]
\newtheorem{corollary}[theorem]{Corollary}
\newtheorem{lemma}[theorem]{Lemma}
\newtheorem{proposition}[theorem]{Proposition}

\theoremstyle{remark}
\newtheorem{remark}[theorem]{Remark}
\newtheorem{definition}[theorem]{Definition}
\newtheorem{example}[theorem]{Example}

\numberwithin{equation}{section}

\def\R{{\mathbb R}}
\def\C{{\mathbb C}}

\newcommand{\E}{{\mathbb E}}
\renewcommand{\P}{{\mathbb P}}
\newcommand{\F}{{\mathscr F}}
\renewcommand{\H}{{\underline{H}}}

\newcommand{\HH}{\mathscr{H}}

\renewcommand{\a}{\alpha}

\newcommand{\ga}{\gamma}
\newcommand{\Ga}{\Gamma}

\newcommand{\e}{\varepsilon}
\renewcommand{\l}{\lambda}
\newcommand{\om}{\omega}
\renewcommand{\O}{\Omega}
\newcommand{\Om}{\Omega}
\renewcommand{\S}{\Sigma}
\newcommand{\eps}{\varepsilon}
\newcommand{\sgn}{{\rm sgn}}
\renewcommand{\th}{\theta}

\renewcommand{\L}{L^2(0,T)}

\newcommand{\beq}{\begin{equation}}
\newcommand{\eeq}{\end{equation}}
\newcommand{\bal}{\begin{aligned}}
\newcommand{\eal}{\end{aligned}}
\newcommand{\ben}{\begin{enumerate}}
\newcommand{\beni} {\begin{enumerate}[(i)]}
\newcommand{\een}{\end{enumerate}}
\newcommand{\bit}{\begin{itemize}}
\newcommand{\eit}{\end{itemize}}
\newcommand{\beqw}{\begin{equation*}}
\newcommand{\eeqw}{\end{equation*}}
\newcommand{\bthm}{\begin{theorem}}
\newcommand{\ethm}{\end{theorem}}
\newcommand{\bpr}{\begin{proposition}}
\newcommand{\epr}{\end{proposition}}
\newcommand{\ble}{\begin{lemma}}
\newcommand{\ele}{\end{lemma}}
\newcommand{\blem}{\begin{lemma}}
\newcommand{\elem}{\end{lemma}}
\newcommand{\bpf}{\begin{proof}}
\newcommand{\epf}{\end{proof}}
\newcommand{\bex}{\begin{example}}
\newcommand{\eex}{\end{example}}
\newcommand{\bre}{\begin{example}}
\newcommand{\ere}{\end{example}}
\newcommand{\bma}{\begin{bmatrix}}
\newcommand{\ema}{\end{bmatrix}}

\renewcommand{\Re}{\hbox{\rm Re}}
\renewcommand{\Im}{\hbox{\rm Im}}
\newcommand{\Dom}{{\mathsf D}}
\newcommand{\Homdom}{\dot{\mathsf D}}
\newcommand{\D}{{\mathsf D}_p}

\newcommand{\Dq}{{\mathsf D}_q}

\newcommand{\wt}{\widetilde}

\newcommand{\calA}{{\mathscr A}}
\newcommand{\calL}{{\mathscr L}}

\newcommand{\n}{\Vert}
\newcommand{\one}{{{\bf 1}}}
\newcommand{\embed}{\hookrightarrow}
\newcommand{\s}{^*}
\newcommand{\lb}{\langle}
\newcommand{\rb}{\rangle}

\newcommand{\limn}{\lim_{n\to\infty}}

\newcommand{\da}{\downarrow}

\newcommand{\ip}[1]{\langle {#1}\rangle}

\newcommand{\Ran}{\mathsf{R}}
\newcommand{\Null}{\mathsf{N}}
\newcommand{\lin}{\mathrm{lin}}

\renewcommand{\Re}{{\rm{Re}}\;}
\renewcommand{\Im}{{\rm{Im}}\;}

\newcommand{\ov}{\overline}
\newcommand{\ot}{\otimes}
\newcommand{\sot}{{\hbox{\tiny\textcircled{s}}}}

\newcommand{\ua}{\underline{a}}
\newcommand{\uA}{\underline{A}}

\newcommand{\ul}{\underline{l}}
\newcommand{\uL}{\underline{L}}

\newcommand{\uQ}{\underline{Q}}

\newcommand{\uD}{D_V}
\newcommand{\uuD}{\underline{D}_V}
\renewcommand{\DH}{D}
\newcommand{\AB}{A}
\newcommand{\uAB}{\uA}
\newcommand{\LB}{L}
\newcommand{\uLB}{\uL}

\newcommand{\PB}{P}
\newcommand{\uPB}{\uP}
\newcommand{\SB}{S}
\newcommand{\uSB}{\uS}
\newcommand{\QB}{Q}
\newcommand{\uQB}{\uQ}
\newcommand{\PiB}{\Pi}
\newcommand{\TB}{T}

\newcommand{\uP}{\underline{P}}
\newcommand{\uS}{\underline{S}}

\renewcommand{\L}{\mathscr{L}}

\begin{document}

\title[Elliptic operators on abstract Wiener spaces]
{Boundedness of Riesz transforms for elliptic operators on abstract
Wiener spaces}

\author{Jan Maas}
\address{Delft Institute of Applied Mathematics\\
Delft University of Technology\\ P.O. Box 5031\\ 2600 GA Delft\\The
Netherlands}
\email{J.Maas@tudelft.nl}

\author{Jan van Neerven}
\address{Delft Institute of Applied Mathematics\\
Delft University of Technology\\ P.O. Box 5031\\ 2600 GA Delft\\The
Netherlands}
\email{J.M.A.M.vanNeerven@tudelft.nl}

\thanks{The authors are supported by VIDI subsidy 639.032.201
(JM and JvN) and VICI subsidy 639.033.604 (JvN)
of the Netherlands Organisation for Scientific Research (NWO).
The first named author
acknowledges partial support by the ARC Discovery Grant DP0558539.}



\begin{abstract} Let $(E,H,\mu)$ be an abstract Wiener space
and let  $\uD := V\DH$, where $\DH$ denotes the Malliavin
derivative and $V$ is a closed and densely defined operator
from $H$ into another Hilbert space $\H$. Given a bounded
operator $B$ on $\H,$ coercive on the range $\ov{\Ran(V)}$, we
consider the operators $\AB:= V\s BV$ in $H$ and $\uAB:= VV\s
B$ in $\H$, as well as the realisations of the operators $\LB:
= \uD\s B\uD$ and $\uLB := \uD\uD\s B$ in $L^p(E,\mu)$ and
$L^p(E,\mu;\H)$ respectively, where
$1<p<\infty$. Our main result asserts that the following four assertions are equivalent:

\medskip
\ben
\item[\rm(1)] $\Dom(\sqrt{\LB}) = \Dom(\uD)$ with $\n \sqrt{\LB}f\n_{p} \eqsim \n \uD
f\n_{p}$ for $f\in \Dom(\sqrt{\LB})$;
\item[\rm(2)] $\uLB$ admits a bounded $H^\infty$-functional calculus on
$\overline{\Ran(\uD)}$;
\item[\rm(3)] $\Dom(\sqrt{\AB}) = \Dom(V)$ with
  $\n \sqrt{\AB}h\n \eqsim \n Vh \n$ for $h\in \Dom(\sqrt{\AB})$;
\item[\rm(4)] $\uAB$ admits a bounded $H^\infty$-functional calculus on
$\overline{\Ran(V)}$. \een Moreover, if these conditions are
satisfied, then $\Dom(\LB) = \Dom(D_V^2) \cap \Dom(D_A)$.

The equivalence (1)--(4) is a nonsymmetric generalisation of
the classical Meyer inequalities of Malliavin
calculus (where $\H=H$, $V = I$, $B = \frac12 I$). A  one-sided
version of (1)--(4), giving $L^p$-boundedness of the Riesz
transform $\uD/\sqrt{\LB}$ in terms of a square function estimate,
is also obtained.

As an application let $-A$ generate an analytic $C_0$-contraction
semigroup on a Hilbert space $H$ and let $-L$ be the $L^p$-realisation
of the generator of its second quantisation. Our results
imply that two-sided bounds for
the Riesz transform of $L$ are equivalent with the Kato
square root property for $A$.
\end{abstract}

\dedicatory{Dedicated to Professor Alan M$^{\rm c}$Intosh on the occasion of his
65th  birthday}

\date\today

\maketitle

\section{Introduction}

%%%%%%%%%%%%%%%%%%%%%%%%%%%%%%%%%%%

Let $(E,H,\mu)$ be an abstract Wiener space, i.e., $E$ is a
real Banach space and $\mu$ is a centred Gaussian Radon
measure on $E$ with reproducing kernel Hilbert space $H$. In
this paper we prove square function estimates and boundedness
of Riesz transforms for abstract second order elliptic
operators $L$ in divergence form acting on $L^p(E,\mu)$,
$1<p<\infty$. Our main result (Theorem \ref{thm:Kato}) gives
necessary and sufficient conditions for the domain equality
\beq\label{eq:Kato} \Dom(\sqrt{\LB}) = \Dom(\uD) \eeq in
$L^p(E,\mu)$ with equivalence of norms
\beq\label{eq:equiv-norms} \n \sqrt{\LB} f\n_p \eqsim \n \uD
f\n_p \eeq for a class of divergence form elliptic operators
of the form $$\LB = \uD\s B \uD.$$ Here $\uD := V \DH,$ where
$\DH$ is the Malliavin derivative in the direction of $H$,
$V:\Dom(V)\subseteq H \to \H$ is a closed and densely defined
operator, and $B$ is a bounded operator on $\H$ which is
coercive on $\ov{\Ran(V)}$. Our main result asserts that \eqref{eq:Kato} and
\eqref{eq:equiv-norms} hold if and only if the sectorial
operator $\uAB:= VV^*B$ on $\H$ admits a bounded
$H^\infty$-functional calculus. In particular, if
\eqref{eq:Kato} and \eqref{eq:equiv-norms} hold for one
$1<p<\infty$, then they hold for all $1<p<\infty$. By
well-known examples, cf. \cite[Theorem 4 and its
proof]{McIYag}, sectorial operators on $\H$ of the form $TB$
with $T:\Dom(T)\subseteq \H\to\H$ positive and self-adjoint
and $B$ coercive on $\H$ need not always have a bounded
$H^\infty$-calculus. In our setting, such examples can be
translated into examples of operators $\LB$ for which
\eqref{eq:equiv-norms} fails (e.g., take $H=\H$ and $V =
\sqrt{T}$).

Returning to \eqref{eq:equiv-norms}, we shall prove the more
precise result that the inclusions
$$
\Dom(\sqrt{\LB}) \embed \Dom(\uD),
$$
respectively
$$
\Dom(\sqrt{\LB}) \hookleftarrow \Dom(\uD),
$$
hold in $L^p(E,\mu)$ if and only if the operator $\uAB$ satisfies a
lower, respectively upper square function estimate in
$\ov{\Ran(V)}$.

The simplest example to which our results apply is the classical
Ornstein-Uhlen\-beck operator of Malliavin calculus. This example is
obtained by taking $\H=H$, $V=I$, and $B = \frac12 I$. With these
choices, $\uD$ reduces to the Malliavin derivative $D$ in the
direction of $H$ and $L = \frac12 \DH\s \DH $ is the classical
Ornstein-Uhlenbeck operator. The equivalences \eqref{eq:Kato} and
\eqref{eq:equiv-norms} then reduce to the classical Meyer
inequalities \cite{Me82}. Various proofs of these inequalities have been given; see, e.g., \cite{Gun, Pis88}.
For further references on this subject we refer to
Nualart \cite{Nua}.

A second and non-trivial application concerns the computation
of the $L^p$-domains of second quantised operators. Let
$(E,H,\mu)$ be an abstract Wiener space and suppose that
$S=(S(z))_{z\in \overline\Sigma}$ is an analytic
$C_0$-contraction semigroup defined on the closed sector
$\overline\Sigma$. By this we mean that $S$ is a
$C_0$-semigroup of contractions on $\overline\Sigma$ which is
analytic on the interior of $\overline\Sigma$. Let $-A$ denote
the generator of $S$. For $1<p<\infty$, by second quantisation
(see Section \ref{sec:examples} for the details) we obtain an
analytic $C_0$-contraction semigroup
$(\Gamma(S(t)))_{z\in\overline\Sigma}$, with generator $-L$,
on $L^p(E,\mu)$. As we will show, the operators $A$ and $L$
are always of the form $A = V\s BV $ and $L = \uD\s B\uD$ for
suitable choices of $V$ and $B$, and Theorem \ref{thm:Kato}
implies that
 \begin{align}\label{eq:Riesz}
  \Dom(\sqrt{L}) = \Dom(\uD) \ \ \hbox{with} \ \ \n \sqrt{L}f\n_p \eqsim \n \uD
  f\n_p,
\quad 1<p<\infty,
 \end{align} if and only if $A$
admits a bounded $H^\infty$-calculus on the homogeneous form
domain associated with $A$. As before, one-sided versions of
this result can be formulated in terms of square function
estimates. By restricting \eqref{eq:Riesz} to the first
Wiener-It\^o  chaos of $L^p(E,\mu)$ (see section
\ref{sec:examples}) and using that the $L^p$-norms are
pairwise equivalent on every chaos, we see that a necessary
condition for \eqref{eq:Riesz} is given by
\begin{align*}
  \Dom(\sqrt{A}) = \Dom(V) \ \ \hbox{with} \ \ \n \sqrt{A}h\n \eqsim \n
  Vh\n.
\end{align*}
Since $\Dom(V)$ equals the domain of the form associated with
$A$, this is nothing but Kato's square root property for $A.$
Thus our main result asserts that this necessary condition is
also sufficient.

Second quantised operators arise naturally as generators of
transition semigroups  associated with solutions of linear
stochastic evolution equations with additive noise; see for
instance \cite{CG96, CG01}. In a forthcoming paper we shall
apply our results to obtain Meyer inequalities for
non-symmetric analytic Ornstein-Uhlenbeck operators in
infinite dimensions. These extend previous results of
Shigekawa \cite{Sh92} and Chojnowska-Michalik and Goldys
\cite{CG01} for the symmetric case, and of Metafune, Pr\"uss,
Rhandi, and Schnaubelt \cite{MPRS02} for the finite
dimensional case, and they improve results of \cite{MN2}
where a slightly more general class of non-symmetric
Ornstein-Uhlenbeck operators was considered.

Preliminary versions of this paper have been presented during
the Semester on Stochastic Partial Differential Equations at
the Mittag-Leffler institute (Fall 2007) and the 8th
International Meeting on Stochastic Partial Differential
Equations and Applications in Levico Terme (January 2008).

\section{Statement of the main  results}
\label{sec:elliptic}

The domain, kernel, and range of a (possibly unbounded) linear operator $T$ are
denoted by $\Dom(T)$, $\Null(T)$, and $\Ran(T)$, respectively.
When considering an operator $T$ acting consistently on a scale of
(vector-valued) $L^p$-spaces,
$\D(T)$, $\Ran_p(T)$, and $\Null_p(T)$ denote the domain, range, and kernel of
the $L^p$-realisation of $T$.

We introduce the setting studied in this paper in the form of a list of assumptions which will be in force throughout the paper,
with the exception of the intermediate Sections
\ref{sec:examples}, \ref{sec:preliminaries1}, and \ref{sec:preliminaries2}.

\medskip\noindent
{\bf Assumption (A1).}
{\em $(E,H,\mu)$ is an abstract Wiener space}.

\medskip\noindent
More precisely, we assume that $E$ is real Banach space, $H$ is a real Hilbert
space with inner product $[\cdot,\cdot]$, and $\mu$ is a centred Gaussian Radon measure on $E$ with reproducing
kernel Hilbert space $H$. Recall that this implies that $H$ is continuously
embedded in $E$; we shall write $i: H \hookrightarrow E$ for the inclusion
mapping.
The covariance operator of $\mu$ equals $i\circ i\s$ (here and in what follows, we identify $H\s$ and $H$ via the Riesz
representation theorem).

For $h \in H$ we may define a linear function $\phi_h : iH \to \R$ by
$\phi_h(ig) := [h,g].$ Although $\mu(iH) = 0$ if $H$ is
infinite dimensional \cite[Theorem 2.4.7]{Bo}, there exists a $\mu$-measurable
linear extension $\phi_h: E\to \R$ which is uniquely defined
$\mu$-almost everywhere \cite[Theorem 2.10.11]{Bo}. Note that for
$x^* \in E^*$ we have $\phi_{i^*x^*}(x) = \ip{x,x^*}$ $\mu$-almost everywhere.
The
identity
$$\int_E \lb x,x\s\rb^2\,d\mu(x) = \n i\s x\s\n^2, \quad x\s\in E\s,$$
shows that $h\mapsto \phi_h$, as a mapping
from $H$ into $L^2(E,\mu)$, is an isometric embedding.

\medskip\noindent
{\bf Assumption (A2).}
{\em $V$ is a closed and densely defined linear operator
from $H$ into another real Hilbert space $\H$.}

\medskip\noindent
When $H_0$ is a linear subspace of $H$ and $k\ge 0$ is an integer,
we let $\F C_{\rm b}^k(E;H_0)$
denote the vector space of all $(\mu$-almost everywhere defined)
functions $f : E \to \R$ of the form
$$
f(x) := \varphi(\phi_{h_1}(x) , \ldots, \phi_{h_n}(x))
$$
with $n\ge 1$, $\varphi\in C_{\rm b}^k(\R^n),$ and $h_1,
\dots, h_n\in H_0$. Here $C_{\rm b}^k(\R^n)$ is the space
consisting of all bounded continuous functions having bounded
continuous derivatives up to order $k$. In case $H_0=H$ we
simply write $\F C_{\rm b}^k(E)$. For $f \in \F C_{\rm
b}^1(E;\Dom(V))$ as above  the  gradient $\uD f$ `in the
direction of $V$' is defined by
$$
\uD  f(x) := V (\DH f(x)) = \sum_{j=1}^n \partial_j
\varphi(\phi_{h_1}(x) , \ldots, \phi_{h_n}(x) )\ot Vh_j,
$$
where $\DH$ denotes the Malliavin derivative and $\partial_j
\varphi$ denotes the $j$-th partial derivative of $\varphi$.

We shall write
$$L^p := L^p(E,\mu), \quad \uL^p:= L^p(E,\mu;\H)$$ for
brevity. As in \cite[Theorem 3.5]{GGvN03}, the proof of which can be repeated
almost {\em verbatim}, the operator $\uD$ is closable as an operator from $L^p$ into
$\uL^p$ for all $1\leq p<\infty$. From now on, $\uD$ denotes its closure;
domain and range of this closure will be denoted by $\D(\uD)$ and
$\Ran_p(\uD)$ respectively.

\medskip\noindent
{\bf Assumption (A3).} {\em $B$ is a bounded operator on $\H$
which is coercive on $\ov{\Ran(V)}$.}

\medskip\noindent
More precisely, $B$ is a bounded operator on $\H$ which satisfies
the coercivity
condition
$$[BV h,Vh] \geq k \|Vh\|^2, \qquad h \in \Dom(V),$$
where $k>0$ is a constant independent of $h\in \Dom(V)$.
Clearly $B$ satisfies (A3) if and only if $B\s$ satisfies (A3).

At this point we pause to observe that the assumptions (A1), (A2), (A3)
continue to hold after complexifying. In what follows we shall be
mostly dealing with the complexified operators, which we do not
distinguish notationally from their real counterparts as this would
only overburden the notations. The complexified version of (A3) reads
$$\Re[BVh,Vh] \geq k \|Vh\|^2, \qquad h \in \Dom(V).$$

If (A1), (A2), (A3) hold, the operator $$\LB := \uD\s B \uD$$ is well
defined, and $-\LB$ generates an analytic $C_0$-contraction semigroup
$(\PB(t))_{t\ge 0}$ on $L^p$ for all $1<p<\infty$, which
coincides with the second quantisation of the analytic
$C_0$-contraction semigroup on $H$ generated by $-\AB$, where
$$ \AB := V\s B V $$ (Theorem \ref{thm:2nd-q}). In the converse direction
we show that every second quantised analytic $C_0$-contraction
semigroup arises in this way (Theorem \ref{thm:repr}).

It is not hard to see (Proposition \ref{prop:sectL}) that the
operator $\uAB := VV\s B$ is sectorial on $\H$ and that $-\uAB$
generates a bounded analytic $C_0$-semigroup on this space.
Associated with this operator is the operator $\uLB = \uD \uD^* B$ on
$\overline{\Ran_p(\uD)}$. This operator is well defined and sectorial
 on $\overline{\Ran_p(\uD)},$ and $-\uLB$ generates a bounded analytic
$C_0$-semigroup on this space (Theorem \ref{thm:uLsemigroup}
and Definition \ref{def:uPB}).

The main results of this paper read as follows.

\begin{theorem}[Domain of $\sqrt{\LB}$]\label{thm:Kato}
Assume {\rm(A1)}, {\rm (A2)}, {\rm (A3)}, and let
$1<p<\infty$.
The following assertions are equivalent:
\ben
\item[\rm(1)] $\D(\sqrt{\LB}) = \D(\uD)$ with $\n \sqrt{\LB}f\n_{p} \eqsim \n \uD
f\n_{p}$ for $f\in \D(\sqrt{\LB})$;
\item[\rm(2)] $\uLB$ admits a bounded $H^\infty$-functional calculus on
$\overline{\Ran_p(\uD)}$;
\item[\rm(3)] $\Dom(\sqrt{\AB}) = \Dom(V)$ with
  $\n \sqrt{\AB}h\n \eqsim \n Vh \n$ for $h\in \Dom(\sqrt{\AB})$;
\item[\rm(4)] $\uAB$ admits a bounded $H^\infty$-functional calculus on $\overline{\Ran(V)}$.
\een
\ethm

For the precise definition of operators admitting a bounded
$H^\infty$-functional calculus we refer to Section
\ref{sec:preliminaries2}.
 Moreover we shall see
(Lemma \ref{lem:HilbertspaceKato}) that if the equivalent conditions of Theorem
\ref{thm:Kato}
hold, then their analogues where $B$ is replaced by $B\s$ hold as well.
Some further equivalent conditions are given at the end of Section
\ref{sec:Kato}.

 \bthm[Domain of $\LB$]\label{thm:DL}
 Let $1<p<\infty$ and let the
 equivalent conditions of Theorem \ref{thm:Kato} be satisfied.
Then $$\D(\LB) = \D(D_V^2) \cap
\D(D_A)$$ with equivalence of norms
 \begin{align*}
 \|f\|_p +  \|\LB f \|_p
 \eqsim
 \|f\|_p + \| \uD f \|_p + \| \uD^2 f \|_p + \|  D_{\AB} f \|_p.
 \end{align*}
 \ethm

Here $D_A = AD$, where $D$ is the
Malliavin derivative in the direction of $H$.
For the precise definitions of $\uD^2$ and $D_A$ we refer
to Section \ref{sec:DL}
where Theorem \ref{thm:DL} is proved.

The conditions of Theorem \ref{thm:Kato} are automatically
satisfied in each of the following two cases:
 \beni
 \item $B$ is self-adjoint.
  In this case $\AB$ is self-adjoint and
   therefore (3) holds by the theory  of symmetric forms (since $\AB$ is associated
   with a closed symmetric form with domain $\Dom(V)$; see Section \ref{sec:DBD}).
 \item $V$ has finite dimensional range.
In this case (4) is satisfied; since $\uAB$ is injective on
the (closed) range of $V$ (see Lemma \ref{lem:uABinjective}), the
$H^\infty$-functional calculus of $\uAB$ is given by the
Dunford calculus.
 \een

In fact we shall prove the stronger result that one-sided
inclusions in (1) and (3) of Theorem \ref{thm:Kato} hold if and only if $\uAB$ and/or
$\uLB$ satisfies a corresponding square function  estimate. In
particular, $L^p$-boundedness of the Riesz transform
$$\n (\uD/\sqrt{\LB}) f\n_{p} \lesssim \n f\n_{p}
$$
is characterised by the square function estimate
$$
 \|u\| \lesssim  \Big( \int_0^\infty \| t \uAB\, \uSB(t) u\|^2 \,\frac{dt}{t}
  \Big)^{1/2}, \quad u\in \ov{\Ran(V)}.
$$
Here $\uSB$ is the bounded analytic semigroup on $\H$
generated by $-\uAB.$

 In Section \ref{sec:examples} we present two applications of
Theorem \ref{thm:Kato}. The first gives an extension of the classical
Meyer inequalities for the Ornstein-Uhlenbeck operator. The second
concerns $L^p$-estimates for the square root of the
$L^p$-realisation of generators of the
second quantisation of analytic $C_0$-contraction semigroups on Hilbert spaces.
We show that for such semigroups the square root property of Theorem \ref{thm:Kato} (3) is
preserved under second quantisation.

The proof of Theorem \ref{thm:Kato} depends crucially on the following gradient
bounds for the semigroup $\PB$ generated by $-\LB$ and the first part of the
Littlewood-Paley-Stein inequalities below.

\begin{theorem}[Gradient bounds]\label{thm:gradient_est}
Assume {\rm(A1)}, {\rm (A2)}, {\rm (A3)}, and let $1<p<\infty$.
\ben
\item[\rm(1)] For all $f \in \F C_{\rm b}(E)$
and $t >0$ we have, for
$\mu$-almost all $x\in
E$,
 \begin{align*}
  \sqrt{t}\| \uD \PB(t) f(x)\| \lesssim (\PB(t)|f|^2(x))^{1/2}.
 \end{align*}
\item[\rm(2)] The set
 $\{ \sqrt{t} \uD \PB(t): \ t\ge 0 \}$
is $R$-bounded in $\L(L^p, \uL^p)$.
\een
\end{theorem}

The notion of $R$-boundedness is a strengthening of the notion
of uniform boundedness and is discussed in Section
\ref{sec:preliminaries1}.

\begin{theorem}[Littlewood-Paley-Stein inequalities]\label{thm:LPS}
Assume {\rm(A1)}, {\rm (A2)}, {\rm (A3)}, and let $1<p<\infty$.
For all $f\in L^p$ we have the square function estimate
$$\|f- P_{\Null_p(L)} f\|_p \lesssim
\Big\n \Big( \int_0^\infty \|\sqrt{t} \uD \PB(t) f \|^2
     \;\frac{dt}{t} \Big)^{1/2}\Big\n_{p}
     \lesssim \n f\n_{p},
$$
where $P_{\Null_p(L)}$ is the projection onto $\Null_p(L)$ along the direct sum decomposition $L^p = \Null_p(L)\oplus \overline{\Ran_p(L)}$.
%
%If the equivalent conditions of Theorem
%\ref{thm:Kato} hold, then we also have
%$$\|f- P_{\Null_p(L)} f\|_p \lesssim \Big\n \Big( \int_0^\infty \|\sqrt{t} \uD %\PB(t) f \|^2 \;\frac{dt}{t} \Big)^{1/2}\Big\n_{p},
%$$
%where $P_{\Null_p(L)}$ is the projection onto $\Null_p(L)$ along the direct sum %decomposition $L^p = \Null_p(L)\oplus \overline{\Ran_p(L)}$.
\end{theorem}

Theorem \ref{thm:Kato} is proved in Sections \ref{sec:Kato} and \ref{sec:Hodge}, and Theorems
\ref{thm:gradient_est} and \ref{thm:LPS}
are proved in Section \ref{sec:gradient}. At this point we emphasise
that in the present nonsymmetric setting, it is not possible
to derive the cases $p>2$ from the cases $1<p\le 2$ by means of
duality arguments (as is done, for example, in \cite{CG01,
Sh92}). New ideas are required; see Section \ref{subsec:p2}.

It follows from \cite[Proposition 2.2]{AKM} that the following
Hodge decompositions hold:
\beq\label{eq:HodgeH}
H = \overline{\Ran(V\s B)} \oplus \Null(V), \quad
  \H = \overline{\Ran(V)}\oplus \Null(V\s B).
\eeq
Here $V\s B$ is interpreted as a closed densely defined
operator from $\H$ to $H$. The second decomposition, however,
shows that the closures of the ranges of $V\s B$ and its
restriction to $\overline{\Ran(V)}$ are the same. Therefore,
in the first decomposition we may just as well interpret $V\s
B$ as an unbounded operator from $\ov{\Ran(V)}$ to $H$. This
observation is relevant for the formulation of the following
Gaussian $L^p$-analogues of the above decompositions, which
are proved in Section \ref{sec:Hodge}.

\begin{theorem}[Hodge decompositions]\label{thm:Hodge-main}
Assume {\rm(A1)}, {\rm(A2)}, {\rm (A3)}, and let $1<p<\infty$.
One has the direct sum decomposition
$$ L^p = \overline{\Ran_p(\uD\s B)} \oplus \Null_p(\uD),$$
where $\uD\s B$ is interpreted a closed densely defined operator from
$\overline{\Ran_p(\uD)}$ to $L^p$.
If the equivalent conditions of Theorem
\ref{thm:Kato} hold, then the above decomposition remains true
when $\uD\s B$ is interpreted as a closed densely defined operator from
$\uL^p$ to $L^p$. In that case one has the direct sum decomposition
$$ \uL^p = \overline{\Ran_p(\uD)}\oplus \Null_p(\uD\s B),$$
where $\uD\s B$ is interpreted as a closed densely defined
operator from $\uL^p$ to $L^p$.
\end{theorem}

In the proofs of these theorems we use the Hodge-Dirac
formalism introduced recently by Axelsson, Keith, and M$^{\rm
c}$Intosh \cite{AKM} in the context of the Kato square root
problem. In the spirit of this formalism, let us define the
Hodge-Dirac operator $\PiB$ associated with $\uD$ and $\uD\s
B$ by the operator matrix
$$\PiB := \bma 0 & \uD\s B \\ \uD & 0\ema.$$
Using Theorems \ref{thm:gradient_est}, \ref{thm:LPS}, and \ref{thm:Hodge-main} we
shall prove:

\bthm[$R$-bisectoriality] \label{thm:R-bisectorial-main}
Assume {\rm(A1)}, {\rm(A2)}, {\rm (A3)}, and let $1<p<\infty$.
The operator $\PiB$ is $R$-bisectorial on $L^p\oplus \overline{\Ran_p(\uD)}$.
If the equivalent conditions of Theorem
\ref{thm:Kato} hold, then $\PiB$ is $R$-bisectorial on $L^p\oplus
\uL^p$.
\ethm

For the definition of $R$-(bi)sectorial operators we refer to
Section \ref{sec:preliminaries2}.
The analogue of Theorem \ref{thm:R-bisectorial-main} for
the more general framework considered in \cite{AKM} generally fails
for $p\not=2$. It this therefore a non-trivial fact that the theorem
does hold in the special case considered here.
Its proof depends on
Theorems \ref{thm:gradient_est}, \ref{thm:LPS}, and
a delicate $L^p$-analysis of the operators
$\uD$ and $\uD\s B$, which is carried out in Section \ref{sec:DandDsB}.

From the fact that on $L^p\oplus \overline{\Ran_p(\uD)}$ one has
$$ \PiB^2 = \bma \LB & 0 \\ 0 & \uLB\ema$$
we deduce that $\PiB$ admits a bounded $H^\infty$-functional
calculus on $L^p\oplus \overline{\Ran_p(\uD)}$ if and only if
$\uLB$ admits a bounded $H^\infty$-functional calculus on
$\overline{\Ran_p(\uD)}$, i.e., if and only if condition (2)
in Theorem \ref{thm:Kato} is satisfied. An alternative proof
of the implication (2)$\Rightarrow$(1) of Theorem
\ref{thm:Kato} can now be derived from the
$H^\infty$-functional calculus of $\PiB$ applied to the
function $\sgn(z) = z/\sqrt{z^2}$; this is done in the final
Section \ref{sec:Hodge}.

\section{Consequences}\label{sec:examples}

Before we start with the proofs of our main results we discuss a number of
situations where operators of the form studied in this paper arise naturally.

\subsection{The Ornstein-Uhlenbeck operator}

Let (A1) be satisfied. Taking $\H=H$ and $V=I$, the derivative
$\uD$ reduces to the Malliavin derivative in the direction of
$H$. Assumption (A2) is then obviously satisfied.
Let  $B$ be an arbitrary operator satisfying (A3). Since
$\uAB = B$ is bounded and sectorial, condition (4) of Theorem
\ref{thm:Kato} is satisfied.

For the special choice $B = \frac12 I$, the resulting operator $L =  \tfrac12\uD\s \uD$
is the classical Ornstein-Uhlenbeck operator of Malliavin
calculus, and the two-sided $L^p$-estimate for $\sqrt{L}$ of
Theorem \ref{thm:Kato} reduces to the celebrated Meyer
inequalities.

\subsection{Linear stochastic evolution equations}

In this subsection we shall describe an application of our results to stochastic evolution equations. 
This application will be worked out in more detail in a forthcoming paper. For unexplained terminology and background material we refer to \cite{DaPZab, vNW05}.

Consider the following linear stochastic evolution equation in a Banach space $E$:
$$
\left\{
\bal
dU(t) & = \mathscr{A}U(t)\,dt + \sigma dW(t), \quad t\ge 0, \\
 U(0) & = x.
\eal
\right.
$$
Here, $\mathscr{A}$ is assumed to generate a $C_0$-semigroup on $E$,
$\sigma$ is a bounded linear operator from a Hilbert space $\mathscr{H}$ to $E$,
and $W$ is an $\mathscr{H}$-cylindrical Brownian motion. 
For later reference we put $\H:=\mathscr{H} \ominus \Null(\sigma)$.

Let us assume now that for each initial value $x\in E$ the above problem admits a unique weak solution $U^x = (U^x(t))_{t\ge 0}$ and that these solutions admit an invariant measure; necessary and sufficient conditions for this to happen can be found in \cite{DaPZab, vNW05, vNW06}. Under this assumption one has weak convergence
$\lim_{t\to\infty} \mu_t = \mu$, where $\mu_t$ is the distribution of the $E$-valued centred Gaussian random variable $U^0(t)$ corresponding to the initial value $x=0$. The limit measure $\mu$ is invariant as well; in a sense that can be made precise it is the minimal invariant measure associated with the above problem. The reproducing kernel Hilbert space associated with $\mu$ is denoted by $H$ and the corresponding inclusion operator $H\embed E$ by $i$.

Define, for bounded continuous functions $f:E\to \R$, 
$$ P(t)f(x) := \E f(U^x(t)), \quad t\ge 0, \ x\in E.$$
The operators $P(t)$ extend in a unique way to a $C_0$-contraction semigroup 
on $L^p(E,\mu)$ for all $1\le p<\infty$. It has been shown in \cite{MN1} that if $P$ is analytic for some (equivalently, for all) $1<p<\infty$, then its
infinitesimal generator $-L$ is of the form considered in Section \ref{sec:elliptic}. More precisely, there exists a unique coercive operator 
$B$ on $\H$ such that $$L = \uD\s B\uD,$$ where 
$V: \Dom(V)\subseteq H\to \H$ is the closed linear operator defined by
$$ V (i\s  x\s):= \sigma\s x\s, \quad x\s\in E\s.$$
It is easy to see that $\uD$ is nothing but the Fr\'echet derivative on $E$ in the direction of $\H$ (where we think of $\H$ as a Hilbert subspace of $E$ under the identification $u \mapsto \sigma u$).

As a consequence of our main results we obtain the following result.

\begin{theorem}
In the above situation, suppose that the transition semigroup $P$ is analytic on $L^p(E,\mu)$ for some (all) $1<p<\infty$. 
\ben
\item[\rm(1)]
The $C_0$-semigroup generated by $A$ leaves $\H$ invariant and restricts to a bounded analytic $C_0$-semigroup on $\H$;  
\item[\rm(2)]
We have $\D(L) = \D(\uD)$ with equivalence of norms
$\n Lf\n_p\eqsim \n \uD f\n_p$ if and only if the negative generator of the restricted semigroup admits a bounded $H^\infty$-functional calculus
on $\H$.
\een
\end{theorem}

\subsection{Second quantised operators} \label{subsec:secQ}

Theorem \ref{thm:Kato} can be applied to the second
quantisation of an arbitrary generator $-A$ of an analytic
$C_0$-contraction semigroup on a Hilbert space $H$. The idea
is to prove that such operators $A$ can be represented as $V\s
BV$ for certain canonical choices of operators $V$ and $B$
satisfying (A2) and (A3). If $E$ and $\mu$ are given such that
(A1) holds, the second observation is that the generator of
the second quantised semigroup on $L^p = L^p(E,\mu)$ equals
the operator $-\uD\s B \uD$ and therefore Theorem
\ref{thm:Kato} can be applied.

We begin with recalling the definition and elementary properties of
second quantised operators. For more systematic discussions we refer
to \cite{Janson,Si}. We work over the real scalar field and
complexify afterwards.

Let ${H}^{(\le 0)} := \R {\bf 1}$
and define ${H}^{(\le n)}$ inductively
as the closed linear span of ${H}^{(\le (n-1))}$
together with all products of the form $\phi_{h_1}\cdot\hdots\cdot \phi_{h_n}$ with
 $h_1,\dots,h_m\in H$. Then we let
 ${H}^{(0)} := \R {\bf 1}$ and define $H^{(n)}$ as the
orthogonal complement of $H^{(\le (n-1))}$ in $H^{(\le n)}$. The space $H^{(n)}$
is usually referred to as the $n$-th {\em Wiener-It\^o chaos}.
We have
the orthogonal {\em Wiener-It\^o decomposition}
$$ L^2 = \bigoplus_{n=0}^\infty H^{(n)}.$$
It is well known that for all $1\leq p\leq q<\infty$  there exist
constants $C_{n,p,q}>0$ such that
 \begin{align*}
 \|F\|_p \leq \|F\|_q \leq C_{n,p,q} \|F\|_p, \quad F \in H^{(n)}.
 \end{align*}
Denoting by $I_n$ the orthogonal projection in $L^2$ onto
$H^{(n)}$, we have the identity
$$
[ I_n(\phi_{h_1}\cdot\ldots\cdot
\phi_{h_n}),I_n(\phi_{h_1'}\cdot\ldots\cdot \phi_{h_n'})]
=\frac{1}{n!}\sum_{\sigma\in S_n}[h_1,h_{\sigma
(1)}']\cdot\ldots\cdot[h_n,h_{\sigma (n)}'],
$$
where $S_n$ is the permutation group on $n$ elements. This
shows that $H^{(n)}$ is canonically isometric to the $n$-fold
symmetric tensor product $H^{\sot n},$ the isometry being
given explicitly by
$$
I_n(\phi_{h_1}\cdot\ldots\cdot
\phi_{h_n}) \mapsto \frac1{\sqrt{n!}} \sum_{\sigma\in S_n}
h_{\sigma(1)}\otimes\dots\otimes h_{\sigma(n)}.
$$
Thus the
Wiener-It\^o decomposition induces a canonical isometry of $L^2$ and
the (symmetric) {\em Fock space}
$$\Ga(H) := \bigoplus_{n=0}^\infty H^{\sot n}.$$

Let $T\in\calL(H)$ be a contraction. We denote by
$\Ga(T) \in \calL(\Ga(H))$ the (symmetric) {\em second quantisation}
of $T$, which is defined on
$H^{\sot n}$ by
$$  \Ga(T) \sum_{\sigma\in S_n}
h_{\sigma(1)}\otimes\dots\otimes h_{\sigma(n)} =
\sum_{\sigma\in S_n} Th_{\sigma(1)}\otimes\dots\otimes
Th_{\sigma(n)}.$$
 By the Wiener-It\^o
isometry, $T$ induces a contraction on $L^2$ and we have
\beq\label{eq:WI} \Ga(T) I_n(\phi_{h_1}\cdot\ldots\cdot
\phi_{h_n})
 =  I_n(\phi_{Th_1}\cdot\ldots\cdot  \phi_{Th_n}).
\eeq
Moreover, $\Gamma(T)$ is a positive operator on $L^2$. We have the identities
 \begin{align} \label{eq:quantidentities}
\Ga(I) = I, \quad \Ga(T_1 T_2) =  \Ga(T_1)\Ga(T_2),
\quad(\Ga(T))\s = \Ga(T\s).
 \end{align}
For all $1\le p \le\infty$, $\Ga(T)$ extends to a positive
contraction on $L^p$ and \eqref{eq:quantidentities} continues
to hold.

For later reference we collect some further properties of second
quantised operators which will not be used in the present section.
The following formula is known as {\em Mehler's formula} \cite{Nua}:
if $f = \varphi(\phi_{h_1},\ldots,\phi_{h_n})$ with $\varphi\in
C_{\rm b}(\R^n)$ and $h_1, \ldots, h_n \in H$, then for $\mu$-almost
all $x\in E$ we have
\begin{align}
\label{eq:Mehler} \Gamma(T) f(x)=
\int_E \varphi( \phi_{Th_1}(x) + \phi_{\sqrt{I-T\s T} h_1}(y),
     \ldots,           \phi_{T h_n}(x) +
          \phi_{\sqrt{I-T\s T}h_n}(y))\,d\mu(y).
\end{align}

For $h\in H$ define
 \begin{align} \label{eq:exponent}
E_h = \sum_{n=0}^\infty \frac1{n!} I_n (\phi_h^n) =
\exp(\phi_h - \tfrac12\|h\|^2).
 \end{align}
This sum converges absolutely in $L^p$ for all $1\le p<\infty$, and
the linear span of the functions $E_h$ is dense in $L^p$
\cite[Chapter 1]{Nua}. From a routine approximation argument
using the closedness of $\uD$ we
obtain that $h\in \Dom(V)$ implies $E_h\in \Dom_p(\uD)$ and
\beq\label{eq:exp-fc1}\uD E_h = E_h \ot Vh.\eeq From \eqref{eq:WI}
and \eqref{eq:exponent} one has the identity
 \begin{align}\label{eq:GammaExp}
\Ga(T)E_h = E_{Th}.
 \end{align}

Let us now turn to the situation where $-A$ be the generator of an
analytic $C_0$-contraction semigroup on a Hilbert space $H$. It is
well known \cite[Theorem 1.57, Theorem 1.58 and the remarks
following these results]{OU05} that $A$ is associated with a
sesquilinear form $a$ on (the complexification of) $H$ which is
densely defined, closed and sectorial, i.e., there exists a constant
$C\geq0$ such that
 \begin{align*}
 |\Im a(h,h)| \leq C \Re a(h,h), \quad h \in \Dom(a).
 \end{align*}
The next result may be known to experts, but as we could not
find an explicit reference we include a proof for the
convenience of the reader.

\bthm\label{thm:repr}
There exists a Hilbert space $\H,$ a closed operator
$V: \Dom(V) \subseteq H\to \H$ with dense domain $\Dom(V) = \Dom(a)$ and dense range,
and a bounded coercive operator $B \in \L(\H)$ such that
$$A = V\s B V.$$
\ethm More precisely, this identity means that we have $a(g,h)
= [ BV g, Vh]$ for all $g,h\in \Dom(V)$; cf. Section
\ref{sec:DBD}. \bpf Writing $a(h) := a(h,h)$ by
\cite[Proposition 1.8]{OU05} we have
$$|a(g,h)|
    \lesssim (\Re a(g))^{1/2} (\Re a(h))^{1/2},\quad g,h \in \Dom(a).$$

We claim that $N := \{h \in \Dom(a) : \Re a(h) = 0\}$ is a
closed subspace of $\Dom(a).$ Indeed, if $h_n \to h$ in
$\Dom(a)$ and $\Re a(h_n) = 0,$ then
 \begin{align*}
 |\Re a(h)| \leq |a(h) - a(h_n)|
   & \leq (\Re a(h))^{1/2}(\Re a(h-h_n))^{1/2}
   \\ & \qquad+ (\Re a(h-h_n))^{1/2}(\Re a(h_n))^{1/2},
 \end{align*}
which becomes arbitrary small as $n\to\infty.$

On the quotient $\Dom(a)/N$ we define a sesquilinear form
$$[Vg,Vh] := \frac12 (a(g,h)
+ \overline{a(h,g)}), \qquad g,h \in \Dom(a),$$ where $V$ denotes
the canonical mapping from $\Dom(a)$ onto $\Dom(a)/N.$ This form is
well defined, since for $n,n' \in N$ we have
 \begin{align*}
 |a(g+n,h+n') - a(g,h)|
 & \leq (\Re a(n))^{1/2}(\Re a(h))^{1/2}
 \\ & \qquad  + (\Re a(g))^{1/2}(\Re a(n'))^{1/2}
 \\ & \qquad  + (\Re a(n))^{1/2}(\Re a(n'))^{1/2}
 \\ & = 0.
 \end{align*}
Since $\Re a(h) = 0$ implies $[h] = [0],$ the form
$[\cdot,\cdot]$ is an inner product on $\Dom(a)/N.$ We
put
 $$\H := \ov{\Dom(a)/N},$$ where the completion is taken with respect
to the norm induced by $[\cdot,\cdot].$
 We interpret $V$ as a linear
operator from $H$ into $\H$ with dense domain
$\Dom(V)=\Dom(a)$ and dense range. To show that $V$ is closed,
we take a sequence $(h_n)_{n\ge 1}$ in $\Dom(a)$ such that
$h_n \to h$ in $H$ and $V h_n \to u$ in $\H.$ Since $\Re a(h_n
- h_m) = \|V(h_n - h_m)\|^2 \to 0$ as $m,n \to \infty,$ the
sequence $(h_n)_{n\ge 1}$ is Cauchy in $\Dom(a).$ Thus the
closedness of $a$ implies that $(h_n)_{n\ge 1}$ has a limit in
$D(a),$ which is $h$ since $h_n \to h$ in $H.$ Consequently,
$\|Vh_n - Vh\|^2 = \Re a(h_n-h) \to 0$. We conclude that $V$
is closed.

Now we define a sesquilinear form $b$ on $\Ran(V)$ by
$$b(Vg, Vh) := a(g,h).$$
This is well defined, since $Vg = V\widetilde g$ and $Vh =
V\widetilde h$ imply that
\begin{align*}
 | a(g,h) -a(\widetilde g,\widetilde h) |
  & \le | a(g-\widetilde g,h)| + |a(\widetilde g,h-\widetilde h)) |
 \\ &\le ( \Re a(g-\widetilde g) \,\Re a(h))^{1/2}
    +(\Re a(\widetilde g)\, \Re a(h-\widetilde h))^{1/2}
 \\ &= \|V(g-\widetilde g)\|\, \|Vh\|
 + \|V\widetilde g\|\, \|V(h-\widetilde h)\| = 0.
\end{align*}
Moreover, the associated operator $B$ extends to a bounded operator
on $\H$, since
$$|b(Vg,Vh)| = |a(g,h)|
    \lesssim (\Re a(g))^{1/2} (\Re a(h))^{1/2} = \|Vg\|\, \|Vh\|.$$
We conclude that $a(g,h) = [BVg,Vh].$
By the identity
$$ \|Vh\|^2 = \Re a(h)  = \Re [BVh,Vh]$$
and the boundedness of $B$ we infer that $\|u\|^2 = \Re[Bu,u]$ for
all $u \in \H,$ and the coercivity of $B$ follows. \epf

Although the triple $(\H,V,B)$ is not unique, the next result implies
that the statements in Theorem \ref{thm:Kato} do not depend on the
choice of $(\H,V,B).$

 \begin{proposition}
Let $-A$ be the generator of an analytic $C_0$-contraction
semigroup on $H$. Let $(\H,V,B)$ and $(\widetilde
\H,\widetilde V, \widetilde B)$ be triples with the properties
as stated in Theorem \ref{thm:repr}. Then:

 \ben
 \item[\rm(i)] The coercivity constants $k$ and $\wt k$
 of $B$ and $\widetilde B$ coincide;
 \item[\rm(ii)]
$\Dom(V) = \Dom(\widetilde V)$ with
$ \|V h\|\eqsim \|\widetilde V h\|.$
\een
If in addition to the above assumptions $(E,H,\mu)$ is an abstract Wiener
space, then for $1\leq p<\infty$ we have
\ben
 \item[\rm(iii)] $\D(\uD) = \D(D_{\widetilde V})$ with
$\|\uD f\|_p \eqsim \|D_{\widetilde V} f\|_p.$ \een \end{proposition}

 \begin{proof}
(i): \  This follows from the identity $ [BVh,Vh] = a(h,h)
  = [\widetilde B \widetilde V h,\widetilde V h]$ for $h \in \Dom(a)$ and the fact that
  $V$ and $\widetilde V$ have dense range.

(ii): \  For $h \in \Dom(A)$ we have
$$
k\|Vh\|^2 \leq \Re[BVh,Vh] = \Re[Ah,h]
  = \Re[\widetilde B \widetilde V h,\widetilde V h]
  \leq \| \widetilde B \| \, \|\widetilde V h\|^2.
$$
Since $\Dom(A)$ is a core for both $\Dom(V)$ and $\Dom(\widetilde
V)$ the result follows.

(iii): \ Let $\DH$ denote the Malliavin derivative, which is
well defined as a densely defined closed  operator from
$L^p(E,\mu)$ to $L^p(E,\mu;H)$, $1\le p<\infty$.  For $f \in
\F C_{\rm b}^1(E;\Dom(V))$ we have, by (ii),
$$\|\uD f\|_p^p = \int_E \| V \DH f \|^p\,d\mu
   \eqsim  \int_E \| \widetilde V \DH f \|^p\,d\mu
       = \| D_{\widetilde V} f\|_p^p.
      $$
The claim follows from this since $\F C_{\rm b}^1(E;\Dom(V))$ is a
core for $\D(\uD)$ and $\D(D_{\widetilde V}).$
 \end{proof}

Let $N:= \{h \in \Dom(a) :\Re a(h)= 0\}$ and let $\Homdom(a):=\H$ be
defined as in the proof of Theorem \ref{thm:repr}, i.e.,
$\Homdom(a)$ is the completion of $\Dom(a)/N$ with respect to the
norm $$\| Vh \|_{\Homdom(a)} := \sqrt{\Re(a(h))},$$ where $V$
denotes the canonical operator from $\Dom(a)$ onto $\Dom(a)/N.$ In
the proof of Theorem \ref{thm:repr} we showed that $V$ is a closed
operator from $H$ into $\Homdom(a)$ with dense domain and dense
range.
 We also constructed a
coercive operator $B \in \L(\Homdom(a))$ such that $a(g,h) =
[BVg,Vh]$ for $g,h\in \Dom(a).$

In Lemma \ref{lem:SVcommute} below we show that the semigroup $S$
generated by $-A$ induces a bounded analytic $C_0$-semigroup $\uS$ on
$\Homdom(a)$ in the sense that $\uS(t) Vh = VS(t)h$ for all $h \in
\Dom(V).$ Its generator will be denoted by $-\uA.$

Now let $(E,H,\mu)$ be an abstract Wiener space and let $D_V := VD$ as before.
For cylindrical functions $f = f(\phi_{h_1},\dots,\phi_{h_n})$ with
$h_j\in \Dom(V) = \Dom(a)$ we have
$$ \uD f  = \sum_{j=1}^n \partial_j f(\phi_{h_1},\dots,\phi_{h_n})
\otimes V h_j.$$
 As will be
shown in Section \ref{sec:DBD}, the realisation on $L^2(E,\mu)$
of the operator $L:=\uD\s B \uD$
extends to a sectorial operator $L$ on $L^p(E,\mu)$ for $1<p<\infty$,
and $-L$ equals the generator of the second quantisation
on $L^p(E,\mu)$ of the semigroup $S$
generated by $-A$ on $H$, i.e.,
$$P(t) = \Gamma(S(t)), \quad t\ge 0.$$
As a consequence, Theorem \ref{thm:Kato} can be translated into the
following result.

\begin{theorem}\label{thm:Kato-sec-quant}
Assume {\rm(A1)} and let $-A$ be the generator of an analytic
$C_0$-contraction semigroup $S$ on $H$. Let $1<p<\infty$ and let $-L$
denote the realisation on $L^p(E,\mu)$ of the generator of the second
quantisation of $S$. The following assertions are equivalent: \ben
\item[\rm(1$'$)] $\D(\sqrt{L}) = \D(\uD)$ with
 $\n \sqrt{L}f\n_{p} \eqsim \n \uD f\n_{p}$ for $f\in \D(\sqrt{L})$;
\item[\rm(3$'$)] $\Dom(\sqrt{A}) = \Dom(a)$ with
  $\n \sqrt{A}h\n \eqsim \sqrt{a(h)}$ for $h\in \Dom(a)$;
\item[\rm(4$'$)] the realisation of $A$ in $\Homdom(a)$
admits a bounded $H^\infty$-functional calculus.
\een
\end{theorem}

The main equivalence here is  (1$'$)$\Leftrightarrow$(3$'$). It
asserts that the square root property with homogeneous norms is preserved when passing
from $H$ to $L^p(E,\mu)$ by means of second quantisation.

The equivalence (3$'$)$\Leftrightarrow$(4$'$) is probably known, although we
could not find a reference for it. The related equivalence
(3$''$)$\Leftrightarrow$(4$''$), with
\ben
\item[\rm(3$''$)] $\Dom(\sqrt{A}) = \Dom(a)$;
\item[\rm(4$''$)] \textit{the realisation of $A$ in $\Dom(a)$
admits a bounded $H^\infty$-functional calculus,}
\een
 is stated explicitly in \cite[Theorem 5.5.2]{Arendt04}.

We have already mentioned that Theorem \ref{thm:Kato} is obtained by
combining two one-sided versions of it involving  Riesz transforms.
The same is true for Theorem \ref{thm:Kato-sec-quant}.

\section{The operator $\LB$}
\label{sec:DBD}

In this section we give a rigorous definition of the operator $\LB$ as
a closed and densely defined operator acting in $L^p:=
L^p(E,\mu),$ where $1<p<\infty$.

We begin with an analysis in
the space $L^2 := L^2(E,\mu).$
Associated with the (complexified) operators $B:\H\to \H$,
$V:\Dom(V)\subseteq H\to \H$, and $\uD:\Dom(\uD)\subseteq L^2\to \uL^2$
are the sesquilinear forms $a$ on $H$ and $l$ on $L^2$ defined by
$\Dom(a):= \Dom(V)$ and
$$a(h_1,h_2) := [BVh_1,Vh_2],
$$
and $\Dom(l):= \Dom(\uD)$ and
$$ l(f_1,f_2) := [B\uD f_1,\uD  f_2],
$$
where in the second line we identify $B$ with the operator
$I\otimes B$. Here, and in what follows, we write
$$ \Dom(\uD) := \Dom_2(\uD), \quad \Ran(\uD) := \Ran_2(\uD).$$

The forms $a$ and $l$ are easily seen to be closed, densely defined
and sectorial in $H$ and $L^2$, respectively. The operators
associated with these forms are denoted by $\AB$ and $\LB$,
respectively; their domains will be denoted by $\Dom(\AB)$ and
$\Dom(\LB)$. We may write
$$
\bal \AB & = V\s BV, &&
\LB = \uD \s B \uD; \eal
$$
this notation is justified by the
observation that
$$ \bal
h\in \Dom(\AB) &\Longleftrightarrow [h\in \Dom(V) \ \hbox{and} \ BV h \in \Dom(V\s)],
\\
f\in \Dom(\LB) &\Longleftrightarrow [f\in \Dom(\uD) \ \hbox{and} \
B\uD f \in \Dom(\uD\s)], \eal
$$
in which case we have $$ \AB h = V\s(BVh), \quad \LB f = \uD\s(B\uD)
f.$$ Let us also note (cf. \cite[Lemma 1.25]{OU05}) that $\Dom(\AB)$
and $\Dom(\LB)$ are cores for $\Dom(V)$ and $\Dom(\uD)$,
respectively.

For later use we observe that if $B$ satisfies (A3), then also $B\s$ satisfies
(A3) and we have
$$ \bal \AB\s & = V\s B\s V, &&
\LB\s = \uD \s B\s \uD \eal
$$
with similar justifications.

The proof of the next lemma can be found in \cite{GGvN03}.
Recall that $\phi:H\to L^2$ is the isometric embedding defined in Section
\ref{sec:elliptic}.

\blem \label{lem:Dster} Let $1<p<\infty.$ For all $f\in \F
C_{\rm b}^1(E;\Dom(V))$ and $u \in \Dom(V^*)$ we have
 $$f \ot u \in \Dom_p(\uD ^*) \textrm{ and }
 \uD ^*(f \ot u) = f \phi_{V^* u} - [\uD f, u].$$
\elem

\blem\label{lem:part} Identifying $H$ with its image $\phi(H)$
in $L^2$, $\AB$ is the part of $\LB$ in $H$. \elem \bpf
Suppose first that $h \in \Dom(\AB)$. Then, by the form
definition of $\AB$, we have $h\in\Dom(V)$ and  $BVh\in
\Dom(V\s)$. Hence Lemma \ref{lem:Dster} gives us $\one\otimes
BVh\in \Dom(\uD\s)$ and
$$\uD\s(\one\otimes  BVh) = \phi_{V\s BVh} = \phi_{\AB h}.$$
From this, combined with the identity $$\uD  \phi_{h} =
\one\otimes Vh$$ which follows from the definition of
$\uD,$ we deduce that
$$ [B\uD \phi_{h}, \uD f]
= [\one\otimes BVh, \uD f] = [\phi_{\AB h}, f] $$
for all $f\in \Dom(\uD)$.
Therefore, $\phi_{h}\in\Dom(\LB)$ and $\LB  \phi_{h}  = \phi_{\AB
h}.$ Denoting the part of $\LB$ in $H$ by $\LB^H$ for the moment,
this argument shows that $\AB\subseteq \LB^H$.

On the other hand, if $\phi_h\in \Dom(\LB^H)$, then $\phi_h\in
\Dom(\LB)$ and $\LB \phi_h = \phi_{h'}$
for some $h'\in H$. Hence for all $g\in
\Dom(V)$ we obtain
$$
 [BV h, V g]
= [B \uD \phi_{h}, \uD \phi_{g}]
= l(\phi_{h}, \phi_{g})
= [\LB \phi_{h}, \phi_g]= [\phi_{h'}, \phi_g]
= [h',g].
$$
It follows that $h\in \Dom(\AB)$ and $[\AB h, g] = [{h'},g]$.
This shows that $\AB h = h'$, and we have proved the opposite
inclusion $\AB\supseteq \LB^H$.
\end{proof}

It follows from the theory of forms (cf. \cite[Proposition
1.51]{OU05}) that $-\AB$ and $-\LB$ generate analytic
$C_0$-contraction semigroups $\SB = (\SB(t))_{t\ge 0}$ and $\PB =
(\PB(t))_{t\ge 0}$ on $H$ and $L^2$, respectively. In fact we have
the following more precise result. Recall that the constant $k>0$
has been introduced in Assumption (A3).

\begin{proposition}\label{prop:sector}
The operators $-\AB$ and $-\LB$ generate analytic $C_0$-contraction
semigroups on $H$ and $L^2$ of angle $\arctan\ga$, where $\ga =
\tfrac{1}{2k} \|B-B^*\|$.
\end{proposition}
\begin{proof}
We prove this for $\LB$; the proof of $\AB$ is similar
(alternatively, the result for $\AB$ follows from the result
for $\LB$ via Lemma \ref{lem:part}).

By the Lumer-Phillips theorem it suffices to show that $\LB$ has
numerical range in the sector of angle $\arctan\ga$. Using
that $B$ is a (complexified) real operator, for $f\in
\Dom(\LB)$ we have, with $F = \Re \uD f$ and $G = \Im \uD f,$
\beq\label{eq:coercive}
\bal \big|\Im [\LB f,f]\big|
  & =    \big|[ (B-B^*) F,G ]\big|
 \\ & \leq \frac12\|B-B^*\| \;
 \big(\|F\|^2 + \|G\|^2\big)
 \\ & \leq \frac{1}{2k}  \|B-B^*\|
     \big([B F, F]
        + [B G,  G] \big)
 \\ & =   \frac{1}{2k} \|B-B^*\|  \,\Re [\LB f,f].
\eal
\eeq
\end{proof}

The first main result of this section identifies $\PB$ as the second
quantisation of $\SB$.

\begin{theorem}\label{thm:2nd-q}
For all $t\ge 0$ we have $\PB(t) = \Ga(\SB(t))$.
\end{theorem}
\begin{proof}
We recall from Lemma \ref{lem:part} that $\PB(t)\phi_h =
\SB(t)h$ for all $h \in H$.

First we check that for all $h\in \Dom(\AB)$, the functions
$E_h \in L^2$ are in the domains of $\LB$ and $\widetilde
\LB$, where $-\widetilde \LB$ is the generator of $\Ga(\SB)$,
and that both generators agree on those functions. Using
\eqref{eq:exp-fc1} and Lemma \ref{lem:Dster} we obtain
 \begin{align*}
\LB E_h
 & = \uD ^* B \uD  E_h
\\ &  = \uD ^* ( E_h \ot BVh )
\\ &  = E_h \phi_{V^*BVh} - [BV h,Vh]E_h
\\ &  = (\phi_{\AB h} - [\AB h,h] ) E_h,
 \end{align*}
while on the other hand, using \eqref{eq:exponent} and
\eqref{eq:GammaExp} combined with a simple approximation argument, we have
 \begin{align*}
\widetilde \LB E_h
 & = \lim_{t \downarrow 0} \tfrac1t (E_{S(t)h} - E_h)
 \\ & = E_h \,\frac{d}{dt}\Big|_{t=0} \Big(\phi_{\SB(t)h}- \frac12\|\SB(t)h\|^2\Big)
 \\ & = (\phi_{\AB h} - [\AB h,h]) E_h.
 \end{align*}
 The set $\lin \{E_h : h \in \Dom(\AB)\}$ is dense in $L^2$
 and invariant under the semigroup
$\Ga(\SB).$ As a consequence, this set is a core for
$\Dom(\widetilde \LB)$. It follows that $\Dom(\widetilde
\LB)\subseteq \Dom(\LB)$. Since both $-\widetilde \LB$ and
$-\LB$ are generators this implies
 $\Dom(\widetilde \LB)= \Dom(\LB)$ and therefore  $\widetilde \LB = \LB$.
\end{proof}

So far we have considered $\PB$ as a $C_0$-semigroup in $L^2$. Having
identified $\PB$ as a second quantised semigroup on $L^2$, we are in a
position to prove that $\PB$ extends to the spaces $L^p$.

 \begin{theorem}\label{thm:sector}\label{thm:D(L)dense} For $1 \leq p <\infty$,
the semigroup $\PB$ extends to a $C_0$-semigroup of positive
contractions on $L^p$ satisfying
$\|P(t)f\|_\infty \leq \|f\|_\infty$ for $f\in L^\infty.$ The measure
$\mu$ is an invariant measure for $\PB,$ i.e.,
 $$\int_E \PB(t) f \,d\mu = \int_E f \,d\mu,
   \qquad f \in L^p,\ t \geq 0.$$
For $1<p<\infty$, $\PB$ is an analytic $C_0$-contraction semigroup on
$L^p$.
 \end{theorem}

 \begin{proof}
The extendability to a $C_0$-contraction semigroup on $L^p$, as well as
the $L^\infty$-contractivity and positivity
follow from general results on second quantisation. The invariance of
$\mu$ follows from
$$\int_E \PB(t)f \,d\mu
  = \int_E f \, \PB^*(t)\one\,d\mu
  = \int_E f \,d\mu, \quad f \in L^p.
$$
Here we use that $\PB^*$ is a second quantised semigroup as
well, and therefore satisfies $\PB\s(t)\one=\one$ for all
$t\ge 0$.

It remains to prove the last statement. We have seen in
Proposition \ref{prop:sector} that $\PB$ extends to an
analytic $C_0$-contraction semigroup on $L^2$. The
extension to an analytic $C_0$-contraction semigroup on
$L^p$, $1<p<\infty$, follows from a standard argument
involving the Stein interpolation theorem and duality.
\end{proof}

 \begin{remark}
We mention that a different argument to establish analyticity
in $L^p$ has been given in the case of Ornstein-Uhlenbeck
semigroups in \cite{CFMP05, MN1}. This argument also works in
the more general setting considered here and yields an angle
of analyticity which is better than the one obtained by Stein
interpolation. For Ornstein-Uhlenbeck semigroups (for which we
have $B+B^*=I,$ see \cite{MN1}),  this angle is optimal.
 \end{remark}

\begin{definition} On $L^p$ we define the operator $\LB$
as the negative generator of the semigroup $\PB$.
\end{definition}

 \blem \label{lem:Fcore} For all $1<p<\infty$,
$\F C_{\rm b}^\infty(E;\Dom(\AB))$ is a $\PB$-invariant core for
$\D(\LB)$. Moreover, for $f,g \in \F C_{\rm b}^\infty(E;\Dom(\AB)) $ and $\psi\in C_{\rm
b}^\infty(\R)$ we have
\ben
\item[\rm(1)] {\rm (Product rule)} $fg\in \F C_{\rm b}^\infty(E;\Dom(\AB))$ and
$$ \LB(fg) = f \LB g + g \LB f - [(B+B^*)\uD f, \uD g];$$
\item[\rm(2)] {\rm (Chain rule)} $\psi\circ
f\in\F C_{\rm b}^\infty(E;\Dom(\AB))$ and
$$\LB (\psi\circ f) = (\psi' \circ f) \LB f
 - (\psi''\circ f) [ B \uD f, \uD f].$$
\een

 \elem

 \begin{proof}
First we show that $\F C_{\rm b}^\infty(E;\Dom(\AB))$ is
contained in $\D(\LB)$; we thank Vladimir Bogachev for
pointing out an argument which simplifies our original proof.
Pick a function $f\in \F C_{\rm b}^\infty(E;\Dom(\AB))$ and
notice that $f \in \Dom(\LB) \cap L^p$. The space $L^p$ being
reflexive, by a standard result from semigroup theory (cf.
\cite {ButBer}) it suffices to show that
$$ \limsup_{t\da 0} \frac1t \n P(t)f-f\n_p < \infty.$$
Using that $\LB = \uD\s B \uD$ in $L^2$, an explicit calculation using Lemma
\ref{lem:Dster} shows that $\LB f \in L^2\cap L^p$.
Moreover, in $L^2$ we have the identity
$$ \frac1t (P(t)f-f) = \frac1t \int_0^t P(s)Lf\,ds.$$
Since $Lf\in L^p$, the right-hand side can be interpreted as a Bochner integral
in $L^p$, which for $0<t\le 1$ can be estimated in $L^p$ by
$$\Big\n \frac1t \int_0^t P(s)Lf\,ds\Big\n_p \le \n Lf\n_p.$$
This gives the desired bound for the limes superior.

To show that $\F C_{\rm b}^\infty(E;\Dom(\AB))$ is invariant under $\PB,$ we take $f$ of the form
$$f = \varphi(\phi_{h_1}, \ldots, \phi_{h_n}),$$
with $\varphi \in C_{\rm b}^\infty(\R^n)$ and $h_1, \ldots, h_n \in
\Dom(\AB).$ Let $R(t) := \sqrt{I - \SB^*(t)\SB(t)}.$  By Mehler's
formula, for $\mu$-almost all $x\in E$ we have
\beq\label{eq:PBcylindrical}
\bal
 \PB(t) f(x) &= \int_E \varphi( \phi_{\SB(t)h_1}(x) + \phi_{R(t)h_1}(y),
     \ldots  \\ & \qquad\qquad \ldots,        \phi_{\SB(t)h_n}(x) +
          \phi_{R(t)h_n}(y))\,d\mu(y)
      \\&    = \psi_t(\phi_{\SB(t)h_1}(x), \ldots, \phi_{\SB(t)h_n}(x)),
 \eal
\eeq
where $$\psi_t(\xi_1,\ldots,  \xi_n)
   = \int_E \varphi( \xi_1 + \phi_{R(t)h_1}(y),
     \ldots,           \xi_n +
          \phi_{R(t)h_n}(y))\,d\mu(y).$$
Since $\psi_t \in C_{\rm b}^\infty(\R^n)$ and $\SB(t)h_j \in \Dom(\AB)$ for
$j=1,\ldots, n,$ it follows that the subspace $\F C_{\rm b}^\infty(E;\Dom(\AB))$ is invariant under $\PB.$ Since
it is dense in $L^p$ and contained in $\D(\LB)$, it is a core for $\D(\LB)$.

The identities (1) and (2) follow by direct computation, using
the identity $\LB = \uD^* B \uD$ and Lemma \ref{lem:Dster}.
 \end{proof}

\begin{remark}The same proof shows that $\F C_{\rm
b}^\infty(E;\Dom(\AB^k))$ is a $\PB$-invariant core for
$\D(\LB)$ for every $k \geq 1.$ \end{remark}

\section{The operator $\uLB$}
\label{sec:DDB}

Our next aim is to give a rigorous description of the operator
$\uLB$ on the spaces $\overline{\Ran_p(\uD)},$ $1 < p
<\infty,$ where the closure is taken in $\uL^p:=
L^p(E,\mu;\H)$.

On $\H$ and $\uL^2$ we consider the sesquilinear forms $ \ua :
\Dom(V^*) \times \Dom(V^*) \to \C$,
$$  \ua(u_1,u_2) := [V^*u_1,V^*u_2]
$$
and $\ul : \Dom(\uD ^*)\times \Dom(\uD ^*) \to \C$,
$$\ul(F_1,F_2) : = [\uD ^*F_1,\uD ^*F_2].
$$
Here, $\uD\s: \Dom(\uD\s)\subseteq \uL^2\to L^2$ is the adjoint of the
operator $\uD:\Dom(\uD)\subseteq L^2\to \uL^2$. The forms $a$ and $l$
are closed, densely defined and sectorial. The operators associated
with these forms are denoted by $\uA_I$ and $\uL_I$ respectively,
with domains $\Dom(\uA_I)$ and $\Dom(\uL_I)$. We may write
$$
\uA_I = VV\s, \qquad  \uL_I = \uD \uD \s
$$
with similar justifications as before.
These operators are self-adjoint; see e.g. \cite[Proposition
1.31]{OU05}.
We introduce next the operators
$$\bal \Dom(\uAB) & := \{h\in \H: \ Bh\in\Dom(\uA_I)\}, \quad &&\uAB := \uA_I B;\\
       \Dom(\uLB) & := \{F\in \uL^2: \ BF\in \Dom(\uL_I)\}, \quad &&\uLB :=  \uL_I B.
\eal
$$
Note that
$$
\uAB  = VV\s B, \qquad \uLB = \uD  \uD \s B.
$$
It follows from standard operator theory \cite[Lemma 4.1]{AKM} that
$\uAB$ and $\uLB$ are closed and densely defined and satisfy
$$ \uAB = (B\s \uA_I)\s, \qquad \uLB = (B\s \uL_I)\s.$$

\bpr\label{prop:sectL} The operators $\uAB$ and $\uLB$ are
sectorial on $\H$ and $\uL^2$ of angle $\arctan\ga,$ where $ \ga :=
\tfrac{1}{2k} \n B-B\s\n.$ For all
$u\in \Dom(\uAB)$ we have $\one\ot u\in \Dom(\uLB)$ and
$$\uLB(\one\ot u) = \one\ot \uAB u.$$
\epr
 \begin{proof}
Writing $v := \Re u$ and $w:= \Im u$, by estimating as in \eqref{eq:coercive}
we obtain
$$
|\Im [B u,u]|
\le \frac{1}{2k}\|B-B^*\|  \Re [B  u, u].
$$
This shows that the numerical range of $B$ is contained in the closed
sector around $\R_+$ of angle $\arctan \ga$. The same is true
for the operator $B$ as an operator acting on $\uL^2$. Hence it follows
from \cite[Proposition 7.1]{AMN} (in which `positive' may be weakened
to `non-negative') that the operators $\uAB = \uA_I B$ and $\uLB =
\uL_I B$ are sectorial of angle $\arctan \ga$.
The final identity follows from
 \begin{align*}
  \uLB (\one \ot u)
   = \uD\uD^* (\one \ot Bu)
   = \uD  (\phi_{V^*Bu})
   = \one \ot VV^*Bu = \one\ot \uAB u.
 \end{align*}
\end{proof}

As a consequence, $-\uAB$ and $-\uLB$ generate bounded
analytic $C_0$-semigroups of angle $\rm{arccot}\,\ga$ on $\H$
and $\uL^2.$ In what follows we denote these semigroups by
$\uSB$ and $\uPB$.

\blem\label{lem:AVcommute} If $h \in \Dom(\AB)$ and $\AB h \in \Dom(V),$
then $V h \in \Dom(\uAB)$ and $$\uAB V h = V \AB h.$$
 \elem
 \begin{proof}
Since $h\in \Dom(\AB)$, the definition of $\AB$ as the operator
associated with the form $(h,g)\mapsto [BV h,V g]$ implies that $h \in
\Dom(V )$, $BV h \in \Dom(V ^*)$, and $\AB h = V ^*(BV h).$

To check that we have $V h \in \Dom(\uAB)$, in view of the
identity $\uAB = (B^*\uA_I)^*$ we must find $h' \in \H$ such
that $[B^*\uA_I g,V h] = [g,h']$ for all $g\in \Dom(\uA_I).$
But $h':= V \AB h$ does the job, since $[g,V \AB h] = [g, V V
^*BV h] = [B^*\uA_I g,V h];$ this implies that $V h \in
\Dom(\uAB)$ and $\uAB V h = V \AB h.$
\end{proof}

\blem\label{lem:SVcommute}
 For all $h\in \Dom(V )$ and $t\ge 0$
we have $\SB(t)h\in \Dom(V )$
and $$V  \SB(t)h = \uSB(t)V h.$$
\elem
\bpf
We may assume that $t>0$.

First let $g\in \Dom(\AB^2)$. Then $\AB g\in \Dom(\AB)\subseteq\Dom(V
)$, and therefore Lemma \ref{lem:AVcommute} implies that $V g\in
\Dom(\uAB)$ and $ \uAB V g = V \AB g.$ For $\l>0$ it follows that
$(I+ \l \uAB)V g = V (I + \l \AB)g$. Applying this to $g= (I + \l
\AB)^{-1}h$ with $h\in\Dom(\AB)$ we obtain
$$ V (I + \l \AB)^{-1}h = (I + \l \uAB)^{-1}V h.$$
Taking  $\l = \frac{t}{n}$ and repeating this argument $n$ times  we
obtain, for all $h\in \Dom(\AB)$,
$$ V (I + \tfrac{t}{n} \AB)^{-n} h =
     (I + \tfrac{t}{n}\uAB)^{-n} V h.$$
Taking limits $n\to\infty$ and using  the closedness of $V $, we
obtain $\SB(t)h\in D(V )$ and
$$ V \SB(t)h = \uSB(t)V h.$$
We are still assuming that $h\in \Dom(\AB)$. However, this assumption
may now be removed by recalling the fact that $\Dom(\AB)$ is a core
for $\Dom(V)$.
\end{proof}

\begin{lemma}\label{lem:uABinjective}
For all $t\ge 0$ we have $\uSB(t)\ov{\Ran(V)} \subseteq \ov{\Ran(V)}$.
Moreover, the part of
 $\uAB$ in $\ov{\Ran(V)}$ is injective.
\end{lemma}

 \bpf
The first assertion follows from Lemma \ref{lem:SVcommute}.
Suppose that $\uAB u = VV^*B u = 0$ for some $u$ belonging to the domain
of the part of $\uAB$ in
$\ov{\Ran(V)}$. Then $\|V^*B u\|^2 = 0,$ so $Bu\in
\Null(V^*)$. Thus $[Bu, Vh]=0$ for all $h\in \Dom(V)$.
Since $u \in \ov{\Ran(V)}$ it follows that $[Bu,u] = 0$, and therefore
$u=0$ by the coercivity of $B$ on $\ov{\Ran(V)}.$
 \epf

Next we show that the semigroups $\uPB$ and $\PB\otimes \uSB$ agree
on $\ov{\Ran(\uD)}$. We need two lemmas which are formulated, for
later reference, for the $L^p$-setting.

\begin{lemma}\label{lem:polyn-core}
For $1<p<\infty$, $\F C_{\rm b}^\infty(E;\Dom(\AB))$
 is a core for $\D(\uD)$.
\end{lemma}
\bpf First let $f= \varphi(\phi_{h_1},\dots, \phi_{h_n})$ with
$\varphi\in C_{\rm b}^1(\R^n)$ and $h_1, \dots, h_n\in
\Dom(V)$. Choose sequences $(h_{jk})_{k\ge 1}$ in $\Dom(\AB)$
with $h_{jk}\to h_j$ in $\Dom(V)$ as $k\to\infty$. Then $f_k
\to f$ in $L^p$ and $\uD f_k \to \uD f$ in $\uL^p$, where $f_k
= \varphi(\phi_{h_{1k}},\dots.\phi_{h_{nk}})$. Since $\F
C_{\rm b}^1(E;\Dom(V))$ is a core for $\D(\uD)$, this proves
that $\F C_{\rm b}^1(E;\Dom(\AB))$ is a core for $\D(\uD)$.
Now a standard mollifier argument, convolving $\varphi$ with a
smooth function of compact support, shows that $\F C_{\rm
b}^\infty(E;\Dom(\AB))$ is a core for $\D(\uD)$. \epf

The next result is well known in the context of Ornstein-Uhlenbeck semigroups;
see, e.g., \cite[Lemma 2.7]{CG01}, \cite[Proposition 3.5]{MN2}.

\bthm \label{thm:uLsemigroup} For all $1<p<\infty$,
the semigroup $\PB\otimes \uSB$
restricts to a bounded analytic $C_0$-semigroup on $\overline{\Ran_p(\uD)}.$
For $f\in \D(\uD)$ and $t\ge 0$ we have $\PB(t)f \in
\D(\uD)$ and
$$ \uD \PB(t)f  = (\PB(t)\otimes \uSB(t))\uD f.$$
\ethm

\bpf
First we show that for all $f\in \D(\uD)$
we have $\PB(t)f \in
\D(\uD)$ and $ \uD \PB(t)f  = (\PB(t)\otimes \uSB(t))\uD f.$
Since $\uD$ is closed and $\F C_{\rm b}^\infty(E;\Dom(\AB))$ is a core for $\D(\uD)$ by Lemma \ref{lem:polyn-core}, it
suffices to check this for functions $f\in \F C_{\rm
b}^\infty(E;\Dom(\AB)).$

We use the notations of Lemma \ref{lem:Fcore}. By
\eqref{eq:PBcylindrical} and Lemma \ref{lem:SVcommute}, for
functions $f= \varphi(\phi_{h_1},\dots, \phi_{h_n})$  we have,
for $\mu$-almost all $x\in E$,
 \begin{align*}
 \uD \PB(t) f(x)
  &= \sum_{j=1}^n \partial_j
   \psi_t(\phi_{\SB(t)h_1}(x), \ldots, \phi_{\SB(t)h_n}(x))
     \ot V \SB(t)h_j
 \\&= \sum_{j=1}^n \int_E \partial_j
   \varphi(\phi_{\SB(t)h_1}(x) + \phi_{R(t)h_1}(y),\ \ldots
\\& \qquad\qquad\qquad     \ldots \ , \phi_{\SB(t)h_n}(x)
                             + \phi_{R(t)h_n}(y))\, d\mu(y)
           \ot \uSB(t) V h_j
  \\&= (\PB(t) \ot \uSB(t) )  \uD f (x).
 \end{align*}
This identity shows that $\PB(t)\ot \uSB$ maps
$\Ran_p(\uD)$ into itself,
and therefore $\PB\ot\uSB$ restricts to a bounded
$C_0$-semigroup on $\overline{\Ran_p(\uD)}$.
The invariance of $\overline{\Ran_p(\uD)}$ under the operators
$\PB(z)\otimes \uSB(z)$, where $z\in\C$ is in the sector of bounded
analyticity of $\PB$, follows by uniqueness of analytic continuation (consider
the quotient mapping from $\uL^p$ to $\uL^p/\overline{\Ran_p(\uD)}$).
\epf

In the next result we
return to the $L^2$-setting and show that the semigroups
$\PB\ot \uSB$ and $\uPB$ on $\uL^2$ agree on
$\overline{\Ran(\uD)}$.

\bthm \label{thm:DDsBsemigroup} Both $\uPB$ and $\PB\otimes \uSB$ restrict
to bounded analytic $C_0$-semigroups on $\ov{\Ran(\uD)}$, and their
restrictions coincide:
$$\uPB(t)F =\PB(t)\otimes
  \uSB(t)F, \quad F\in \ov{\Ran(\uD)}.$$
\ethm
 \begin{proof}
The invariance of $\ov{\Ran(\uD)}$ under $\PB\otimes \uSB$
follows from the previous theorem. Let us write $-N$ for the
generator of $\PB\otimes \uSB$ on $\ov{\Ran(\uD)}$. From
$V(\Dom(\AB^2))\subseteq \Dom(\uAB)$ (cf. the proof of Lemma
\ref{lem:SVcommute}) and $\F C_{\rm
b}^\infty(E;\Dom(\AB^2))\ot \Dom(\uAB) \subseteq \Dom(\LB) \ot
\Dom(\uAB)$ we see that the subspace $U := \{ \uD f : f \in \F
C_{\rm b}^\infty(E;\Dom(\AB^2)) \}$ is contained in $\Dom(N)$.
This subspace is dense in $\ov{\Ran(\uD)}$ since $\F C_{\rm
b}^\infty(E;\Dom(\AB^2))$ is a core for $\Dom(\LB)$ (by Lemma
\ref{lem:Fcore} and the remark following it) and $\Dom(\LB)$
is a core for $\Dom(\uD).$  Since $(\PB\otimes \uSB)U
\subseteq U$ by Theorem \ref{thm:uLsemigroup}, it follows that
$U$ is a core for $\Dom(N).$

For functions $f \in \F C_{\rm
b}^\infty(E;\Dom(\AB^2))$ we obtain
$$N \uD f = \uD \LB f = \uLB \uD f.$$
The first identity follows from
Theorem \ref{thm:uLsemigroup} and the second from a direct computation.
Alternatively, the second identity
can be deduced from the analogue of Lemma \ref{lem:AVcommute} for
$\uD$ and $\LB$.

Thus $N = \uLB$ on the core $U$ of $\Dom(N)$. It follows that
$\Dom(N)\subseteq \Dom(\uLB)$ and $N = \uLB$ on $\Dom(N)$. Let
$\l>0.$ Multiplying the identity $\l+N = \l+\uLB$ from the
right with $(\l+N)^{-1}$ and from the left with
$(\l+\uLB)^{-1},$ we obtain $(\l+N)^{-1}= (\l+\uLB)^{-1}$ on
$\ov{\Ran(\uD)}$. In particular, $ (\l+\uLB)^{-1}$ maps
$\ov{\Ran(\uD)}$ into itself. As in Lemma \ref{lem:SVcommute}
it follows that $\uPB$ leaves $\ov{\Ran(\uD)}$ invariant and
that the restriction of $\uPB$ to $\ov{\Ran(\uD)}$ equals the
semigroup generated by $-N$, which is $\PB\otimes
\uSB|_{\ov{\Ran(\uD)}}$.
\end{proof}

\begin{definition}\label{def:uPB} Let $1<p<\infty$. On $\overline{\Ran_p(\uD)}$
we define  $ \uPB := \PB\ot \uSB|_{\overline{\Ran_p(\uD)}}.$
The negative generator of $\uPB$ is denoted by $\uLB$.
\end{definition}

By Theorem \ref{thm:DDsBsemigroup}, for $p=2$ this definition
is consistent with the one given at the beginning of this section.

\section{Intermezzo I: $R$-boundedness and radonifying operators}
\label{sec:preliminaries1}

Before proceeding with the proofs of the main results we
insert a section containing a concise discussion of the
notions of $R$-boundedness, radonifying operators, and square
functions. For more information and further results we refer
to the excellent sources \cite{DHP, KuW} as well as the papers
\cite{vNW05,NVW} and the references given therein. The
notations in this section will be independent of those in the
previous ones.

\subsection{$R$-boundedness}

Throughout this section, unless otherwise stated $(M,\mu)$ is an
arbitrary $\sigma$-finite measure space and $\H$ is an arbitrary
Hilbert space. In analogy to previous notations we write $L^p:=
L^p(M,\mu)$ and $\uL^p := L^p(M,\mu;\H)$.

Let $X$ and $Y$ be Banach spaces and let $(r_j)_{j\ge 1}$ be a
sequence of independent {\em Rademacher variables}, i.e., $\P(r_j=
1) = \P(r_j=-1) = \tfrac12$ for each $j$.

A collection of bounded linear operators $\mathcal{T} \subseteq
\L(X,Y)$ is said to be {\em $R$-bounded} if there exists $C\ge 0$ such that
for all $k=1,2,\dots$ and  all choices of $x_1, \ldots, x_k \in X$ and $T_1,
\ldots, T_k  \in \mathcal{T}$
 we have $$\E \Big\|\sum_{j=1}^k r_j T_j x_j\Big\|^2
   \leq C^2 \E \Big\|\sum_{j=1}^k r_j  x_j\Big\|^2. $$
The smallest constant $C$ for which this
inequality holds is denoted by $R(\mathcal{T}).$ By the
Kahane-Khintchine inequalities one may replace the exponents 2 by
arbitrary $p\in [1,\infty);$ this only changes the value of the
constant $C.$ Every bounded subset of operators on a Hilbert space
is $R$-bounded. If $\mathcal{T}$ is $R$-bounded, then the closure
with respect to the strong operator topology of the absolutely
convex hull of $\mathcal{T}$ is $R$-bounded as well, with constant
at most $C$ (in the real case) or $2C$ (in the complex case). A
useful consequence of this is the following result \cite[Corollary
2.14]{KuW} which we formulate for real spaces $X$ and $Y$ (in the
complex case an extra constant $2$ appears).

 \begin{proposition} \label{prop:Rboundedintegral}
Let  $\mathcal{T}\subseteq \L(X,Y)$ be $R$-bounded, and let $f:M\to
\L(X,Y)$ be a function with values in $\mathcal{T}$ such that
$\xi\mapsto f(\xi)x$ is strongly  $\mu$-measurable for all $x\in E$.
For $\phi\in L^1$ define
$$T_{\phi,f}x := \int_M \phi(t) f(t)x \,d\mu(t), \qquad x\in X.$$
Then the collection $\{ T_{\phi,f} : \ \|\phi\|_{L^1} \leq
1\}$ is $R$-bounded in $\L(X,Y)$.
 \end{proposition}

The next result may be known to specialists, but since we
couldn't find a reference for it we include a proof.

 \bpr\label{prop:Rtensor}
Let $1 \leq p< \infty.$ If $\mathcal{T} \subseteq \L(L^p)$ is
$R$-bounded and $\mathcal{S} \subseteq \L(\H)$ is bounded, then
 $\mathcal{T} \ot \mathcal{S} \subseteq \L(\uL^p)$ is $R$-bounded.
 \epr

 \bpf
Since $T \ot S = (T\ot I)(I \ot S)$ and $I \ot \mathcal{S}$ is $R$-bounded by the Fubini theorem, it suffices to show that $\mathcal{T} \ot I$ is $R$-bounded.

Let $(h_i)_{i=1}^n$ be an orthonormal system in $\H$ and let
$F_1,\dots, F_k$ be functions in $\uL^p$ of the form $F_j :=
\sum_{i=1}^n f_{ij} \ot h_i$. Let $(r_i)_{i\ge 1}$ and $(\widetilde r_i)_{i\ge
1}$ be
independent Rademacher sequences. Then, putting $g_i :=\sum_{j=1}^k
r_j T_j f_{ij}$,
 \begin{align*}
 \E\Big\| \sum_{j=1}^k r_j (T_j \ot I) F_j \Big\|_p^p
     & =  \E \Big\| \sum_{j=1}^k r_j \sum_{i=1}^n T_j f_{ij} \ot h_i   \Big\|_p^p
  \\ & =  \E \int_M  \Big\|\sum_{j=1}^k r_j \sum_{i=1}^n  T_j f_{ij} \ot h_i \Big\| ^p
  \,d\mu
  \\ & =  \E \int_M  \Big\|  \sum_{i=1}^n g_i \ot h_i \Big\| ^p   \,d\mu
  \\ & \eqsim  \E \int_M \widetilde\E
         \Big|  \sum_{i=1}^n \widetilde r_i g_i \Big|^p   \,d\mu
  \\ & =   \widetilde\E \E \Big\|
           \sum_{j=1}^k r_j T_j \Big(
            \sum_{i=1}^n \widetilde r_i f_{ij} \Big) \Big\|_p^p
  \\ & \lesssim  \widetilde\E \E \Big\|
           \sum_{j=1}^k r_j \Big(
            \sum_{i=1}^n\widetilde r_i f_{ij} \Big) \Big\|_p^p
  \\ & \eqsim  \E\Big\| \sum_{j=1}^k r_j  F_j \Big\|_p^p.
 \end{align*}
The last step follows by performing the computation in reverse
order. The result follows by an application of the
Kahane-Khintchine inequalities.
 \epf

We need the following duality result for $R$-bounded families
\cite[Proposition 3.5]{KaKuWe}.
Let $I_1 \in \L(L^2(E,\mu))$ be the orthogonal projection defined in
Section \ref{subsec:secQ} and let $I_X$ be the identity operator on
a Banach space $X.$ Then $X$ is said to be $K$-convex if the
operator $I_1 \ot I_X$ on $L^2(E,\mu) \ot X$ extends to a bounded
operator on the Lebesgue-Bochner space $L^2(E,\mu;X)$ (see, e.g.,
\cite{DJT,Pi89}).

\bpr \label{prop:Bconvex} If $X$ and $Y$ are $K$-convex Banach
spaces, then a family  $\mathscr{T}\subseteq\calL(X,Y)$ is
$R$-bounded if and only if the adjoint family
$\mathscr{T}\s\subseteq \calL(Y\s,X\s)$ is $R$-bounded. \epr

We shall apply this proposition to the $K$-convex spaces $X=L^p$ and
$Y=\uL^p$ for $1<p< \infty$.

\subsection{Radonifying operators}

It will be convenient to exploit the connection between square
functions and radonifying norms. Let $(\ga_n)_{n\ge 1}$ be a
Gaussian sequence, i.e., a sequence of independent standard
normal random variables. If $\HH$ is a Hilbert space and $X$
is a Banach space, we denote by $\ga(\HH,X)$ the completion of
the finite rank operators from $\HH$ to $X$ with respect to
the norm
$$ \Big\n \sum_{j=1}^k h_j\otimes x_j\Big\n_{\ga(\HH,X)}^2
= \E \Big\n \sum_{j=1}^k \ga_j x_j
\Big\n^2,
$$ where it is assumed that the vectors $h_1,\dots,h_k$ are orthonormal in
$\HH$. We have a continuous inclusion $\ga(\HH,X)\embed \L(\HH,X)$.
Operators in $\L(\HH,X)$ belonging to $\ga(\HH,X)$ are called {\em
radonifying}; this terminology is explained by the fact that an
operator $T\in \L(\HH,X)$ is radonifying if and only if there exists
a centred Gaussian Radon measure on $X$ whose covariance operator
equals   $TT\s$.

We continue with an observation about repeated radonifying norms which follows
from the Kahane-Khintchine inequalities and Fubini's theorem.
For a proof see, e.g., \cite{vNWeis}.

\begin{proposition}\label{prop:alpha}
Let $(S_1,\sigma_1)$, $(S_2,\sigma_2)$ be $\sigma$-finite
measure spaces, and let $1\le p<\infty$. The mapping
$$ f_1\otimes (f_2\otimes g) \mapsto(f_1\otimes f_2)\otimes g,\quad
 f_i \in L^2(S_i,\sigma_i),\ g\in \uL^p, $$
extends uniquely to an isomorphism of Banach spaces
$$\ga(L^2(S_1,\sigma_1), \ga(L^2(S_2,\sigma_2), \uL^p))\simeq
\ga(L^2(S_1 \times S_2,\sigma_1\otimes \sigma_2), \uL^p).
$$
\end{proposition}

The next multiplier result is a slight extension of a result due to
Kalton and Weis \cite{KWinprep} and can be proved in the same way.

\begin{proposition}\label{prop:KW} Let $X$ and $Y$ be Banach spaces, and
let $K:M \to \calL(X,Y)$ be a function such that $K(\cdot)x$ is
strongly $\mu$-measurable for all $x\in X$. If the set $\mathscr{T}_K =
\{K(\xi): \ \xi\in M\}$ is $R$-bounded, then the mapping
$$T_K:  f(\cdot)\otimes x \mapsto f(\cdot)\otimes K(\cdot)x, \quad
 f\in L^2, x\in X, $$
extends uniquely to a bounded operator $T_K$ from $\ga(L^2,X)$
to $\ga(L^2,Y)$ of norm $\n T_K\n\le R(\mathscr{T}_K).$
\end{proposition}

\subsection{Square functions}

In this subsection we recall how
$R$-bounded families in $L^p$-spaces and radonifying operators
into $L^p$-spaces can be characterised by square functions.

The first result follows from a standard application of the
Kahane-Khintchine inequalities; see \cite{KuW}.

\begin{proposition}
\label{prop:R-bdd-SQF} A family $\mathscr{T}\subseteq
\calL(L^p,\uL^p)$ is $R$-bounded if and only if there exists a
constant $C$ such that for all $T_1,\dots,T_N \in \mathscr{T}$ and
$f_1,\dots, f_N\in L^p$,
$$
\Big\|  \Big(\sum_{n=1}^N \n T_n f_n\n^2\Big)^{1/2} \Big\|_p \le
C\Big\|  \Big(\sum_{n=1}^N |f_n|^2\Big)^{1/2} \Big\|_p.
$$
\end{proposition}

The next result is another consequence of the
Kahane-Khintchine inequalities; see \cite{BrvN03, NVW}.

\bpr\label{prop:sqf-radonif} Let $(S,\sigma)$ be a $\sigma$-finite
measure space and let $1\le p<\infty$ and $\frac1p+\frac1q=1$. Let
$k: S\to \uL^p$ be a strongly $\sigma$-measurable function such that
for all $g\in \uL^q$ the function $s\mapsto \ip{ k(s),g}$ is square
integrable. The following assertions are equivalent: \ben
\item[\rm(1)] The operator
$I_k \in \calL(L^2(S,\sigma), \uL^p)$ defined by
$$\lb I_k f,g\rb = \int_S  f(s)\ip{ k(s),g }\,d\sigma(s),
\quad f\in
L^2(S,\sigma), \ g\in \uL^q,$$
is radonifying;
\item[\rm(2)] The square function
$\big( \int_S \| k(s) \|^2
     \;d\sigma(s) \big)^{1/2}$
defines an element of $L^p$.
\een
In this situation we have an equivalence of norms
$$ \n I_k\n_{\ga(L^2(S,\sigma),\uL^p)}\eqsim
\Big\n \Big( \int_S \| k(s) \|^2
     \;d\sigma(s) \Big)^{1/2}\Big\n_{p}.
$$
\epr

In what follows we always identify $k$ with the operator $I_k$.

\section{Intermezzo II: $H^\infty$-functional calculi}
\label{sec:preliminaries2}

In this section we recall some basic facts concerning
$H^\infty$-functional calculi. For more information we refer to the
monographs \cite{DHP,Haase}, the lecture notes \cite{ADMlect,KuW},
and the references given therein.

For $\om \in (0,\pi)$ we consider the open sector
$$\S_\om^+ := \{ z
\in \C: \ z \neq 0,\ |\arg z| < \om \}.$$ A closed operator $A$
acting on a Banach space $X$ is said to be {\em sectorial}  of angle
$\om\in (0,{\pi})$ if $\sigma(A) \subseteq \overline{\S_\om^+}$ and
the set $\{ \l (\l-A)^{-1} : \l \notin \overline{\S_\th^+}\}$ is bounded
for all $\th \in (\om,{\pi})$. The least angle of sectoriality is
denoted by $\om^+(A)$. If $A$ is sectorial and the set $\{ \l
(\l-A)^{-1} : \l \notin \overline{\S_\th^+}\}$ is $R$-bounded for all
$\th \in (\om,{\pi})$, then $A$ is said to be $R$-{\em sectorial} of
angle $\om\in (0,{\pi})$. The least angle of $R$-sectoriality is
denoted by $\om_R^+(A).$

We will frequently use the fact \cite[Proposition 2.1.1(h)]{Haase}
that a sectorial operator $A$ on a reflexive Banach space $X$ induces a
direct sum decomposition
 \begin{align} \label{eq:reflexive}
X = \Null(A) \oplus \ov{\Ran(A)}.
 \end{align}

Let $H^\infty(\S_\th^+)$ be the space of all bounded holomorphic
functions on $\S_\th^+$, and let $H_0^\infty(\S_\th^+)$ denote the
linear subspace of all $\psi \in H^\infty(\S_\th^+)$ which satisfy an
estimate
$$|\psi(z)| \leq C
\Big(\frac{|z|}{1+|z|^2}\Big)^\a, \quad z \in \S_\th^+,
$$
for some $\a>0$ and $C\geq0$. If $A$ is a sectorial operator
and $\psi$ is a function in $H_0^\infty(\S_\th^+)$ with $0 <
\om^+(A) < \th'< \th<\pi$, we may define the bounded operator
$\psi(A)$ on $X$ by the Dunford integral
 $$\psi(A)x := \frac{1}{2\pi i} \int_{\partial \S_{\th'}^+} \psi(z) (z-A)^{-1}x \, dz, \quad x \in X, $$
where $\partial_{\S_{\th'}^+}$ is parametrised
counter-clockwise. By Cauchy's theorem this definition does
not depend on the choice of $\th'.$

A sectorial operator $A$ on $X$ is said to admit a {\em bounded
$H^\infty(\S_{\th}^+)$-functional calculus}, or a {\em bounded
$H^\infty$-functional calculus of angle $\theta$},  if there exists a
constant $C_\theta\ge 0$ such that for all $\psi \in
H_0^\infty(\S_\th^+)$ and all $x\ \in X$ we have
 \begin{align*}
  \| \psi(A) x\| \leq C_\theta \|\psi\|_{\infty} \|x\|,
 \end{align*}
where $\|\psi\|_{\infty}=\sup_{z\in \S_\th^+}|\psi(z)|$. The infimum
over all possible angles $\th$ is denoted $\om_{H^\infty}^+(A).$ We
say that a sectorial operator $A$ admits a {\em bounded
$H^\infty$-functional calculus} if it admits a bounded
$H^\infty(\S_{\th}^+)$-functional calculus for some $0<\theta<\pi$.

The following result is well known; see, e.g., \cite[Theorem
2.20]{KuW}.

\blem\label{lem:Rbdd} Let $A$ be $R$-sectorial of angle
$\omega_R^+(A)<\frac12 \pi$ on $X$, and let $S$ be the bounded
analytic $C_0$-semigroup generated by $-A$. The family $\{S(t): \
t\ge 0\}$ is $R$-bounded in $\calL(X)$. \elem

In the remainder of this section we work in an $L^p$-setting and
use the notations of the previous section. As before we write $L^p =
L^p(M,\mu)$ and $\uL^p = L^p(M,\mu;\H)$, where $(M,\mu)$ is a
$\sigma$-finite measure space and $\H$ is a Hilbert space.

The following result is taken from \cite{CDMY,LeMerdy2} where the
result is proved for scalar-valued $L^p$-spaces.
An extension to more a general class of Banach spaces can be
found in \cite{KaKuWe}.

\bpr\label{prop:Hinfty-sfe}
 Let $1<p<\infty$ and let $A$ be an $R$-sectorial operator on $\uL^p.$
Let $\om_R^+(A) < \th <\pi.$ For all non-zero $\varphi,\psi \in
H_0^\infty(\S_\th^+)$ we have
 $$\Big\| \Big( \int_0^\infty \| \varphi(tA) F\|^2 \,\frac{dt}{t}
   \Big)^{1/2} \Big\|_p \eqsim
  \Big\| \Big( \int_0^\infty \| \psi(tA) F\|^2 \,\frac{dt}{t}
   \Big)^{1/2} \Big\|_p,$$
with implied constants independent of $F.$
Moreover, the following assertions are equivalent:
\begin{enumerate}
\item[\rm(1)] $A$ admits a bounded $H^\infty$-calculus;
\item[\rm(2)] For
some (equivalently, for all) non-zero $\psi \in H_0^\infty(\S_{\theta}^+)$ we have $$ \|F- P_{\Null(A)} F\|_p \lesssim  \Big\|
\Big( \int_0^\infty \| \psi(tA) F\|^2 \,\frac{dt}{t}
   \Big)^{1/2} \Big\|_p\lesssim \|F\|_p, \quad F \in \uL^p.
$$
\end{enumerate}
In {\rm (2)},
 $P_{\Null(A)}$ is the projection onto $\Null(A)$ with kernel
$\ov{\Ran(A)}$  along the decomposition \eqref{eq:reflexive}. If
these equivalent conditions are fulfilled, then $\om_R^+(A) =
\om_{H^\infty}^+(A)$. \epr

In the next result we let $1<p<\infty$ and consider two
$R$-sectorial operators $L$ and $\uA$. We assume that $-L$ and $-\uA$ generate
$R$-bounded analytic $C_0$-semigroups $P$ and $\uS$
 on $L^p$ and $\H$.
We denote by $-\uL$ the generator of the tensor product
$C_0$-semigroup
 $\uP = P \ot \uS$ on $\uL^p.$ This operator is $R$-sectorial of angle $\max\{\om_R^+(L),\om^+(\uA)\} <
\frac12\pi$ on $\uL^p.$

We consider the following three square function norms:
 \begin{align*}
 \|u\|_{\uA} & :=
    \Big( \int_0^\infty \| t \uA\, \uS(t) u\|^2 \,\frac{dt}{t}
  \Big)^{1/2}, \quad u\in \H; \\
 \|f\|_{p,L} & :=
   \Big\|  \Big( \int_0^\infty | t L P(t) f|^2 \,\frac{dt}{t}
  \Big)^{1/2} \Big\|_p,\quad f\in L^p;
  \\
 \|F\|_{p,\uL} &:=
   \Big\|  \Big( \int_0^\infty \n t \uL\, \uP(t) F\n^2 \,\frac{dt}{t}
  \Big)^{1/2} \Big\|_p, \quad F\in \uL^p.
 \end{align*}

\bpr \label{prop:tensorcalculus} Under the above assumptions we have:
\ben
\item[\rm(1)] If  $\| u\|_{\uA} \lesssim \|u\|$ for all $u\in \H$
and $\| f \|_{p,L} \lesssim \|f\|_p$ for all $f\in L^p$, then $\| F
\|_{p,\uL} \lesssim \|F\|_p$ for all $ F\in \uL^p.$
\item[\rm(2)] If  $\| u \|_{\uA} \gtrsim  \|(I-P_{\Null(\uA)})u\|$ for all $u\in \H$
and $\| f \|_{p,L} \gtrsim \|(I-P_{\Null(L)})f\|_p$ for all $f\in
L^p,$  then
$\| F \|_{p,\uL} \gtrsim \|(I-P_{\Null(\uL)})F\|_p$ for all
$F\in \uL^p.$
\een
As a consequence, if  $\uA$ and $L$ have
bounded $H^\infty$-functional calculi of angles less than
$\frac12 \pi$, then ${\uL}$ has a bounded $H^\infty$-functional
calculus of angle less than $\frac12 \pi$. \epr

 \bpf
Let us first show that $(1)$ implies $(2).$ It is
well known that the assumptions of $(2)$ imply the dual estimates
$  \|u\|_{\uA^*} \lesssim \|u\|$ and $  \|f\|_{q,L^*} \lesssim \|f\|_{q}$,
where $\frac1p+\frac1q=1.$ By $(1)$ we find that $  \|F\|_{q,\uL^*}
\lesssim \|F\|_{q}$ and by duality we obtain the conclusion of $(2).$

The final assertion follows by combining $(1)$ and $(2)$ with
Proposition \ref{prop:Hinfty-sfe}.

It remains to prove $(1).$ We proceed in three steps.

{\em Step 1:} \ We prove that $$
  \|  t (I\ot \uA) (I\ot \uS(t))
    F\|_{\ga(L^2(\R_+,\frac{dt}{t}),\uL^p)}
  \lesssim \|F\|_p.$$

For $F \in \uL^p$ we have, for $\mu$-almost all $x\in M$,
$$\Big( \int_0^\infty \| t(I\ot \uA) (I\ot \uS(t)) F(x) \|^2
\frac{dt}{t} \Big)^{1/2}
  \lesssim \|F(x)\|.$$
Integrating this estimate over $M$ yields
 $$\Big\| \Big( \int_0^\infty \| t(I\ot \uA) (I\ot \uS(t)) F \|^2
 \frac{dt}{t} \Big)^{1/2}  \Big\|_p
    \leq \| F  \|_p. $$

{\em Step 2:}\ We prove that
 $$\| t (L\ot I)(P(t)\ot I)  F \|_{\ga(L^2(\R_+,\frac{dt}{t}),\uL^p)}
 \lesssim \|F\|_p.$$

 Let $(h_j)_{j=1}^k$ be a finite orthonormal
system in $\H$ and pick $F := \sum_{j=1}^k f_j \ot h_j \in \uL^p.$
For $f\in L^p$ let $(U f)(t) := t LP(t)f$, and notice that $U$ is a
bounded operator from $L^p$ into $\ga(L^2(\R_+,\frac{dt}{t}),L^p)$ by
the assumption in (1) and Proposition \ref{prop:sqf-radonif}.

Let $(r_j')_{j\ge 1}$ and $(\ga_j')_{j\ge 1}$ be a Rademacher and a
Gaussian sequence respectively on a probability space $(\O',\P')$.
Noting the pointwise equality
$$\| t(L\ot I) (P(t)\ot I) F \|^2 = \sum_{j=1}^k | U f_j(t) |^2$$
we have
 \begin{align*}
\Big\| \Big( \int_0^\infty \| t(L\ot I) (P(t)\ot I) F \|^2
     \frac{dt}{t} \Big)^{1/2}  \Big\|_p
  & = \Big\| \Big( \int_0^\infty \sum_{j=1}^k | U f_j(t) |^2
     \frac{dt}{t} \Big)^{1/2}  \Big\|_p
  \\ & = \Big\n \Big(\int_0^\infty \E' \Big|\sum_{j=1}^k r_j'
U f_j (t)\Big|^2\frac{dt}{t}\Big)^{1/2}\Big\n_p
  \\ & \eqsim \Big\n \sum_{j=1}^k r_j' U
  f_j\Big\n_{\ga(L^2(\R_+ \times \Om',
\frac{dt}{t}\otimes \P'), L^p)}
  \\ & \stackrel{(*)}{\eqsim} \Big\n U \sum_{j=1}^k r_j' f_j
      \Big\n_{\ga(L^2(\Om',\P'), \ga(L^2(\R_+,\frac{dt}{t}), L^p))}
 \\ & \stackrel{(**)}{\lesssim} \Big\n
   \sum_{j=1}^k r_j' f_j\Big\n_{\ga(L^2(\Om',\P'), L^p)}
  \\ & \stackrel{(***)}{=}
\Big(\E' \Big\n \sum_{j=1}^k \ga_j' f_j \Big\n_{L^p}^2 \Big)^{1/2}
 \\ & \eqsim \Big(\E' \Big\n \sum_{j=1}^k \ga_j' f_j \Big\n_{L^p}^p \Big)^{1/p}
 \\ & = \Big\n\Big(  \E' \Big|\sum_{j=1}^k \ga_j' f_j \Big|^p \Big)^{1/p} \Big\n _{L^p}
 \\ & \eqsim \Big\n   \Big(\sum_{j=1}^k |  f_j |^2\Big) ^{1/2} \Big\n
_{L^p}
 \\ & =  \n F \n_p.
 \end{align*}
In $(*)$ we used Proposition \ref{prop:alpha}, in $(**)$ we used the
boundedness of $U$ from $L^p$ into $\ga(L^2(\R_+,\frac{dt}{t}),L^p)$,
and in $(***)$ the definition of the radonifying norm.

\medskip
{\em Step 3:} \ We combine the previous estimates. By Lemma
\ref{lem:Rbdd} the family $\{P(t): \ t\ge 0\}$ is $R$-bounded on
$L^p$. Hence by Proposition \ref{prop:Rtensor} the family $\{ P(t)
\ot I: \ t \geq 0\}$ is $R$-bounded on $\uL^p$. Also, by a simple
application of Fubini's theorem, $\{ I \ot \uS(t) :\ t \geq 0\}$ is
$R$-bounded. Combining these facts with Proposition \ref{prop:KW},
for $F\in \uL^p$ we obtain
 \begin{align*}
\Big\| \Big( \int_0^\infty \| t \uL \,\uP F \|^2
     \frac{dt}{t} \Big)^{1/2}  \Big\|_p
  & \eqsim \|  t \uL\, \uP F  \|_{\ga(L^2(\R_+,\frac{dt}{t}),\uL^p)}
 \\ & \lesssim \| (I\ot S(t)) t (L\ot I)(P(t)\ot I)  F  \|_{\ga(L^2(\R_+,\frac{dt}{t}),\uL^p)}
 \\ & \qquad       +  \| (P(t)\ot I) t (I\ot \uA) (I\ot \uS(t))F  \|_{\ga(L^2(\R_+,\frac{dt}{t}),\uL^p)}
\\ & \lesssim \| t (L\ot I)(P(t)\ot I)  F  \|_{\ga(L^2(\R_+,\frac{dt}{t}),\uL^p)}
 \\ & \qquad         +  \|  t (I\ot \uA) (I\ot \uS(t))F  \|_{\ga(L^2(\R_+,\frac{dt}{t}),\uL^p)}
 \\ & \lesssim \|F\|_p.
 \end{align*}
\epf

\begin{remark}
The final assertion in Proposition \ref{prop:tensorcalculus} is due
to Lancien, Lancien, and Le Merdy \cite[Theorem 1.4]{LLM} who proved
it using operator-valued $H^\infty$-functional calculi.
\end{remark}

The next proposition has been proved in \cite[Theorem 3.5,
Remark 3.6]{LeMerdy2} (for $\H=\C$) and can be extended to a
more general class of Banach spaces including the spaces
$\uL^p$ \cite{HaKu06,vNWeis} (where a generalisation of the
crucial ingredient \cite[Proposition 3.3]{LeMerdy2} is
obtained).

\begin{proposition}\label{prop:sq-fc1}
Let $L$ be $R$-sectorial on $L^p$ of angle $\omega_R^+(L)<\frac12 \pi$, and let
$P$ be the bounded analytic $C_0$-semigroup
$P$ on $L^p$ generated by $-L$. Let
$U: \D(L) \to \uL^p$ be a linear operator, bounded with respect to
the graph norm of $\D(L)$. Consider the following statements.
 \begin{enumerate}
 \item[\rm(1)]
$\displaystyle \Big\| \Big(\int_0^\infty \|\sqrt{t} UP(t) f \|^2 \,\frac{dt}{t}
\Big)^{1/2} \Big\|_p    \lesssim \|f\|_p, \qquad f\in \D(L);$
 \item[\rm(2)] The family $\{\sqrt{t} UP(t): \ t>0\}$ is
$R$-bounded in $\L(L^p,\uL^p)$.
 \end{enumerate}
Then $(1)$ implies $(2).$ If $L$ satisfies the square function
estimate
$$
   \Big\|  \Big( \int_0^\infty | t L P(t) f|^2 \,\frac{dt}{t}
  \Big)^{1/2} \Big\|_p \lesssim \|f\|_p,$$ then $(2)$ implies $(1)$.
 \end{proposition}

 \begin{remark}\label{rem:sq-fc1}
In \cite{LeMerdy2} and other works in the mathematical systems
theory literature, condition (2) is replaced by the following equivalent
condition:
 \beni
  \item[\rm(2$'$)] The family $\{t U(I+t^2L)^{-1}: \ t>0\}$ is
$R$-bounded.
 \een
That (2) implies (2$'$) follows by taking Laplace transforms
and the opposite direction is observed in
\cite[(3.12)]{LeMerdy2}. Since our computations involve
semigroups rather than resolvents we find it more natural to
use (2).
 \end{remark}

Below we shall also need the notion of an ($R$-){\em bisectorial} operator,
which is analogous to
that of an ($R$-)sectorial operator, the only difference being that
the sector $\Sigma_\theta^+$ is replaced by the bisector
$\Sigma_\theta = \Sigma^+_\theta \cup\Sigma_\theta^-$, where
$\Sigma_\theta^- = - \Sigma_\theta^+$. Many results in the literature on
($R$-)sectorial operators carry over to ($R$-)bisectorial operators,
with only minor changes in the proofs. We refer to the lecture notes
\cite{ADM} for more details.

\section{Proof of Theorems \ref{thm:gradient_est} and \ref{thm:LPS}}\label{sec:gradient}

We return to the main setting of the paper and take up our study of the
operators $\LB$ and $\uLB$ introduced in Sections
\ref{sec:DBD} and \ref{sec:DDB}. Notations are again as in these sections.

For functions $f\in \F C_{\rm b}^\infty(E;\Dom(\AB))$ we consider the
Littlewood-Paley-Stein square functions
  \begin{align*}
  \HH f(x) &:= \Big( \int_0^\infty \| \sqrt{t}\uD \PB(t) f(x) \|^2
     \;\frac{dt}{t} \Big)^{{1/2}}, \quad x\in E,
\\  \mathcal{G} f(x) &:= \Big( \int_0^\infty
            \| t\uD \QB(t) f(x) \|^2   \,\frac{dt}{t} \Big)^{1/2}, \quad x\in E.
  \end{align*}
where $\QB$ denotes the analytic $C_0$-semigroup generated by $-\sqrt{\LB}$.

The functions $t\mapsto \uD \PB(t)f$ are analytic in a sector
containing $\R_+$, and therefore a well-known result of Stein
\cite{Ste70} allows us to select a pointwise version $(t,x)
\mapsto \uD \PB(t)f(x)$ which is analytic in $t$ for every
fixed $x$. Using such a version, we see that $\HH f$ is well
defined almost everywhere (but possibly infinite). The square
function $\mathcal{G} f$ is well defined by similar reasoning.

In Section \ref{sec:Kato} we shall need the following inequality. The argument is taken from \cite{CDL}.

 \blem \label{lem:HdominatesG}
For all $f\in \F C_{\rm b}^\infty(E;\Dom(\AB))$ we have $\mathcal{G}f \leq  \mathcal{H}f$
$\mu$-almost everywhere.
 \elem

 \begin{proof}
Using the representation
$$\QB(t) f
 = \frac{1}{\sqrt{\pi}} \int_0^\infty
 \frac{e^{-u}}{\sqrt{u}} \PB\big(\frac{t^2}{4u}\big)f \,du$$
and the closedness of $\uD$,
 \begin{align*}
 \mathcal{G}^2 f(x)
  & = \int_0^\infty
            \|t \uD \QB(t) f(x) \|^2   \,\frac{dt}{t}
 \\ & \leq \frac{1}{\pi} \int_0^\infty
     \Big( \int_0^\infty \Big\|t \uD \PB\big(\frac{t^2}{4u}\big)f(x)\Big\|
      \frac{e^{-u}}{\sqrt{u}} \,du \Big)^2 \frac{dt}{t}.
 \end{align*}
Since $\int_0^\infty \frac{e^{-u}}{\sqrt{u}}\,du =\sqrt{\pi}$ we may apply Jensen's inequality to obtain
 \begin{align*}
 \mathcal{G}^2 f(x)
 & \leq \frac{1}{\sqrt{\pi}} \int_0^\infty
     \Big( \int_0^\infty  \Big\|t \uD \PB\big(\frac{t^2}{4u}\big)f(x)\Big\|^2 \frac{e^{-u}}{\sqrt{u}} \,du \Big)
     \,\frac{dt}{t}
 \\& = \frac{1}{\sqrt{\pi}} \int_0^\infty
     \Big( \int_0^\infty \Big\|t \uD \PB\big(\frac{t^2}{4u}\big)f(x)\Big\|^2
     \,\frac{dt}{t}\Big)\frac{e^{-u}}{\sqrt{u}}  \,du
 \\& = \frac{2}{\sqrt{\pi}} \int_0^\infty
     \Big( \int_0^\infty  \big\|\sqrt{s}\uD \PB(s) f(x)\big\|^2
     \,\frac{ds}{s}\Big) \sqrt{u}e^{-u} \,du
 \\& =  \mathcal{H}^2 f (x).
 \end{align*}

 \end{proof}

The main results of this section are the following two theorems, which together imply part (2) of Theorem \ref{thm:gradient_est} as well as
Theorem \ref{thm:LPS}. Part (1) of Theorem \ref{thm:gradient_est} is contained
in Theorem \ref{thm:pointwisegrad}.

\begin{theorem}[$R$-Gradient bounds]\label{thm:Rgradientestimates}
Assume {\rm(A1)}, {\rm(A2)}, {\rm (A3)}, and let $1<p<\infty$.
Then $\D(\LB)$ is a core for $\D(\uD)$ and the families
$$\big\{\sqrt{t}\uD \PB(t): \ t>0\big\} \ \hbox{ and } \ \big\{ t \uD (I + t^2 \LB)^{-1} : \ t>0\big\}
$$
are $R$-bounded in $\calL(L^p,\uL^p).$
\end{theorem}

 \begin{theorem}[Littlewood-Paley-Stein inequalities]\label{thm:square_functions}
Assume {\rm(A1)}, {\rm(A2)}, {\rm (A3)}, and let $1<p<\infty$.
For all $f\in \F C_{\rm b}^\infty(E;\Dom(\AB))$ we have the square function estimate
$$\|f- P_{\Null_p(L)} f\|_p \lesssim
\|\mathcal{H}f\|_p \lesssim \|f\|_p,
$$
where $P_{\Null_p(L)}$ is the projection onto $\Null_p(L)$ along the direct sum decomposition $L^p = \Null_p(L)\oplus \overline{\Ran_p(L)}$.
%If the equivalent conditions of Theorem
%\ref{thm:Kato} hold, then we also have
%$$\|f- P_{\Null_p(L)} f\|_p \lesssim \|\mathcal{H}f\|_p,
%$$
%where $P_{\Null_p(L)}$ is the projection onto $\Null_p(L)$ along the direct sum %decomposition $L^p = \Null_p(L)\oplus \overline{\Ran_p(L)}$.
 \end{theorem}

By Theorem \ref{thm:square_functions} the square function
$\mathcal{H}f$ is actually well-defined
for arbitrary $f\in L^p$, and by approximation Theorem \ref{thm:square_functions}
extends to all of $L^p$. Since we do not need these
observations we leave the details to the reader.

For the proofs of both theorems we distinguish between the
cases $1<p\leq 2$ and $2<p<\infty$.
For $1<p\leq 2$ we show by a direct argument that $\HH$ is
$L^p$-bounded and deduce from this that $\D(\LB)$ is a core for $\D(\uD)$.
Theorem \ref{thm:Rgradientestimates} is then a
consequence of Proposition \ref{prop:sq-fc1}. For $2<p<\infty$ we
first derive
Theorem \ref{thm:Rgradientestimates} from a pointwise
gradient bound and a duality argument involving maximal functions.
Since $\LB$ has a bounded $H^\infty$-calculus of angle $< \frac12
\pi$ by Lemma \ref{lem:LBcalc}, the right-hand side estimate of Theorem \ref{thm:square_functions}
then follows by an application of Proposition \ref{prop:sq-fc1}.
Finally, the left-hand side inequality  of Theorem \ref{thm:square_functions} is proved, for $1<p<\infty$, by a duality argument. 

We begin with an easy observation.

 \begin{lemma} \label{lem:LBcalc}
Let $1<p<\infty$. The operator $\LB$ is $R$-sectorial and admits a
bounded $H^\infty$-calculus on $L^p$ of angle $ \om_{H^\infty}^+(\LB)
= \om_R^+(\LB) < \frac12 \pi.$ Moreover,
\ben
\item[\rm(1)]
The family $\{\PB(t): \ t\ge
0\}$ is $R$-bounded in $\calL(L^p)$;
\item[\rm(2)]
The family $\{\uPB(t): \ t\ge 0\}$ is $R$-bounded in $\calL(\overline{\Ran_p(\uD)})$. \een
 \end{lemma}

 \begin{proof}
Since $-\LB$ generates an analytic $C_0$-semigroup of positive
contractions on $L^p$ for all $1<p<\infty,$ the first part
follows from \cite[Corollary 5.2 and Theorem 5.3]{KW}.
Assertion (1) follows from Lemma \ref{lem:Rbdd}, and assertion
(2) follows by combining (1) with the identity $\uPB = \PB
\otimes \uSB$ and Proposition \ref{prop:Rtensor}. \epf

We continue with a simple extension of a well-known result of Cowling
\cite[Theorem 7]{Cow83} (see also \cite{Tag}). For the convenience of
the reader we give a sketch of the proof.

 \bpr \label{prop:Cowling}
Let $(M,\mu)$ be a $\sigma$-finite measure space and let $T$
be an analytic $C_0$-semigroup of positive operators on $L^2
:= L^2(M,\mu)$ satisfying $\|T(t)f\|_{p} \leq \|f\|_p$ for all
$f \in L^2 \cap L^p,$ $t \geq 0$ and $1\leq p \leq \infty.$
Let
$$T_\star f(x) := \sup_{t>0} |T(t)f(x)|.$$ Then for
$1<p<\infty$ we have
$$\|T_\star f\|_p \lesssim \|f\|_p, \quad f \in L^p.$$
 \epr

 \bpf
Let $-L$ denote the generator of $T$ in $L^p.$ By \cite[Corollary
5.2]{KW}, $L$ has a bounded $H^\infty$-calculus of angle $\om <
\frac12 \pi$. The key idea of the proof is to write
$$T(t)f = \frac1t \int_0^t T(s)f \,ds + \Big( T(t)f -
\frac{1}{t} \int_0^t T(s)f \,ds \Big) = \frac1t \int_0^t T(s)f
\,ds + m(tL)f,$$ where $m(z) := e^{-z} - \int_0^1 e^{-sz}
\,ds$.

By the
Hopf-Dunford-Schwartz ergodic theorem \cite[Theorem
6.12]{Krengel} we have
$$\Big\| \sup_{t>0} \Big|\frac1t \int_0^t T(s)f \,ds\Big|\, \Big\|_p \lesssim \|f\|_p,$$
so that it remains to prove that $\big\| \sup_{t
>0} |m(tL) f|\,\big\|_p \lesssim \|f\|_p.$

Let $n := m \circ \exp$ and let $\hat n$ be its Fourier
transform. Using the identities $$m(z) = \frac{1}{2\pi}
\int_\R \hat n(u) z^{iu}\,du,\qquad \hat n (u) = (1 -
(1+iu)^{-1})\Gamma(iu),$$ and the estimate $|\hat n (u)| \leq
C e^{-\frac12 \pi|u|}$ (see \cite{Cow83}) we obtain
 \begin{align*}
 \sup_{t >0} |m(tL)F|
   \leq  \sup_{t >0} \frac{1}{2\pi} \int_\R | \hat
 n(u)| \,|(tL)^{iu}F| \,du
 \lesssim
    \frac{1}{2\pi} \int_\R e^{-\frac12 \pi|u|} |L^{iu}F|
   \,du.
 \end{align*}
From the $H^\infty$-calculus of $L$ we obtain $\|L^{iu} f\|_p \lesssim
e^{\om|u|}\|f\|_p.$ Taking $L^p$-norms we obtain
 \begin{align*}
 \big\|\sup_{t >0} |m(tL)F| \,\big\|_p
 & \lesssim
  \frac{1}{2\pi} \int_\R e^{-\frac12 \pi|u|}
  \|L^{iu}F\|_p
   \,du
  \lesssim
  \frac{1}{2\pi} \int_\R e^{(\om-\frac12 \pi)|u|}\|F\|_p
    \,du
    \lesssim \|F\|_p.
 \end{align*}

 \epf

\subsection{The case $1<p\leq 2$}

We begin with some preliminary observations.

 \blem \label{lem:uSBintegral}
For $h \in \Dom(V)$ we have
 \begin{align*}
 \int_0^\infty \| \uSB(t) Vh \|^2\,dt   \leq (2k)^{-1}\|h\|^2.
 \end{align*}
 \elem

 \begin{proof}
Let $t >0.$ Using Lemma \ref{lem:SVcommute} and the fact that
$\SB(t)h \in \Dom(A)$ by analyticity, we
obtain
 \begin{align*}
 \| \uSB(t) Vh \|^2
 = \|V \SB(t)h \|^2
  & \leq k^{-1}[BV \SB(t)h, V  \SB(t)h]
 \\ & = k^{-1}[\AB \SB(t)h,  \SB(t)h]
 \\ & = -(2k)^{-1} \frac{d}{dt}\| \SB(t)h\|^2.
 \end{align*}
Hence
 \begin{align*}
 \int_0^\infty \| \uSB(t) Vh \|^2\,dt
&  \leq (2k)^{-1}\limsup_{T \to \infty}  \int_0^T -
\frac{d}{dt}\|
 \SB(t)h\|^2 \,dt
\\ & = (2k)^{-1} \Big(\|h\|^2 - \liminf_{T\to\infty} \|\SB(T)h\|^2
     \Big)
\\ &  \leq (2k)^{-1}\|h\|^2.
 \end{align*}
 \end{proof}

 \begin{lemma} \label{lem:Hfinite}
Let $f\in\F C_{\rm b}^\infty(E;\Dom(\AB))$ and $F\in
\mathcal{F}C_{\rm b}^\infty(E;\Dom(\AB)) \ot \Dom(\AB)$ be such that
$\uD f= (I\ot V)F$. Then for all $1<p<\infty $ we have  $\HH f \in
L^p$ and $\n \HH f\n_p \lesssim \n F\n_p$.
 \end{lemma}

 \begin{proof}
By Proposition \ref{prop:Rtensor} and Lemma \ref{lem:Rbdd},
the set $\{\PB(t)\ot I: \ t \geq 0\}$ is $R$-bounded in
$\calL(\uL^p)$. Hence, by Propositions \ref{prop:KW},
\ref{prop:sqf-radonif},
 and Lemma \ref{lem:uSBintegral},
 \begin{align*}
 \|\HH f\|_p
  &  = \Big\|\Big(\int_0^\infty
      \|\uPB(t)\uD f\|^2 \,dt\Big)^{1/2}\Big\|_p
  \\& =  \Big\|\Big(\int_0^\infty
      \| (\PB(t)\ot I)(I\ot\uSB(t))(I\ot V)F\|^2
  \,dt\Big)^{1/2}\Big\|_p
  \\& \lesssim \Big\|\Big( \int_0^\infty
      \| (I\ot\uSB(t))(I\ot V)F\|^2
  \,dt\Big)^{1/2}\Big\|_p
  \\& \leq (2k)^{-1/2} \| F\|_p.
    \end{align*}
 \end{proof}

The following proof is based on a classical argument which
goes back to Stein \cite{Ste70}. The same idea has been
applied in the related works \cite{CG01,CDL,MN2,Sh92}. For the
convenience of the reader we include a proof.

 \begin{proof}[Proof of the first part of Theorem \ref{thm:square_functions}, $1 < p\leq 2$]
First we show that it suffices to prove the estimate for functions $f \in
\F  C_{\rm b}^\infty(E;\Dom(\AB))$ satisfying $f \geq \eps$ for some
$\eps >0$.

Fix $f = \varphi(\phi_{h_1}, \ldots, \phi_{h_k}) \in \F C_{\rm
b}^\infty(E;\Dom(\AB))$ of the usual form. Pick functions $m_n
\in C_{\rm b}^\infty(\R^k)$ satisfying $m_n\geq0,$
${\rm{supp}}(m_n) \subseteq [-\frac1n,\frac1n]^k$, and
$\|m_n\|_1 = 1,$ and put
$$
\bal \psi_{n,\pm} & := \big(\varphi^\pm + \frac1n\big)*m_n, \\
    g_{n,\pm} & := \psi_{n,\pm}(\phi_{h_1}, \ldots, \phi_{h_k}),\\
   g_{n,\pm,j} & := \partial_j\psi_{n,\pm}(\phi_{h_1}, \ldots, \phi_{h_k}).
\eal
$$
Clearly $g_{n,\pm} \in \F C_{\rm b}^\infty(E;\Dom(\AB))$ satisfy
$\frac1n \leq g_{n,\pm}
 \leq \|\varphi\|_\infty + 1$, and
 \begin{align*}
\big\| \big(f^\pm + \frac1n\big) - g_{n,\pm}\big\|_p \to 0
 \end{align*}
by dominated convergence. From Lemma \ref{lem:Hfinite}
it follows that
  \begin{align*}
 \|  \HH f-\HH (g_{n,+} - g_{n,-})   \|_p
  &\leq \| \HH (f-(g_{n,+} - g_{n,-})) \|_p
  \\&\lesssim \Big\| \sum_{j=1}^k
   (f_j - (g_{n,+,j}-g_{j,n,-,j}))  \ot h_j \Big\|_p,
  \end{align*}
where $f_j = \partial_j\varphi(\phi_{h_1}, \ldots, \phi_{h_k}).$
Since the functions
$$g_{n,\pm,j} = \big(\partial_j \varphi(\phi_{h_1}, \ldots,
\phi_{h_k})\one_{\{\pm\varphi(\phi_{h_1}, \ldots,
\phi_{h_k})>0\}}\big)*m_n,$$ belong to $L^\infty$ uniformly in $n$,
we conclude by dominated convergence that $\|f_j - (g_{n,+,j} -
g_{n,-,j})\|_p \to 0.$ Therefore
 $\HH (g_{n,+}-g_{n,-}) \to \HH f$ in $L^p$ as $n\to\infty$.
Hence if $\| \HH g_{n,\pm}  \|_p \lesssim \|g_{n,\pm}\|_p$ with
constants not depending on $n$, then
 $$
 \bal\| \HH f \|_p
     & =          \limn \| \HH(g_{n,+} - g_{n,-}) \|_p
  \\ & \leq \limsup_{n\to\infty}  (\| \HH g_{n,+}  \|_p   + \|\HH g_{n,-}   \|_p)
  \\ & \lesssim \limsup_{n\to\infty}  (\|  g_{n,+}  \|_p  + \|  g_{n,-} \|_p)
  \\ & =  \| f^+ \|_p  + \| f^- \|_p
  \\ & \leq 2\| f \|_p.
 \eal $$
Thus it suffices to prove the result
for $f \in \F C_{\rm b}^\infty(E;\Dom(A))$ satisfying $f \geq \eps$
for some $\eps
>0$. Set
$$u(t,x) := \PB(t)f(x), \qquad x\in E, \ t>0,$$
and notice that by Mehler's formula \eqref{eq:Mehler} we have
$u(t,x)\geq\eps$ for all $x\in E$ and $t\geq0$.
 By Lemma \ref{lem:Fcore} we
have $u(t,\cdot)\in \F  C_{\rm b}^\infty(E;\Dom(\AB))
\subseteq \D(\LB)$ for all $t\ge 0$. Arguing as in
\cite{CG01,CDL,Sh92}, for $1 < p \leq 2$ we use Lemma
\ref{lem:Fcore} and a truncation argument to obtain that
$u(t,\cdot)^p\in\D(\LB)$ and
$$\bal
 (\partial_t + \LB)u(t,x)^p
 &  = p u(t,x)^{p-1} (\partial_t + \LB)u(t,x)
  \\& \quad - p(p-1) u(t,x)^{p-2}
   [B\uD u(t,x),\uD u(t,x) ]
 \\&  = - p(p-1) u(t,x)^{p-2}   [B\uD u(t,x),\uD u(t,x) ].
\eal$$ Hence, using the coercivity Assumption (A3),
\begin{align*}
\  \|\uD u(t,x)\|^2 &\leq k^{-1} [B\uD u(t,x),\uD u(t,x) ]
 \\ &  = -  \frac{1}{kp(p-1)}  u(t,x)^{2-p}(\partial_t + \LB)u(t,x)^p.
 \end{align*}
Now we set $$K(x) := -\int_0^\infty
  (\partial_t + \LB)u(t,x)^p\;dt
$$
 and
$$ u_\star(x) := \sup_{t >0} u(t,x) $$
to obtain
\begin{align*}
 \HH f(x)^2
   & =  \int_0^\infty \|\uD u(t,x)\|^2 \;dt
\\ & \leq - C_{p,k} \int_0^\infty u(t,x)^{2-p}(\partial_t + \LB)u(t,x)^p \;dt
\\ & \leq C_{p,k}  u_\star(x)^{2-p} K(x).
 \end{align*}
H\"older's inequality with exponents $\tfrac{2}{2-p}$ and
$\tfrac{2}{p}$ implies \beq\bal\label{eq:Hf}
 \int_E \HH f(x)^p \;d\mu(x)
  &  \leq  C_{p,k}^{\frac{p}2}
    \int_E u_\star(x)^{\frac{(2-p)p}{2}}
   K(x)^{\frac{p}{2}} \;d\mu(x)
 \\& \leq  C_{p,k}^{\frac{p}2}
   \Big(\int_E  u_\star(x)^{p} \;d\mu(x)\Big)^{\frac{2-p}{2}}
   \Big(\int_E  K(x) \;d\mu(x)\Big)^{\frac{p}2}.
\eal \eeq
Using the invariance of $\mu$ and the $L^p$-contractivity of $\PB$ we
obtain \beq\label{eq:K} \bal
  \int_E  K(x) \;d\mu(x)
 &  = - \int_0^\infty \int_E(\partial_t + \LB)u(t,x)^p\;d\mu(x) \;dt
 \\& = - \int_0^\infty \int_E \partial_t u(t,x)^p    \;d\mu(x) \;dt
 \\& = - \int_0^\infty \partial_t\int_E  u(t,x)^p    \;d\mu(x) \;dt
 \\& \leq \limsup_{t\to\infty} \big(\|f\|_p^p -
 \|u(t,\cdot)\|_p^p\big)
 \\& \leq  \|f\|_p^p,
\eal \eeq where the use of Fubini's theorem is justified by the
non-negativity of the integrand $K,$  and the interchange of
differentiation and integration by the fact that $f\in\F C_{\rm
b}^\infty(E;\Dom(\AB)).$

Combining \eqref{eq:Hf}, \eqref{eq:K} and Proposition
\ref{prop:Cowling} we conclude that
$$
 \| \HH f \|_p^p
  \lesssim \|u_\star\|_p^{\frac{(2-p)p}{2}}
  \|f\|_p^{\frac{p^2}{2}} \lesssim \|f\|_p^p.
$$
 \end{proof}

\begin{proof}[Proof of Theorem \ref{thm:Rgradientestimates}, $1 < p\leq 2$]
First we show that $\D(L)$ is contained in
$\D(\uD)$. Once we know this, Lemmas
 \ref{lem:Fcore} and \ref{lem:polyn-core} imply that
$\D(\LB)$ is even a core for $\D(\uD)$.

Fix a function $f\in \F C_{\rm b}^\infty(E;\Dom(\AB))$. From
Theorem \ref{thm:uLsemigroup} it follows that $s\mapsto e^{-s} \uD
\PB(s)f=e^{-s} \uPB(s)\uD f $\
 is Bochner integrable in $\uL^p$ and
$$ \int_0^\infty e^{-s} \uD \PB(s)f\,ds = (I+\uLB)^{-1}\uD f.$$
Since $s\mapsto e^{-s} \PB(s)f$ is Bochner integrable in $L^p$, the
closedness of $\uD$ implies that $(I+\LB)^{-1}f = \int_0^\infty
e^{-s} \PB(s)f\,ds \in \D(\uD)$ and
$$ \uD (I+\LB)^{-1}f = \uD \int_0^\infty e^{-s} \PB(s)f\,ds = \int_0^\infty e^{-s}
\uD \PB(s)f\,ds.$$ Moreover, by the Cauchy-Schwarz inequality,
$$
\bal \n \uD (I+\LB)^{-1}f\n_p & \le \Big\n \int_0^\infty e^{-s} \n
\uD \PB(s)f\n\,ds \Big\n_p
\\ & \le \frac1{\sqrt 2} \Big\n \Big(\int_0^\infty
\n \uD \PB(s)f\n^2\,ds\Big)^{1/2} \Big\n_p
\\ & = \frac1{\sqrt 2}
\n\HH f\n_p \lesssim \n f\n_p. \eal
$$
It follows that $\uD (I+\LB)^{-1}$ extends to a bounded operator
from $L^p$ to $\uL^p$. In view of the closedness of $\uD$ and Lemma
\ref{lem:Fcore}, the desired inclusion follows from this.
This concludes the proof that $\D(\LB)$ is a core for $\D(\uD)$.

The $R$-boundedness assertions follow from Proposition \ref{prop:sq-fc1}
and Remark \ref{rem:sq-fc1}.
\end{proof}

 \subsection{The case $2<p<\infty$}
\label{subsec:p2}

In case that $\PB$ is symmetric it is possible to use a variant of a
duality argument of Stein \cite{Ste70} to prove the boundedness of
$\HH.$ This approach has been taken in \cite{CG01}, but the proof
breaks down if $\LB$ is non-symmetric and we have to
proceed in a different way.

First we derive an explicit formula for the
semigroup $\PB$ which allows us to prove suitable gradient bounds.
Having obtained those gradient bounds we give a general argument
involving a maximal inequality for $\PB^*$ to prove the
$R$-boundedness of the collection $\{\sqrt{t}\uD \PB(t) :\ t >0\}.$
Since $\LB$ has a bounded $H^\infty$-calculus, we obtain the
boundedness of $\HH$ by an appeal to Proposition \ref{prop:sq-fc1}.
\medskip

We begin with some preliminary observations.
For $0 < t < \infty$ we define the operators $Q_t\in \calL(E\s,E)$ by
$$Q_t x^* := ii^*x^*-
i\SB^*(t)\SB(t)i^*x^*,$$ where $i:H\embed E$ is the inclusion
operator. The operators $Q_t$ are positive and symmetric,
i.e., for all $x^*,y^* \in E^*$ we have $ \ip{Q_t x^*, x^*}
\geq 0$ and $ \ip{Q_t x^*, y^*}= \ip{Q_t y^*, x^*}$. Let $H_t$
be the reproducing kernel Hilbert space associated with $Q_t$
and let $i_t : H_t \embed E$ be the inclusion mapping. Then,
$$i_t i_t\s = Q_t.$$ Since $ \ip{Q_t x^*,x^*} \leq \ip{ii^*
x^*,x^*}$ for all $x\s\in E\s,$ the operators $Q_t$ are
covariances of centred Gaussian measures $\mu_t$ on $E$; see,
e.g., \cite{GvN03}.  This estimate also implies that we have a
continuous inclusion $H_t\embed H$ and that the mapping
 \begin{align*}
 V_t: i^*x^* \mapsto i_t^*x^*, \quad x^* \in E^*,
 \end{align*}
is well defined and extends to a contraction from $H$ into $H_t.$ It is easy to
check that the adjoint operator
$V_t\s$ is the inclusion from $H_t$ into $H.$

Let us also note that for  $s \leq
t$ and $x\s\in E\s$ we have
 \begin{align*}
 \ip{Q_s x^*,x^*}  =\|i^*x\|^2 - \|\SB(s)i^*x^*\|^2
               \leq \|i^*x\|^2 - \|\SB(t)i^*x^*\|^2
                   = \ip{Q_t x^*,x^*}
 \end{align*}
by the contractivity of $\SB$.

In the next proposition we fix $t>0$ and $h \in H_t$ and denote by
$\phi_h^{\mu_t}:E\to \R$ the ($\mu_t$-essentially unique; see Section
\ref{sec:elliptic}) $\mu_t$-measurable linear extension of the function
$\phi_h^{\mu_t}(i_t g) := [g,h]_{H_t}.$

 \bpr \label{prop:Pformula}
For all $f = \varphi(\phi_{h_1},\ldots,
\phi_{h_n})\ot h \in \F C_{\rm b}(E)\otimes \HH$, where $\HH$ is some Hilbert
space, the following identity holds for $\mu$-almost all $x\in E$:
 \begin{align*}
 (P(t) \ot I) f (x) = \int_E \varphi(\phi_{\SB(t)h_1}(x) + \phi_{V_t h_1}^{\mu_t}(y)
                    ,\ldots, \phi_{\SB(t)h_n}(x) + \phi_{V_t h_n}^{\mu_t}(y))h
           \,d\mu_t(y).
 \end{align*}
 \epr

 \bpf
Defining $\psi:E\times \R^n\to \HH$ by
$$\psi(x,\xi) := \varphi(\phi_{\SB(t)h_1}(x) + \xi_1, \ldots,
\phi_{\SB(t)h_n}(x) + \xi_n)h,$$ we have
 \begin{align*}
 \int_E & \varphi(\phi_{\SB(t)h_1}(x) + \phi_{V_t h_1}^{\mu_t}(y)
                    , \ldots, \phi_{\SB(t)h_n}(x) + \phi_{V_t h_n}^{\mu_t}(y))h
           \,d\mu_t(y)
  \\& =  \int_E \psi\big(x, (\phi_{V_t h_1}^{\mu_t}(y),
                       \ldots, \phi_{V_t h_n}^{\mu_t}(y))\big)  \,d\mu_t(y)
  \\& =  \int_{\R^n} \psi(x, \xi)  \,d\ga_t(\xi),
 \end{align*}
where $\ga_t$ is the centred Gaussian measure on $\R^n$ whose
covariance matrix equals $\big( [V_t h_i, V_t h_j] \big)_{i,j=1}^n.$

On the other hand, writing $R(t) = \sqrt{I-\SB^*(t)\SB(t)}$, by
Mehler's formula \eqref{eq:Mehler} we have
 \begin{align*}
 (P(t) \ot I)f(x) &= \int_E \varphi( \phi_{\SB(t)h_1}(x) +
  \phi_{R(t) h_1}(y), \ldots,
  \phi_{\SB(t)h_n}(x) + \phi_{R(t)h_n}(y))h\,d\mu(y)
 \\ & =  \int_E \psi\big(x,(\phi_{R(t) h_1}(y), \ldots,
  \phi_{R(t)h_n}(y))\big)\,d\mu(y)
\\ & =  \int_{\R^n} \psi(x,\xi)\,d\tilde\ga_t(\xi),
 \end{align*}
where $\tilde\ga_t$ is the centred Gaussian measure on $\R^n$ whose
covariance matrix equals $\big( [R(t) h_i, R(t) h_j]
\big)_{i,j=1}^n.$

The result follows from the observation that
 \begin{align*}
 [V_t h_i, V_t h_j]  = [h_i, h_j] - [\SB(t)h_i,\SB(t)h_j]
   = [R(t) h_i, R(t) h_j].
 \end{align*}
 \epf

 \blem \label{lem:SBVinHt}
For all $u \in \H$ and $t>0$ we have $\uSB^*(t)  u\in \Dom(V\s)$,
$V^* \uSB^*(t)  u \in H_t$, and $$\|V^* \uSB^*(t) u\|_{H_t} \lesssim
\frac{1}{\sqrt{t}}\|u\|.$$
 \elem

 \bpf
First we observe that $\SB(s)$ maps $H$ into $\Dom(\AB)
\subseteq \Dom(V)$ for $s>0.$ For $t >0$ we claim that $$J_t:
V_t h \mapsto V \SB(\cdot)h$$ extends to a bounded operator
from $H_t$ into $L^2(0,t;\H)$ of norm $\le \frac1{\sqrt{2k}}$.

Indeed, by the coercivity of $B$ and the definition of $H_t,$ we
obtain for $h \in H,$
 \begin{align*}
 \int_0^t \|V \SB(s) h\|^2 \,ds
 &  \leq \frac{1}{k} \int_0^t [B V \SB(s) h,V \SB(s) h] \,ds
  \\ & =  -\frac{1}{2k} \int_0^t \frac{d}{ds} \|\SB(s) h\|^2 \,ds
  \\ & =  \tfrac{1}{2k} \big(\|h\|^2 - \|\SB(t) h\|^2 \big)
  \\ &  =  \tfrac{1}{2k} \|V_t h\|_{H_t}^2.
 \end{align*}
Recall that $V_t\s$ is the inclusion mapping $H_t\embed H$. Noting
that $\uSB^*(t)$ maps $\H$ into $\Dom(\uAB\s) \subseteq \Dom(V^*)$ and
using Lemma \ref{lem:SVcommute}, the adjoint mapping $J_t^* :
L^2(0,t;\H) \to H_t$ is given by
 \begin{align*}
 V_t^*  J_t^* f = \int_0^t V^* \uSB^*(s) f(s) \,ds, \quad
      f \in L^2(0,t;\H).
 \end{align*}
The resulting identity $V^* \uSB^*(t) u = \frac{1}{t} V_t^*
J_t^* \big(\uSB^*(t-\cdot)
 u \big)$ shows that $V^* \uSB^*(t) u$ can be identified with the element
$\frac{1}{t}J_t^* \big(\uSB^*(t-\cdot)u \big)$ of $H_t$ and we
obtain
 \begin{align*}
  \|   V^* \uSB^*(t) u \|_{H_t}
  &=  \frac{1}{t}\| J_t^* \big( \uSB^*(t-\cdot) u \big) \|_{H_t}
  \\ & \leq\frac{1}{t\sqrt{2k}}\|\uSB^*(t-\cdot)u
        \|_{L^2(0,t;\H)}
  \\ & \leq \frac{1}{\sqrt{2kt}} \sup_{s \geq 0}
       \| \uSB^*(s)\|_{\L(\H)} \| u \|.
 \end{align*}
 \epf

The following pointwise gradient bound is included for reasons of completeness.
We shall only need the special case corresponding to $r=2$, for which a simpler
proof can be given; see Remark \ref{rem:r2}.

 \begin{theorem}[Pointwise gradient bounds]
\label{thm:pointwisegrad} Let $1 < r <\infty.$ For $f \in \F C_{\rm
b}(E)$ and $t >0$ we have, for $\mu$-almost all $x\in E$,
 \begin{align*}
  \sqrt{t}\| \uD \PB(t) f(x)\| \lesssim (\PB(t)|f|^r(x))^{1/r}.
 \end{align*}
 \end{theorem}

 \bpf
For notational simplicity we take $f$ of the form $f =
\varphi(\phi_{h})$ with $\varphi \in C_{\rm b}(\R)$ and $h \in H.$ It is
immediate to check that the argument carries over to general
cylindrical functions in $\F C_{\rm b}(E)$.

By Lemma \ref{lem:SBVinHt} we have $\SB^*(t) V^* u \in H_t$
for $u \in \Dom(V^*)$ and therefore,  for all $h \in H,$
$$\phi_{\SB(t)h}(iV^*u) = [\SB(t)h, V^*u]  = [h, \SB^*(t)V^*u]=
 \phi_{V_t h}^{\mu_t}(i\SB^*(t)V^*u).
$$
By Proposition \ref{prop:Pformula} (with $\HH=\R)$
we find that for all $g \in H,$
 \begin{align*}
 \PB(t)f(x+ iV^*u)
 & = \int_E \varphi(\phi_{\SB(t)h}(x+iV^*u)
   +\phi_{V_t h}^{\mu_t}(y))   \,d\mu_t(y)
 \\ & = \int_E \varphi(\phi_{\SB(t)h}(x)
             + \phi_{V_t h}^{\mu_t}(y+i\SB^*(t)V^*u))
     \,d\mu_t(y).
 \end{align*}
Recalling that $\DH$ denotes the Malliavin derivative we have, for all $u
\in \Dom(V^*)$,
 \begin{align*}
 [\uD \PB(t)f(x),u]
 & = [\DH \PB(t)f(x), V^*u]
 \\& = \lim_{\eps \downarrow 0}\frac1\eps \Big( \PB(t)f(x+\eps iV^*u) - \PB(t)f(x) \Big)
 \\& = \lim_{\eps \downarrow 0}\frac1\eps
\int_E \varphi(\phi_{\SB(t)h}(x) +\phi_{V_t h}^{\mu_t}(y+\eps
i\SB^*(t)V^*u))
   \\&\qquad\qquad\qquad - \varphi(\phi_{\SB(t)h}(x) +\phi_{V_t h}^{\mu_t}(y))  \,d\mu_t(y).
 \end{align*}
Using Lemma \ref{lem:SBVinHt} and the Cameron-Martin formula \cite[Corollary 2.4.3]{Bo} we obtain
$$
 [\uD \PB(t)f(x),u]
 = \lim_{\eps \downarrow 0}\frac1\eps
       \int_E \big(E_{\eps \SB^*(t)V^*u}^{\mu_t}(y) - 1\big)
       \varphi(\phi_{\SB(t)h}(x) +\phi_{V_t h}^{\mu_t}(y))  \,d\mu_t(y),
$$
where $ E_{h}^{\mu_t}(y) = \exp\big(\phi^{\mu_t}_h(y) - \tfrac12\n
h\n_{H_t}^2\big).$ It is easy to see that for each $h\in H_t$ the
family $\big(\frac1\e(E_{\e h}^{\mu_t}-1)\big)_{0<\e<1}$ is uniformly
bounded in $L^2(E,\mu_t)$, and therefore uniformly integrable in
$L^1(E,\mu_t)$. Passage to the limit $\eps\downarrow 0$ now gives
 \begin{align*}
 [\uD \PB(t)f(x),u]
 =  \int_E \phi_{\SB^*(t)V^*u}^{\mu_t}(y)
       \varphi(\phi_{\SB(t)h}(x) +\phi_{V_t h}^{\mu_t}(y))  \,d\mu_t(y).
 \end{align*}
By H\"older's inequality with $\frac1q+\frac1r=1$, using the Gaussianity of
$\phi_{\SB^*(t)V^*u}^{\mu_t}$ on $(E,\mu_t)$ and the
Kahane-Khintchine inequality, Proposition \ref{prop:Pformula}, and
Lemma \ref{lem:SBVinHt} we find that
$$
\bal \ & |[\uD \PB(t)f(x), u]|
\\ & \qquad \le \Big(\int_E|\phi_{\SB^*(t)V^*u}^{\mu_t}(y)|^q\,d\mu_t(y)\Big)^{1/q}
    \Big(\int_E |\varphi(\phi_{\SB(t)h}(x)+\phi_{V_t h}^{\mu_t}(y))|^r\,d\mu_t(y)\Big)^{1/r}
\\ & \qquad \lesssim \Big(\int_E |\phi_{\SB^*(t)V^*u}^{\mu_t}(y)|^2\,d\mu_t(y)\Big)^{1/2} (\PB(t)|f|^r(x))^\frac{1}{r}
\\ & \qquad =  \n \SB^*(t)V^*u\n_{H_t} (\PB(t)|f|^r(x))^{\frac{1}{r}}
\\ & \qquad \lesssim \frac{1}{\sqrt t} \n u\n (\PB(t)|f|^r(x))^{\frac{1}{r}}.
  \eal
$$
The desired estimate is obtained by taking the supremum over all
$u\in \Dom(V^*)$ with $\|u\| \leq 1$.
 \epf

\begin{remark}\label{rem:r2}
There is a well-known elementary trick which
we learned from \cite[p.328]{Led00}) which can be
used to prove Theorem \ref{thm:pointwisegrad} for $r=2$. Using the
product rule from Lemma \ref{lem:Fcore}, the fact that $\|Bu \| \ge
k\|u\|$ for $u \in \overline{\Ran(V)}$, and the positivity of
$\PB(s)$, we obtain  \begin{align*}  \PB(t) f^2 - (\PB(t)f)^2   &  =
\int_0^t  \partial_s \big(\PB(s)\big( |\PB(t-s)f|^2           \big)
\big)\, d s  \\& = -\int_0^t \PB(s) \big( \LB(\PB(t-s)f)^2 - 2
\PB(t-s)f \cdot \LB \PB(t-s) f        \big) \, ds  \\& = 2\int_0^t
\PB(s) \big(      \| B\uD \PB(t-s)f\|^2  \big) \, ds  \\& \ge
2k \int_0^t \PB(s) \big(      \| \uD \PB(t-s)f\|^2  \big) \, ds.
 \end{align*} Next we estimate, for $\mu$-almost all $x\in E$,
 \begin{align*}
        M^2 \PB(r) ( \| \uD f\|^2) (x)
  &   \ge  \PB(r) ( \| \uS(r)\uD f\|^2)(x)
  \\ & \stackrel{(*)}{\ge} \| (\PB(r) \ot I) (  \uS(r)\uD
f)\|^2(x)  
\\ & = \| \uPB(r) \uD f(x)\|^2   \\ & = \|\uD \PB(r)f(x)\|^2,
   \end{align*} where $M:=\sup_{t\ge 0} \n \underline{S}(t)\n$
and $(*)$ follows from Proposition \ref{prop:Pformula} (with
$\HH = \H$) and Jensen's inequality. The case $r=2$ of Theorem
\ref{thm:pointwisegrad} follows from these two estimates.
\end{remark}

The next result is in some sense the dual
version of a maximal inequality. It could be compared with the dual
version of the non-commutative Doob inequality of \cite{Ju02}.

 \begin{proposition} \label{prop:tent}
Let $(M,\mu)$ be a $\sigma$-finite
measure space, $1\le p<\infty$, and let $(T(t))_{t > 0}$ be a family of positive operators on
$L^p:= L^p(M,\mu)$.  Suppose that the maximal
function $T_\star^*f := \sup_{t>0} |T^*(t)f|$ is measurable
and $L^q$-bounded, where $\frac1p+\frac1q = 1.$ Then, for all
$f_1, \ldots, f_n\in L^p$ and all $t_1, \ldots, t_n >0$,
 \begin{align*}
 \Big\| \sum_{k=1}^n T(t_k) |f_k| \Big\|_p
  \lesssim  \Big\| \sum_{k=1}^n  |f_k| \Big\|_p.
 \end{align*}
 \end{proposition}

 \begin{proof}
Taking the supremum over all $g=(g_k)_{k=1}^n \in L^q(\ell_n^\infty)$ of norm
one we obtain
 \begin{align*}
 \Big\| \sum_{k=1}^n T(t_k) |f_k| \Big\|_p
   &= \| \big( T(t_{(\cdot)}) |f_{(\cdot)}| \big) \|_{L^p(\ell_n^1)}
  \\& = \sup_{g}
     \int_E \sum_{k=1}^n  T(t_k) |f_k| \cdot  g_k \,d\mu
  \\& = \sup_{g}
     \int_E \sum_{k=1}^n  |f_k| \cdot T^*(t_k) g_k \,d\mu
  \\& \leq
      \| \big( |f_{(\cdot)}| \big) \|_{L^p(\ell_n^1)}
      \sup_{g} \| \big( T^*(t_{(\cdot)}) g_{(\cdot)} \big) \|_{L^q(\ell_n^\infty)}.
 \end{align*}
Using the positivity of $T^*$ on $L^q$ to obtain $ \sup_{1\le
k\le n} T_\star^* |g_k| \leq T_\star^* (\sup_{1\le k\le n} |g_k|)$ we
estimate
$$\bal
  \big\| \big( T^*(t_{(\cdot)}) g_{(\cdot)} \big) \big\|_{L^q(\ell_n^\infty)}
 & = \big\| \sup_{1\le k\le n} | T^*(t_k) g_k | \,\big\|_{L^q}
 \\ & \leq \big\| \sup_{1\le k\le n}  T_\star^* |g_k| \,\big\|_{L^q}
 \\ & \leq \big\|  T_\star^* (\sup_{1\le k\le n}  |g_k|)  \,\big\|_{L^q}
 \\ & \lesssim \big\|  \sup_{1\le k\le n}  |g_k|  \,\big\|_{L^q}
 \\ & = \| (g_k)  \|_{L^q(\ell_n^\infty)}.
\eal
$$
This completes the proof.
 \end{proof}

The previous two results are now combined to prove:

 \begin{proof}[Proof of Theorem \ref{thm:Rgradientestimates} (for $2< p<\infty$)]
Let $\frac{2}{p}+\frac1q = 1.$ Proposition \ref{prop:Cowling}
implies that the maximal function
$$\PB_\star^* f := \sup_{t>0} |\PB^*(t)f|$$ is bounded on
$L^{q}.$ Using Theorem \ref{thm:pointwisegrad} (for $r=2$) and Proposition
\ref{prop:tent} we obtain, for all $f_1,\dots, f_n \in \F
C_{\rm b}(E)$,
$$
\bal  \Big\| \Big(\sum_{k=1}^n \|\sqrt{t_k}\uD
\PB(t_k)f_k\|^2 \Big)^{1/2} \Big\|_p
 & \lesssim \Big\| \Big(\sum_{k=1}^n \PB(t_k)|f_k|^2 \Big)^{1/2} \Big\|_p
 \\ &  = \Big\|\sum_{k=1}^n \PB(t_k)|f_k|^2  \Big\|_{p/2}^{1/2}
 \\ & \lesssim \Big\|\sum_{k=1}^n |f_k|^2  \Big\|_{p/2}^{1/2}
 \\ & = \Big\|  \Big(\sum_{k=1}^n |f_k|^2\Big)^{1/2} \Big\|_p.
 \eal
$$
By an approximation argument this estimate extends to
arbitrary $f_1,\dots, f_n \in L^p$.
Now Proposition \ref{prop:R-bdd-SQF} implies the $R$-boundedness of $\{\sqrt{t}\uD \PB(t) :\ t >0\}.$

Taking Laplace transforms and using Proposition \ref{prop:Rboundedintegral},
it follows that $\D(\LB)\subseteq \D(\uD)$ and that the collection
$\{t\uD (I+t^2 \LB)^{-1}\ : t >0\}$ is $R$-bounded from $L^p$ into $\uL^p.$
As in the case $1<p\le 2$, Lemmas
 \ref{lem:Fcore} and \ref{lem:polyn-core} imply that $\D(\LB)$ is even a core for
 $\D(\uD)$.
\epf

\begin{proof}[Proof of the first part of Theorem  \ref{thm:square_functions} (for $2< p<\infty$)]
By Lemma \ref{lem:LBcalc}, $\LB$ has a bounded $H^\infty$-calculus of angle
$<\frac12 \pi$, and the result follows from
Proposition \ref{prop:sq-fc1}.
\epf

\subsection{Completion of the proof of Theorem \ref{thm:square_functions}}
It remains to prove, for $1<p<\infty$, the left-hand side inequality of Theorem \ref{thm:square_functions}. We adapt a standard duality argument
(see, e.g., \cite[Section 7, Step 8]{Au07}).

It is enough to prove the estimate for $f\in \Ran_p(\LB)$; for such $f$ we have
$f - P_{\Null_p(\LB)}f =f$. 
First let $f= \LB g$ with $g\in \D(\LB^2)$. Then by \cite[Lemma 9.13]{KuW},
$$ \lim_{t\to\infty} \PB(t)\LB g - \LB g = -\lim_{t\to\infty}\int_0^t \PB(s)\LB^2 g\,ds
=  -\lim_{t\to\infty}\int_0^t \psi(s\LB) \LB g\,\frac{ds}{s}
=-\LB g,$$
where $\psi(z) = ze^{-z}$. Hence,
$\lim_{t\to\infty} \PB(t)\LB g  = 0.$
By a density argument, this implies
$$  \lim_{t\to\infty} \PB(t)f = 0, \quad f\in \ov{\Ran(\LB)}.$$
Fix $f \in L^2\cap \ov{\Ran_p(\LB)}$ and $g \in L^2\cap L^q$
with $\frac1p +\frac1q =1.$ For any $t>0,$
 \begin{align*}
 -  \partial_t \int_E (\PB(t)f) g \, d\mu
     & =  \int_E \LB (\PB(t)f) g \, d\mu \\
     & =  \int_E \LB (\PB(\tfrac12 t)f) \PB\s(\tfrac12 t)g \, d\mu\\
     & =   \int_E [B \uD\PB(\tfrac12 t)f, \uD \PB\s(\tfrac12 t)g] \, d\mu.
 \end{align*}
Since $\int_E f \, d\mu = 0$  we
obtain, using Theorem \ref{thm:square_functions} applied to the adjoint semigroup $\PB\s$ (which is generated by $\LB\s = \uD\s B\s \uD$) in $L^q,$
 \begin{align*}
  \int_E  fg \, d\mu
      & =   \int_E fg \, d\mu - \int_E  f \, d\mu \int_E  g \, d\mu
  \\  & =  \lim_{\eps \downarrow 0} \int_E (\PB(\eps)f) g \, d\mu
         - \lim_{t \to \infty} \int_E (\PB(t)f) g \, d\mu
  \\  & = -  \int_0^ \infty \!\!  \partial_t \int_E (\PB(t)f) g \, d\mu \, dt
  \\  & =  \int_0^ \infty \!\!\!  \int_E [B \uD\PB(\tfrac12 t)f, \uD \PB\s(\tfrac12 t)g]
                   \, d\mu \, dt
  \\  & \leq \| B \|
     \Big\| \Big( \int_0^ \infty\!\! \|\sqrt{t}\uD\PB(\tfrac12 t)f\|^2
             \frac{dt}{t}\Big)^{1/2} \Big\|_p
     \Big\| \Big( \int_0^ \infty \!\!\|\sqrt{t}\uD\PB\s(\tfrac12 t)g\|^2
             \frac{dt}{t}\Big)^{1/2} \Big\|_q
  \\  & \lesssim
     \Big\| \Big( \int_0^ \infty\! \|\sqrt{t}\uD\PB(t)f\|^2
             \frac{dt}{t}\Big)^{1/2} \Big\|_p \| g\|_q.
 \end{align*}
This implies that
 \begin{align*}
  \|f\|_p \lesssim  \Big\| \Big( \int_0^ \infty \|\sqrt{t}\uD\PB(t)f\|^2
             \frac{dt}{t}\Big)^{1/2} \Big\|_p.
 \end{align*}
So far we have assumed that $f\in L^2\cap   \ov{\Ran_p(L)}$.
The extension to general $f\in \ov{\Ran_p(L)}$ follows by a density
argument (using the first part of the theorem to see that the right hand side 
can be approximated as well).

\section{The operators $\uD$ and $\uD^*B$}\label{sec:DandDsB}

In this section we study some $L^p$-properties of the
operators $\uD$ and $\uD^*B$ and provide a rigorous
interpretation of the identities $\LB = \uD\s B \uD$ and $\uLB
= \uD \uD\s B$ in $L^p$ and $\ov{\Ran_p(\uD)}$. From these
operators we build operator matrices which will play an
important role in the proofs of Theorems \ref{thm:Kato},
\ref{thm:Hodge-main}, and \ref{thm:R-bisectorial-main}.

Throughout this section we fix $1<p<\infty$. The operator
$\uD^* B$ is closed and densely defined as an operator from
$\uL^p$ to $L^p$ with domain $$\D(\uD^* B) = \{F\in \uL^p: \ B
F \in \D(\uD\s)\}.$$ Moreover, since $\|BF\|_p \eqsim \|F\|_p$
for $F \in \ov{\Ran_p(\uD)}$, $$\uD^*B = (B^*\uD)^*,$$ where
$B^*\uD$ is interpreted as a operator from $L^q$ to
$\underline{L}^q$, $\frac1p+\frac1q=1$.

 \newcommand{\cC}{\mathcal{C}}

For the next result we recall that $\cC := \F C_{\rm b}^\infty(E;\Dom(\AB))$
is a $\PB$-invariant core for $\D(\LB)$. We
set $\cC^* := \F C_{\rm b}^\infty(E; \Dom(\AB^*));$ this is a
$\PB^*$-invariant core for $\D(\LB^*)$.

 \begin{proposition} \label{prop:Lpformula}
In $L^p$ we have $\LB = (\uD^* B) \uD.$ More precisely, $f \in
\D(\LB)$ if and only if $f \in \D(\uD)$ and $\uD f \in \D(\uD^* B),$
in which case we have $\LB f = (\uD^* B) \uD f.$
 \end{proposition}

 \bpf
First note that for all $f,g \in \cC$ we have $\ip{\LB f, g} =
\ip{\uD f, B^*\uD g}.$ Since $\cC$ is a core for $\D(\LB)$,
and $\D(\LB)$ is core for $\D(\uD)$ by the first part of
Theorem \ref{thm:Rgradientestimates}, this identity extends to
all $f\in \D(\LB)$ and $g\in\D(\uD)$. This implies that $\uD f
\in \D((B^*\uD)^*)$ and $(B^*\uD)^* \uD f = \LB f.$ Since
$(B^*\uD)^* = \uD^* B,$ we find that $\LB \subseteq
(\uD^*B)\uD.$

To prove the other inclusion we take $f \in \D(\uD)$ such that
$\uD f \in \D(\uD^*B).$ We have $\ip{f, \LB^* g} = \ip{\uD f,
B^*\uD g} = \ip{(\uD^* B) \uD f, g}$ for all $g\in \cC^*,$
where the second identity follows from $\uD f \in \D(\uD^*B) =
\D((B^*\uD)^*)$.  Since $\cC^*$ is a core for $\Dom_q(\LB\s)$
this implies that $f\in \D(\LB)$ and $\LB f = (\uD^*B)\uD f.$
 \epf

We shall be interested in the restriction $\uD^*
B|_{\ov{\Ran_p(\uD)}}$ of $\uD^* B$ to $\ov{\Ran_p(\uD)}$. As
its domain we take $$\D(\uD^* B|_{\ov{\Ran_p(\uD)}}) := \{F\in
\ov{\Ran_p(\uD)}: \ BF\in \D(\uD\s)\}
 = \D(\uD^* B)\cap \ov{\Ran_p(\uD)}.$$
In the middle expression, as before we consider $\uD\s$ as a densely defined operator from
$\uL^p$ to $L^p$.

 \begin{corollary} \label{cor:DsBdenselydef}
The restriction $\uD^* B|_{\ov{\Ran_p(\uD)}}$ is closed
and densely defined.
 \end{corollary}

 \bpf
Let $f\in \D(\uD).$ By the first part of Theorem
\ref{thm:Rgradientestimates} there exist functions $f_n \in \D(\LB)$
such that $f_n \to f$ in $\D(\uD).$ Proposition
\ref{prop:Lpformula} implies that $\uD f_n \in
\D(\uD^*B|_{\ov{\Ran_p(\uD)}}).$ This shows that $\uD^*
B|_{\ov{\Ran_p(\uD)}}$ is densely defined on
$\ov{\Ran_p(\uD)}.$ Closedness is clear.
 \epf

\begin{proposition}\label{prop:uLpformula2}
The domain $\D(\uLB)$ is a core for $\D(\uD^*B|_{\ov{\Ran_p(\uD)}})$. Moreover, for all $t>0$
the operators $ (I + t^2 \LB)^{-1} \uD^*B|_{\ov{\Ran_p(\uD)}}$ and
$\PB(t) \uD^*B|_{\ov{\Ran_p(\uD)}}$ (initially
defined on $\D(\uD\s B|_{\ov{\Ran_p(\uD)}})$) extend uniquely to  bounded
operators from $\ov{\Ran_p(\uD)}$ to $L^p$, and for all
$F\in \ov{\Ran_p(\uD)}$ we have
 $$ (I + t^2 \LB)^{-1} \uD^*B F= \uD^* B (I + t^2 \uLB)^{-1}F$$
and $$  \PB(t) \uD^*B F = \uD^* B \uPB(t)F.$$
\end{proposition}

\bpf We split the proof into four steps.

{\em Step 1} - By Proposition \ref{prop:Lpformula}, for all $f \in \D(\LB)$ we have $f \in \D(\uD)$ and $\uD f\in
\D(\uD^* B|_{\ov{\Ran_p(\uD)}})$, and for all $t > 0$
we have $$\PB(t) (\uD^*B) \uD f = \PB(t)\LB f
 = \LB \PB(t) f =(\uD^* B) \uD\PB(t) f= \uD^* B \uPB(t) \uD f.$$
By taking Laplace transforms and  using the closedness of $\uD^*B$,
 this gives $(I +t^2 \uLB)^{-1} \uD f \in \D(\uD^* B |_{\ov{\Ran(\uD)}} )$ and
\beq\label{eq:sec-1}
(I + t^2 \LB)^{-1} (\uD^*B) \uD f= \uD^* B (I + t^2 \uLB)^{-1} \uD f.
\eeq

{\em Step 2} -
By Theorem \ref{thm:Rgradientestimates}, for all $t>0$ the
operator $T(t) :=  B^* \uD (I + t^2 \LB^*)^{-1}$ is bounded
from $L^q$ into $\uL^q,$ $\frac1p+\frac1q=1$.
For all $F\in \D(\uD^* B|_{\ov{\Ran_p(\uD)}})$ and $g \in L^q$ we have
\beq\label{eq:fir-2}
\ip{F,T(t)g}
   = \ip{F, B^* \uD (I + t^2 \LB^*)^{-1}g}
   = \ip{(I + t^2 \LB)^{-1} \uD^* B F,g}.
\eeq
Now let $F \in \ov{\Ran_p(\uD)}$ be arbitrary and take a
sequence $(F_n)_{n\ge 1} \subseteq \D(\uD^* B|_{\ov{\Ran_p(\uD)}})$ converging to $F$ in
$\ov{\Ran_p(\uD)}.$
By Proposition \ref{prop:Lpformula} and the fact that
$\D(L)$ is a core for $\D(\uD)$ we may take the $F_n$ of the form
$\uD f_n$ with $f_n\in \D(L)$.
Then $(I + t^2\uLB)^{-1} F_n \to (I + t^2\uLB)^{-1} F$, and from \eqref{eq:sec-1}
we obtain
$$\uD^*B (I + t^2\uLB)^{-1} F_n = (I + t^2\LB)^{-1} \uD^*B
F_n
  = T^*(t) F_n \to T^*(t) F.$$ The closedness of $\uD^*B$
implies that $(I + t^2\uLB)^{-1} F \in \D(\uD^*B|_{\ov{\Ran_p(\uD)}})$. This proves
 the domain inclusion $\D(\uLB) \subseteq
\D(\uD^*B|_{\ov{\Ran_p(\uD)}})$, along with the identity
$$\uD^*B (I + t^2\uLB)^{-1} F = T^*(t) F, \quad F \in
\ov{\Ran_p(\uD)}.$$ Note that for $F\in \D(\uD\s B)$, from
\eqref{eq:fir-2} we also obtain \beq\label{eq:sec-2} \uD^*B (I
+ t^2\uLB)^{-1} F = T^*(t) F =(I + t^2 \LB)^{-1} \uD^* B
F.\eeq

{\em Step 3} -
By Step 2 the
operator  $\uD^* B (I + t^2 \uLB)^{-1}$ is bounded from $\ov{\Ran_p(\uD)}$ to
$L^p$.
Therefore, by \eqref{eq:sec-1},
the operator $(I + t^2 \LB)^{-1}\uD^*B$ (initially defined on the dense domain
$\D(\uD\s B|_{\ov{\Ran_p(\uD)}})$) uniquely extends to a bounded operator
from $\ov{\Ran_p(\uD)}$ to $L^p$, and for this extension we obtain
the identity
$$ (I + t^2 \LB)^{-1} \uD^*B = \uD^* B (I + t^2 \uLB)^{-1}.$$
On $\D(\uD\s B|_{\ov{\Ran_p(\uD)}})$, the identity  $\uD^* B
\uPB(t)= \PB(t) \uD^*B$ follows from \eqref{eq:sec-2} by real
Laplace inversion (cf. the proof of Lemma
\ref{lem:SVcommute}). The existence of a unique bounded
extension of $\PB(t) \uD^*B$ is proved in the same way as
before.

{\em Step 4} -- It remains to prove that $\D(\uL)$ is a core
for $\D(\uD\s B|_{\ov{\Ran_p(\uD)}})$. Take $F\in \D(\uD\s
B|_{\ov{\Ran_p(\uD)}})$. Then $\lim_{t\to 0} (I+t^2 \uL)^{-1}F
= F$ in $\ov{\Ran_p(\uD)}$ and, by \eqref{eq:sec-2}
$\lim_{t\to 0} \uD\s B(I+t^2 \uLB)^{-1}F = \lim_{t\to 0}
(I+t^2 \LB)^{-1} \uD\s BF =  \uD\s BF$ in $L^p$. This gives
the result. \epf

\begin{proposition}\label{prop:uLpformula1}
For all $F \in\D(\uLB)$ we have $F\in \D(\uD\s B)$,
$\uD^*B F \in \D(\uD)$, and $\uD(\uD^*B)F =
\uLB F.$
 \end{proposition}
\bpf
Since $\D(\LB)$ is a core for $\D(\uD)$, the set
$\mathscr{P} := \{\uD (I+\LB)^{-1}g: \ g\in \D(\uD)\}$
 is a $\uPB$-invariant dense subspace of
$\ov{\Ran_p(\uD)}$. To see that $\mathscr{P}$ is contained in $\D(\uLB)$,
note that if $g\in \D(\uD)$, then $f:= (I+\LB)^{-1}g\in \D(\LB)$
and
$\uD f = \uD (1+\LB)^{-1}g = (1+\uLB)^{-1}\uD g \in \D(\uLB)$ as claimed.
It follows that $\mathscr{P}$ is a core for
$\D(\uLB)$, and hence a core for $\D(\uD\s B|_{\ov{\Ran_p(\uD)}})$ by Proposition
\ref{prop:uLpformula2}.
Moreover,
$(1+\uLB) \uD f = \uD g = \uD (I+\LB)f$, and therefore
$\uLB \uD f = \uD \LB f$.

For $F\in \mathscr{P}$, say $F =\uD f$ with $f=(I+\LB)^{-1}g$ for some
$g\in \D(\uD)$,
we then have 
$$ \uLB F = \uLB \uD f =  \uD \LB f
= \uD ((\uD\s B) \uD) f =  (\uD (\uD\s B)) \uD f = \uD (\uD\s
B) F.$$ To see that this above identity extends to arbitrary
$F\in \D(\uLB)$, let $F_n\to F$ in $\D(\uLB)$ with all $F_n$
in $\mathcal{P}$. It follows from Proposition
\ref{prop:uLpformula2} that $F_n\to F$ in $\D(\uD\s B)$. In
particular, $\uD\s B F_n\to \uD\s B F$ in $L^p$. Since
$\uD(\uD\s B) F_n = \uLB F_n \to \uLB F$ in
$\ov{\Ran_p(\uD)}$, the closedness of $\uD$ then implies that
$\uD\s B F \in \D(\uD)$ and $ \uD(\uD\s B) F = \uLB F$. \epf

In the remainder of this section we consider $\uD^*B$ as a
closed and densely defined operator from  $\ov{\Ran_p(\uD)}$
to $L^p$ and write $\uD^*B$ instead of using the more precise notation
$\uD^*B|_{\ov{\Ran_p(\uD)}}$.

For the proof of Proposition \ref{prop:rankeridentities}
we need the first part of Theorem
\ref{thm:R-bisectorial-main}. Its proof uses the Hodge-Dirac
formalism, introduced by Axelsson, Keith, and M$^{\rm
c}$Intosh \cite{AKM} in their study of the Kato square root
problem. It was by using this formalism that the main results
of this paper suggested themselves naturally.

On the Hilbertian direct sum $H \oplus \ov{\Ran(V)}$ we
consider the closed and densely defined operator
\beq\label{eq:defT}\TB:= \bma 0 & V\s B \\  V & 0\ema.\eeq By
\cite[Theorem 8.3]{AMN} $\TB$ is bisectorial on $H \oplus
\ov{\Ran(V)}.$

On the direct sum $L^p\oplus \ov{\Ran_p(\uD)}$
we introduce the closed and densely defined
operator $\PiB$ by
$$
\PiB :=  \bma 0 & \uD \s B \\ \uD  & 0 \ema.
$$

\begin{proof}[Proof of Theorem \ref{thm:R-bisectorial-main}, first part]
By Theorems \ref{thm:sector} and \ref{thm:uLsemigroup}, $\LB $
and $\uLB $ are sectorial on $L^p$ and
$\overline{\Ran_p(\uD)}$, respectively. From this it is easy
to see that on $L^p\oplus \ov{\Ran_p(\uD)}$ we have
$i\R\setminus\{0\} \subseteq \rho(\PiB)$ and

$$ (I-it\PiB)^{-1} = \bma (1+t^2 \LB )^{-1} & it (I+t^2 \LB )^{-1}\uD \s {B}\\
  it \uD  (I+t^2 \LB )^{-1} & (I+t^2  \uLB )^{-1}  \ema, \quad
t\in \R\setminus\{0\};
$$
the rigorous interpretation of this identity is provided by
the above propositions. Note that the off-diagonal entries are
well defined and bounded by Theorem
\ref{thm:Rgradientestimates} and Proposition
\ref{prop:uLpformula2}; the proof of the latter result also
shows that $(I+t^2 \LB )^{-1}\uD \s {B}$ is the adjoint of
$B^* \uD (I+t^2 \LB^* )^{-1}$.

We check the $R$-boundedness of the entries of the right-hand side matrix for
$t \in \R \setminus \{0\}.$ For the upper left and the lower
right entry this follows from the $R$-sectoriality of $\LB$
and $\uLB$ on $L^p$ and $\overline{\Ran_p(\uD)}$ respectively.
Theorem \ref{thm:Rgradientestimates} ensures the
$R$-boundedness of the lower left entry, and the
$R$-bounded\-ness of the upper right entry follows from
Proposition \ref{prop:Bconvex} (applied with $B$ and $\LB$ replaced by $B\s$ and
$\LB\s$).
 \end{proof}

As a consequence of the bisectoriality of $\Pi$, the operator
$\PiB^2$ is sectorial. Moreover,
 \begin{align*} 
 \PiB^2 = \bma (\uD^* B) \uD  & 0 \\ 0 & \uD(\uD^*B) \ema
   = \bma \LB  & 0 \\ 0 & \uLB  \ema.
 \end{align*}
To justify the latter identity, we appeal to Propositions
\ref{prop:Lpformula} and \ref{prop:uLpformula1} to obtain the
inclusion $\bma \LB & 0
\\ 0 & \uLB  \ema \subseteq \PiB^2.$ Since both operators
are sectorial of angle $< \frac12\pi$, they are in fact equal.

 \begin{proposition} \label{prop:rankeridentities}
On $L^p$ and $\ov{\Ran_p(\uD)}$ the following identities hold:
  \begin{align*}
 \ov{\Ran_p(\LB)} &=    \ov{\Ran_p(\uD^* B)},&
    \Null_p(\LB)&= \Null_p(\uD),\\
 \ov{\Ran_p(\uLB)}  &=  \ov{\Ran_p(\uD)},&
  \Null_p(\uLB) &= \Null_p(\uD^*B) =      \{0\}.
  \end{align*}
Moreover,
$L^p = \ov{\Ran_p(\uD^* B)} \oplus \Null_p(\uD).$
 \end{proposition}

We recall that $\uD^* B$ is interpreted as a densely defined
closed operator from $\ov{\Ran_p(\uD)}$ to $L^p$. In the final
section we will show that under the assumptions of Theorem
\ref{thm:Kato} we have $\ov{\Ran_p(\uD^* B)} =
\ov{\Ran_p(\uD^*)}$ and that in this situation the space
$\ov{\Ran_p(\uD^* B)}$ does not change if we consider $\uD\s
B$ as an unbounded operator from $\uL^p$ to $L^p$.

 \bpf
The bisectoriality of $\PiB$ on $L^p\oplus \ov{\Ran_p(\uD)}$
implies that$$\ov{\Ran_p(\PiB^2)} = \ov{\Ran_p(\PiB)} \quad
\text{ and }\quad \Null_p(\PiB^2)= \Null_p(\PiB).$$ The result
follows from this by considering both coordinates separately.
The fact that $\Null_p(\uD^* B) = \{0\}$ follows from the
bisectorial decomposition $L^p \oplus \ov{\Ran_p(\uD)} =
\ov{\Ran_p(\PiB)} \oplus \Null_p(\PiB)$ and considering the
second coordinate. The final identity follows by inspecting the first coordinate
of the same decomposition.
 \epf

\section{Proof of Theorem \ref{thm:Kato}}\label{sec:Kato}

The main effort in this section is directed towards proving the following comprehensive version
of Theorem \ref{thm:Kato}.

\begin{theorem}\label{thm:Kato-extended}
Assume {\rm(A1)}, {\rm (A2)}, {\rm (A3)}, and let
$1<p<\infty$.

 \begin{enumerate}[\rm(a)]
 \item The following assertions are equivalent: \ben[(1)]
\item[\rm(a1)] $\D(\sqrt{\LB}) \subseteq \D(\uD)$
    with $\n \uD f\n_{p}\lesssim \n \sqrt{\LB}f\n_{p}$;
\item[\rm(a2)] $\uLB$ satisfies a  square function estimate on
$\overline{\Ran_p(\uD)}$:
$$  \|F\|_p  \lesssim  \Big\|  \Big( \int_0^\infty \n t \uLB\, \uPB(t) F\n^2 \,\frac{dt}{t}
  \Big)^{1/2} \Big\|_p;  $$
\item[\rm(a3)] $\Dom(\sqrt{\AB}) \subseteq \Dom(V)$
  with $\n Vh\n \lesssim \n \sqrt{\AB}h\n$;
\item[\rm(a4)] $\uAB$ satisfies a square function estimate on $\overline{\Ran(V)}$:
$$
 \|u\| \lesssim  \Big( \int_0^\infty \| t \uAB\, \uSB(t) u\|^2 \,\frac{dt}{t}
  \Big)^{1/2}.
$$

\een
\vskip3pt

\item[\rm(b)]
The same result holds with `$\lesssim$' and `$\subseteq$'
replaced by `$\gtrsim$' and `$\supseteq$'. \vskip3pt

\item[\rm(c)]
The following assertions are equivalent:
\ben
\item[\rm(c1)] $\D(\sqrt{\LB}) = \D(\uD)$
  with $\n \uD f\n_{p}\eqsim \n \sqrt{\LB}f\n_{p}$;
\item[\rm(c2)] $\uLB$ admits a bounded $H^\infty$-functional calculus on
$\overline{\Ran_p(\uD)}$;
\item[\rm(c3)] $\Dom(\sqrt{\AB}) = \Dom(V)$
  with $\n Vh\n \eqsim \n \sqrt{\AB}h\n$;
\item[\rm(c4)] $\uAB$ admits a bounded $H^\infty$-functional calculus on
$\overline{\Ran(V)}$.
\een
\vskip3pt
  \end{enumerate}
\end{theorem}

The plan of the proof is as follows. First we consider $({\rm
a})$. The equivalence of (a3) and (a4) will be proved in Lemma
\ref{lem:HilbertspaceKato}, while the implications (a1)
$\Rightarrow $ (a3) and  (a2) $\Rightarrow$ (a4) follow by
considering functions of the form $f=\phi_h$ and $F=\one \ot
u$  respectively, and using the equivalence of $L^p$-norms on
the first Wiener-It\^o chaos. In Proposition
\ref{prop:rankeridentities} we have shown that $\uLB$ is
injective on $\ov{\Ran_p(D)}$, and then Proposition
\ref{prop:tensorcalculus} asserts that (a4) implies (a2), so
that it remains to show that (a2) implies (a1).

Next we turn to part (b). The equivalence of (b3) and
(b4) follows from Lemma \ref{lem:HilbertspaceKato}, and  the
implications (b1) $\Rightarrow$ (b3) and (b2) $\Rightarrow$ (b4) follow
as in part (a). Proposition \ref{prop:tensorcalculus}
asserts that (b4) implies (b2), so that it suffices to show that
(b4) implies (b1).

Finally, part (c) follows by putting together the estimates
obtained in (a) and (b) and appealing to Proposition
\ref{prop:Hinfty-sfe}.

\medskip

The next lemma is a variation of Theorem 10.1 in \cite{AMN}.
We are grateful to Alan M$^{\rm c}$Intosh for showing us the
argument below. Keeping in mind
 that
$B$ satisfies (A3) if and only if $B^*$ satisfies (A3), we
write $\uA_{*} := VV^*B^*$ and we denote the semigroup
generated by $-\uA_{*}$ by $\uS_{*}.$

\blem \label{lem:HilbertspaceKato} Assume {\rm (A2)} and {\rm
(A3)}. For $h\in \Dom(A)$ we have
 \begin{align} \label{eq:always}
 \Big( \int_0^\infty \| t \uAB\, \uSB(t) Vh\|^2
\,\frac{dt}{t}
  \Big)^{1/2} \eqsim \| \sqrt{\AB} h \|.
 \end{align}
As a first consequence, the following assertions are
equivalent: \ben
\item[\rm(1)] $\Dom(\sqrt{\AB}) \subseteq \Dom(V)$ with
 $\|\sqrt{\AB}h \| \gtrsim \|Vh\|$,\  $h \in \Dom(\sqrt{\AB})$;
\item[\rm(2)] $\uAB$ satisfies a  square function estimate on
$\ov{\Ran(V)}$:
  $$  \Big( \int_0^\infty \| t \uAB\, \uSB(t) u\|^2 \,\frac{dt}{t}
  \Big)^{1/2}   \gtrsim \|u\|;$$
\item[\rm(3)]  $\Dom(\sqrt{\AB^*}) \supseteq \Dom(V)$ with
 $\|\sqrt{\AB^*} h \| \lesssim \|Vh\|$, \ $h \in \Dom(V)$;
\item[\rm(4)] $\uA_{*}$ satisfies a square function estimate on
$\ov{\Ran(V)}$:
 $$\Big( \int_0^\infty \| t \uA_{*} \uS_{*}(t) u\|^2 \,\frac{dt}{t}
  \Big)^{1/2} \lesssim \|u\|.$$
\een As a second consequence, the following assertions are
equivalent:

\ben
\item[\rm(1$'$)] $\Dom(\sqrt{\AB}) = \Dom(V)$ with
equivalence of norms
 $\|\sqrt{\AB} h \| \eqsim \|Vh\|$;
\item[\rm(2$'$)]
$\uAB$ admits a bounded $H^\infty$-functional calculus on $\ov{\Ran(V)}$;
\item[\rm(3$'$)]  $\Dom(\sqrt{\AB\s}) = \Dom(V)$ with
 $\|\sqrt{\AB\s} h \| \eqsim \|Vh\|$;
\item[\rm(4$'$)]  $\uAB_*$ admits a bounded $H^\infty$-functional calculus on
$\ov{\Ran(V)}$.
\een
 \elem

 \bpf

To prove \eqref{eq:always}, let  $\om \in (\om(T),\frac12\pi),$ where
$\TB$ is defined by \eqref{eq:defT}. By \cite[Proposition 8.1]{AMN}
we have for all $\psi \in H_0^\infty(\S_\om)$ and $\tilde\psi \in
H_0^\infty(\S_{2\om}^+),$
 \begin{align*} 
% \label{eq:sfsecbisec}
  \int_0^\infty \| \psi(tT) u \|^2\, \frac{dt}{t}
   \eqsim
  \int_0^\infty \| \tilde\psi(tT^2) u \|^2\, \frac{dt}{t},
   \quad u \in \ov{\Ran(T)}.
 \end{align*}
Using this, the fact that $A$ has a bounded
$H^\infty$-calculus, and the fact that $\varphi := \sgn \cdot
\psi \in H_0^\infty(\S_\om),$ for $h \in \Dom(A)$ we obtain
 \begin{align*}
  \|  \sqrt{A} h\|^2
  &  \eqsim
  \int_0^\infty \| \tilde\psi(tA) \sqrt{A} h  \|^2 \frac{dt}{t}
  \\& =
  \int_0^\infty \Big\| \tilde\psi(tT^2) \sqrt{T^2}
            \bma h  \\  0\ema  \Big\|^2 \frac{dt}{t}
  \\& \eqsim
  \int_0^\infty \Big\| \psi(tT) \sqrt{T^2}
            \bma h  \\  0\ema  \Big\|^2 \frac{dt}{t}
  \\& =
  \int_0^\infty \Big\| \varphi(tT) T
            \bma h  \\  0\ema  \Big\|^2 \frac{dt}{t}
   \\& =
  \int_0^\infty \Big\| \varphi(tT)
            \bma 0  \\  Vh \ema  \Big\|^2 \frac{dt}{t}
  \\& \eqsim
  \int_0^\infty \Big\| \tilde\psi(tT^2)
            \bma 0  \\  Vh \ema  \Big\|^2 \frac{dt}{t}
  \\& \eqsim
  \int_0^\infty \| \tilde\psi(t\uA) Vh  \|^2 \frac{dt}{t}.
  \end{align*}
The equivalence of (1) and (2) follows immediately by taking $
\tilde\psi(z) = ze^{-z}.$ Replacing $B$ by $B^*$ we obtain the
equivalence of (3) and (4). Finally, the equivalence of $(1)$
and $(3)$ is a well-known consequence of the duality theory of
forms \cite{Ka,Li} (see also \cite[Theorem 10.1]{AMN}).

The equivalence of the primed statements follows in the same
way (or can alternatively be deduced from the equivalence of
the un-primed statements).
 \epf

\begin{proof}[Proof of Theorem \ref{thm:Kato-extended}]
Fix $1< p<\infty$ and let $\frac{1}{p}+\frac{1}{q} = 1$.

{\em {Part {\rm (a)}}}: It remains to prove that (a2) implies
(a1). Perhaps the shortest proof of this implication is based
on a lower bound for the square function associated with the
semigroup $\uQB$ generated by $-\sqrt{\uLB}$. Alternatively,
one could adapt the argument in Lemma
\ref{lem:HilbertspaceKato} to the $L^p$-setting.

Consider the functions $\varphi(z) = ze^{-z}$ and $\psi(z) =
\sqrt{z}e^{-\sqrt{z}}.$ These functions belong to
$H_0^\infty(\S_\th^+)$ for $\th < \frac12 \pi$. Substituting $t=s^2$ we obtain, from Proposition \ref{prop:Hinfty-sfe},
$$ \Big\|  \Big( \int_0^\infty \| \psi(t\uLB) F\|^2 \,\frac{dt}{t}
  \Big)^{1/2} \Big\|_p =
  \sqrt{2}\Big\|  \Big( \int_0^\infty \| s\sqrt{\uLB} \uQB(s) F\|^2 \,\frac{ds}{s}
  \Big)^{1/2} \Big\|_p.$$
Using (a2) and the first part of Proposition
\ref{prop:Hinfty-sfe}, the identity of Theorem
\ref{thm:uLsemigroup} (which extends to the semigroup $\QB$
generated by $-\sqrt{\LB}$), and Lemma \ref{lem:HdominatesG}
and Theorem \ref{thm:square_functions}, for all $f\in \D(\LB)$
we obtain
$$
\bal
  \| \uD f \|_p
 &  \lesssim  \Big\|  \Big( \int_0^\infty \| t{\uLB}\, \uPB(t) \uD f\|^2 \,\frac{dt}{t}
  \Big)^{1/2} \Big\|_p
 \\ & \eqsim  \Big\|  \Big( \int_0^\infty \| s\sqrt{\uLB} \uQB(s) \uD f\|^2
 \,\frac{ds}{s}
  \Big)^{1/2} \Big\|_p
 \\&  = \| \mathcal{G} (\sqrt{\LB}f) \|_p
 \\ &  \leq  \| \HH (\sqrt{\LB}f) \|_p
 \\ & \lesssim  \|  \sqrt{\LB}f \|_p.
 \eal
$$
Since $\D(\LB)$ is a core for both $\D(\sqrt{\LB})$ and
$\D(\uD)$, the desired domain inclusion follows and the norm
estimate holds for all $f \in \D(\sqrt{\LB}).$

\medskip
{\em Part {\rm (b)}}: It remains to show that (b4) implies (b1).

By Lemma \ref{lem:HilbertspaceKato} (applied with the roles of $B$ and $B^*$
reversed),
(b4) implies the estimate
$$  \Big( \int_0^\infty \| t \uA_{B\s} \uS_{B\s}(t) u\|^2 \,\frac{dt}{t}
  \Big)^{1/2}   \gtrsim \|u\|, \quad u\in \ov{\Ran(V)}.$$
It follows from Part $({\rm a})$ (with $B$ replaced by $B^*$) that
$\Dom_q(\sqrt{\LB\s})\subseteq \Dom_q(\uD)$ and, for $f \in
\Dom_q(\sqrt{\LB^*})$,
$$ \|\uD f \|_q \lesssim \| \sqrt{\LB^*} f\|_q, \quad 1<q<\infty.$$

We will use next a standard duality argument to prove the estimate
$\| \sqrt{\LB} g\|_p \lesssim \| \uD g  \|_p$ where $\frac1p+\frac1q
=1.$ For $g \in \F C_{\rm b}^\infty(E;\Dom(\AB))$ we have
  \begin{align*}
  \|\sqrt{\LB} g\|_p
  & =   \sup_{\|\widetilde f\|_q \leq 1} | \lb  \sqrt{\LB} g,
       \widetilde f \rb|.
  \end{align*}
The sectoriality of $\sqrt{\LB^*}$ allows us to use the decomposition
 $\widetilde f = \widetilde f_0 + \widetilde
 f_1 \in \Null(\sqrt{\LB^*}) \oplus
    \ov{\Ran(\sqrt{\LB^*})} = L^q,$
    and since $\F C_{\rm b}^\infty(E;\Dom(\AB\s))$ is a core for
$\Dom_q(\sqrt{\LB^*}),$ it suffices to consider $\widetilde f$ of the
form $\widetilde f = \widetilde f_0 + \sqrt{\LB^*} f,$ with $f\in \F
C_{\rm b}^\infty(E;\Dom(\AB\s))$ and
  $\|\widetilde f\|_q \leq 1.$
Since $\widetilde f_0 \in \Null(\sqrt{\LB^*})$ and $\| \sqrt{\LB^*} f
\|_p \leq \|P_{\ov{\Ran(\sqrt{\LB^*})}}\|_p \|\widetilde f \|_p,$ we
obtain
\begin{align*}
  \|\sqrt{\LB} g\|_p
  & =   \sup_{\|\widetilde f_0 + \sqrt{\LB^*} f\|_q \leq 1}
    | \lb  \sqrt{\LB} g,  \widetilde f_0 + \sqrt{\LB^*} f \rb|
 \\ & \lesssim    \sup_{\| \sqrt{\LB^*} f\|_q \leq 1}
    | \lb  \sqrt{\LB} g, \sqrt{\LB^*} f \rb|
 \\ & =      \sup_{\|\sqrt{\LB^*} f\|_q\leq 1} | \lb  \LB g, f\rb|
 \\ & \lesssim \sup_{\|\uD f\|_q \leq 1} | \lb \LB g, f\rb|
 \\ & =    \sup_{\|\uD f\|_q \leq 1} | \lb  B\uD g,  \uD f\rb|
 \\ & \leq  \sup_{\|\uD f\|_q \leq 1} \| B \| \; \|\uD g\|_p \|\uD f  \|_q
 \\ & =   \| B \|\, \|\uD g  \|_p .
 \end{align*}
Since $\F C_{\rm b}^\infty(E;\Dom(\AB))$ is a core for $\D(\LB)$, the
result of Step 2 follows.

{\em Part {\rm (c)}}: The equivalences follow immediately from (a)
and (b) combined with Proposition \ref{prop:Hinfty-sfe}.
 \epf

We finish this section by pointing out two further equivalences to
the ones of Theorem \ref{thm:Kato} and their one-sided extensions in Theorem
\ref{thm:Kato-extended}.

The conditions (1)-(4) of Theorem \ref{thm:Kato} are
equivalent to
\ben
\item[\rm(5)] $\D(\sqrt{\uLB}) = \D(\uD^* B)$ with
$ \|\sqrt{\uLB} F\|_p \eqsim \|\uD^* B F\|_p$ for $F\in \D(\sqrt{\uLB})$;
\item[\rm(6)] $\Dom(\sqrt{\uAB}) = \Dom(V^* B)$ with
$ \|\sqrt{\uAB}u\|_p \eqsim \|V^* B u\|_p$ for $u\in \Dom(\sqrt{\uAB})$.
\een
Here, in the spirit of Theorem \ref{thm:Kato}, we interpret $\uAB$ as an
operator in $\ov{\Ran(V)}$. This is immaterial, however, in view of the
definition $\uAB = V V\s B$ and the (not
necessarily orthogonal) Hodge
decomposition $\H = \ov{\Ran(V)} \oplus \Null(V^* B)$ (see \eqref{eq:HodgeH})
by virtue of which (6) also holds on the full space $\H$.

To see that (1) implies (5), note that
for $f \in \D(\LB)$ we have
$$\|(\uD^* B) \uD f\|_p
  = \| \LB f\|_p
  \eqsim \| \uD \sqrt{\LB} f \|_p
  = \| \sqrt{\uLB}\uD f \|_p.$$
Since $\uD( \D(\LB))$ is a core for both $\D(\uD^* B)$ and
$\D(\sqrt{\uLB}),$ (5) follows. The converse implication that (5) implies (1)
is proved similarly. The equivalence (3)$\Leftrightarrow$(6) is proved in the
same way. It is clear from the proofs that the one-sided versions of these
implications hold as well.

 \section{Proof of Theorem \ref{thm:DL}}\label{sec:DL}

We continue with the proof of Theorem \ref{thm:DL}.
It will be a standing assumption that the equivalent conditions
of Theorem \ref{thm:Kato} are satisfied. As we have already observed (in Lemma
\ref{lem:HilbertspaceKato}, see also the
discussion below Theorem \ref{thm:Kato}),
the corresponding equivalences obtained by replacing $B$ with $B\s$ then also
hold.

Below, for $k=1,2$ we will use the bounded analytic $C_0$-semigroups
$$
 \uPB_{(k)}(t) := \PB(t) \ot \uSB^{\ot k}(t),
$$
which are defined on the spaces
$$
 \uL_{(k)}^p := L^p(E,\mu;\H^{\ot k}).
$$
Note that $\uL_{(1)}^p = \uL^p$ and $\uPB_{(1)}$ coincides with
$\uPB$ on the closed subspace $\ov{\Ran_p(\uD)}$. The generators of
$\uPB_{(k)}$ will be denoted by $-\uLB_{(k)}$. The semigroups
generated by $-\sqrt{I+\uLB_{(k)}}$ will be denoted by
$\uQB_{(k)}.$

We also consider the operator $\uD \ot I,$ initially defined on the
algebraic tensor product $\D(\uD) \ot \H$, which is viewed as a dense subspace
of $\uL_{(1)}^p$. Using that $\uD$ is a
closed operator from $L^p$ into $\uL^p,$ it is straightforward to
check that $\uD \ot I$ extends to a closed operator
\begin{align*}
  \uuD : \D(\uuD) \subseteq \uL_{(1)}^p \to \uL_{(2)}^p.
\end{align*}

On the algebraic tensor product $L^p \ot \H$, for $t>0$ we define
the operators
$$ \begin{aligned}
\uuD \uPB_{(1)}(t) &:=    (\uD \PB(t))  \ot \uSB(t), \\ \uuD
\uPB_{(1)}^*(t)
   & =   (  \uD \PB^*(t) ) \ot \uSB^*(t).
\end{aligned}\label{eq:DPformulas}
$$
By Theorem \ref{thm:Rgradientestimates} these
operators extend uniquely to bounded operators from
$\uL_{(1)}^p$ to $\uL_{(2)}^p$.

 \bpr \label{prop:Hvalued}
Let $1< p < \infty.$
 \ben
  \item[\rm(i)]
The collections
$  \{ \sqrt{t} \uuD \uPB_{(1)}(t)  : t >0 \}$ and
$  \{ \sqrt{t} \uuD \uPB_{(1)}^*(t)  : t >0 \}$
are $R$-bounded in $\calL(\uL_{(1)}^p,\uL_{(2)}^p).$
  \item[\rm(ii)]
The following square function estimates hold for $F \in
\uL_{(1)}^p$:
 \begin{align*}
\Big\| \Big( \int_0^\infty \| \sqrt{t}\uuD \uPB_{(1)}(t) F \|^2
     \;\frac{dt}{t} \Big)^{{1/2}} \Big\|_p  &\lesssim \|F\|_p, \\
\Big\| \Big( \int_0^\infty \| \sqrt{t}\uuD \uPB_{(1)}^*(t) F \|^2
     \;\frac{dt}{t} \Big)^{{1/2}}\Big\|_p  &\lesssim \|F\|_p.
 \end{align*}
  \item[\rm(iii)]
The domain inclusions $\D(\sqrt{\uLB_{(1)}}) \subseteq \D(\uuD)$ and
$\D(\sqrt{\uL_{(1)}^*}) \subseteq \D(\uuD)$ hold with norm estimates
 \begin{align*}
\| \uuD F\|_{p} &\lesssim  \| F \|_p +  \| \sqrt{\uLB_{(1)}}F\n_{p}
\\
 \text{and}\quad \| \uuD F\|_{p}&\lesssim
    \| F \|_p + \| \sqrt{\uLB_{(1)}^*}F\n_{p}.
 \end{align*}
 \een
 \epr

 \bpf
(i): \  The $R$-boundedness is a consequence from (an easy Hilbert
space-valued extension of) Proposition \ref{prop:Rtensor} combined with
\eqref{eq:DPformulas} and Theorem
\ref{thm:Rgradientestimates}.

(ii): \  Since $\uA$ has a bounded $H^\infty$-calculus on $\H$ of
angle $< \frac12 \pi$, the same holds for $\uA^*.$ Proposition
\ref{prop:tensorcalculus} implies that
$\uL_{(1)}$ and $\uL_{(1)}^*$
have bounded $H^\infty$-functional calculi on $\uL_{(1)}^p$ of angle $< \frac12 \pi.$
The domain inclusions $\D(\uL_{(1)}) \subseteq \D(\uuD)$ and $\D(\uL_{(1)}^*)
\subseteq \D(\uuD)$ follow from (i) by taking Laplace transforms.
By combining (i) and Proposition \ref{prop:sq-fc1} we obtain the
desired result.

(iii): \  Combining the fact that $\sqrt{I+\uL_{(2)}}$ has a bounded
$H^\infty$-calculus of angle $< \frac12 \pi$ with Proposition \ref{prop:Hinfty-sfe},
the commutation relation $\uuD
\uP_{(1)}(t) = \uP_{(2)}(t) \uuD,$ the $\H$-valued analogue of Lemma
\ref{lem:HdominatesG}, and the first estimate of (ii), for all $F \in
\D(\uL_{(1)})$ we obtain
 \begin{align*}
  \| \uuD F \|_p
 &  \lesssim
    \Big\|  \Big( \int_0^\infty
    \| t\sqrt{I+\uLB_{(2)}} \uQB_{(2)}(t) \uuD F\|^2
     \,\frac{dt}{t}  \Big)^{1/2} \Big\|_p
  \\&   =
   \Big\|  \Big( \int_0^\infty \| t \uuD  \uQB_{(1)}(t)
    \sqrt{I+\uLB_{(1)}}F\|^2
  \,\frac{dt}{t}  \Big)^{1/2} \Big\|_p
 \\ &  \leq  \Big\| \Big( \int_0^\infty \| \sqrt{t}\uuD
  e^{-t} \uPB_{(1)}(t)  \sqrt{I+\uLB_{(1)}} F \|^2
   \;\frac{dt}{t} \Big)^{{1/2}} \Big\|_p
 \\ &  \leq  \Big\| \Big( \int_0^\infty \| \sqrt{t}\uuD
  \uPB_{(1)}(t)  \sqrt{I+\uLB_{(1)}} F \|^2
   \;\frac{dt}{t} \Big)^{{1/2}} \Big\|_p
 \\ & \lesssim  \big\|  \sqrt{I+\uLB_{(1)}}F \big\|_p
 \\ & \eqsim \|F\|_p +   \big\|  \sqrt{\uLB_{(1)}}F \big\|_p.
 \end{align*}
This gives the first estimate. Since $\D(\uLB_{(1)})$ is a core
for $\D( \sqrt{ \uLB_{(1)} }  ),$ the domain inclusion follows as
well.

 To prove the second estimate we put $T := \PB^* \ot
{\uSB}_* \ot \uSB^*,$ where ${\uSB}_*$ is the bounded analytic
semigroup generated by $- V V^* B^*$; this notation is as in Section \ref{sec:Kato}.
Note that the negative
generator $C$ of $T$ has a bounded $H^\infty$-calculus of angle $<\frac12\pi$; this follows from
the fact that if Theorem \ref{thm:Kato} holds for $B$, then it also holds for
$B\s$ (see Lemma \ref{lem:HilbertspaceKato}) and therefore the negative
generators of $\uS\s$ and $\uS_*$ both have bounded $H^\infty$-calculi of angle $<\frac12\pi$. Let $R$ be
the semigroup generated by $-\sqrt{I+C}.$ Using the identity
 \begin{align*}
 \uuD \uPB_{(1)}^*(t) F = T(t) \uuD F,
 \end{align*}
and arguing as above, for all $F \in \D(\uL_{(1)}^*)$ we obtain
 \begin{align*}
  \| \uuD F \|_p
 &  \lesssim
    \Big\|  \Big( \int_0^\infty
    \| t\sqrt{I+C} R(t) \uuD F\|^2  \,\frac{dt}{t}
      \Big)^{1/2} \Big\|_p
    \\&   =
   \Big\|  \Big( \int_0^\infty \| t \uuD  \uQB_{(1)}^*(t)
    \sqrt{I+\uLB_{(1)}^*}F\|^2
  \,\frac{dt}{t}  \Big)^{1/2} \Big\|_p
 \\ &  \leq  \Big\| \Big( \int_0^\infty \| \sqrt{t}\uuD
   e^{-t}\uPB_{(1)}^*(t)  \sqrt{I+\uLB_{(1)}^*} F \|^2  \;\frac{dt}{t}
        \Big)^{{1/2}} \Big\|_p
 \\ &  \leq  \Big\| \Big( \int_0^\infty \| \sqrt{t}\uuD
       \uPB_{(1)}^*(t)  \sqrt{I+\uLB_{(1)}^*} F \|^2  \;\frac{dt}{t}
        \Big)^{{1/2}} \Big\|_p
 \\ & \lesssim  \big\|  \sqrt{I+\uLB_{(1)}^*}F \big\|_p
 \\ & \eqsim \|F\|_p +   \big\|  \sqrt{\uLB_{(1)}^*}F \big\|_p .
 \end{align*}
The second domain inclusion now follows from the fact that $\D(\uL_{(1)}^*)$ is
a core for $\D(\sqrt{\uL_{(1)}^*}).$
 \epf

In the following theorem we give a characterisation of
$\D(\sqrt{\uLB_{(1)}}).$ Since $\sqrt{\uLB} = \sqrt{\uLB_{(1)}}$ on
$\ov{\Ran_p(\uD)},$ this gives a further equivalence of norms for
$\sqrt{\uLB}$ on $\ov{\Ran_p(\uD)}$, different from the one in Theorem
\ref{thm:Kato}. In the proof of Theorem \ref{thm:DL}
we use both equivalences to determine the domain of $\LB.$

First we need a simple lemma.

 \blem \label{lem:invariant}
Let $1<p<\infty.$ The semigroup $\uQB_{(1)}$  restricts to
$C_0$-semigroups on the space $\D(\uuD) \cap \D(\sqrt{I \ot \uAB})$.
 \elem

 \bpf
It suffices to prove the result with $\uQB_{(1)}$ replaced by
$\uPB_{(1)}$; the latter is readily seen to restrict to a
$C_0$-semigroup on $\D(\uuD) \cap \D(\sqrt{I \ot \uAB})$ by the identities
$\uuD \PB_{(1)}(t) = \uPB_{(2)}(t)\uuD  $ and $ \sqrt{I\ot
\uAB}\,\uPB_{(1)}(t) = \uPB_{(1)}(t)\sqrt{I\ot \uAB}$.
 \epf

 \bthm \label{thm:sqrtuLBdomain}
Let $1<p<\infty.$ We have equality of domains
 \begin{align*}
 \D(\sqrt{\uLB_{(1)}}) & = \D(\uuD) \cap \D(\sqrt{I \ot \uAB}),
  \end{align*}
with equivalence of norms
 \begin{align*}
  \|F\|_p
 + \big\|\sqrt{\uLB_{(1)}} F \big\|_p & \eqsim
 \|F\|_p +  \|\uuD F\|_p +  \|\sqrt{I \ot \uAB} F\|_p,
 \end{align*}
 \ethm

 \bpf
By a result of Kalton and Weis \cite[Theorem 6.3]{KW}, applied to the sums
$\uLB_{(1)} = L\ot I + I\ot \uAB$ and $\uLB_{(1)}\s = L\s\ot I + I\ot \uAB\s$, we have
the estimates
 \begin{align*}
 \|  (I \ot\uAB) F\|_p &\lesssim
 \|F\|_p + \| \uLB_{(1)} F\|_p, \quad F \in \D(\uLB_{(1)})\\
 \|  (I \ot\uAB^*) F\|_p &\lesssim
 \|F\|_p + \| \uLB_{(1)}^* F\|_p, \quad F \in \D(\uLB_{(1)}^*).
 \end{align*}
Since the square root domains equal the complex interpolation spaces
at exponent $\frac12$ for sectorial operators with bounded imaginary
powers \cite[Theorem 6.6.9]{Haase}, by interpolating the inclusions
$$
 \D(\uLB_{(1)}) \embed \D(I \ot\uAB),  \quad
 \D(\uLB_{(1)}^*) \embed \D(I \ot\uAB^*),
$$
with the identity operator,
we obtain the estimates
\beq
\label{eq:dualest} \bal
 \|\sqrt{I \ot \uAB} F\|_p &\lesssim \|F\|_p
 + \big\|\sqrt{\uLB_{(1)}} F \big\|_p, \quad F \in \D(\uLB_{(1)})\\
  \|\sqrt{I \ot \uAB^*} F\|_p &\lesssim \|F\|_p
 + \big\|\sqrt{\uLB_{(1)}^*} F \big\|_p, \quad F \in \D(\uLB_{(1)}^*).
 \eal
\eeq
Combining these estimates with Proposition \ref{prop:Hvalued} we obtain
 \begin{align*}
 \|F\|_p +  \|\uuD F\|_p +  \|\sqrt{I \ot \uAB} F\|_p &\lesssim \|F\|_p
 + \big\|\sqrt{\uLB_{(1)}} F \big\|_p, \quad F \in \D(\uLB_{(1)})\\
 \|F\|_p +  \|\uuD F\|_p +  \|\sqrt{I \ot \uAB^*}
   F\|_p &\lesssim \|F\|_p
 + \big\|\sqrt{\uLB_{(1)}^*} F \big\|_p, \quad F \in \D(\uLB_{(1)}^*).
 \end{align*}

Next we prove the reverse estimates. For $F \in \D(\LB) \ot
\D( \uAB)$ and $G \in \Dom_q(\LB^*) \ot \D(\uAB^*)$ ($\frac1p+\frac1q =1$) we
have $F\in \D(\uLB_{(1)})$, $G\in \Dom_q(\uLB_{(1)}\s)$,
and
 \begin{align*}
   \ip{\sqrt{I+\uLB_{(1)}} F, G}
    & = \ip{(I+\uLB_{(1)}) F,  1/\sqrt{\uLB_{(1)}^* + I}G}
  \\ & = \ip{ F, 1/\sqrt{I+\uLB_{(1)}^*}G}
   + \ip{(\LB \ot I) F, 1/\sqrt{I+\uLB_{(1)}^*}G}
  \\ & \qquad + \ip{(I \ot \uAB) F, 1/\sqrt{I+\uLB_{(1)}^*}G}
  \\& = \ip{ F, 1/\sqrt{I+\uLB_{(1)}^*}G}
  +
  \ip{B\uuD F, \uuD /\sqrt{I+\uLB_{(1)}^*}G}
  \\ & \qquad  + \ip{ \sqrt{I \ot \uAB} F,
    \sqrt{I \ot \uAB^*} /\sqrt{I+\uLB_{(1)}^*}G}.
 \end{align*}
Using 
the boundedness of the three operators $1/\sqrt{I+\uLB_{(1)}^*}$,
$\uuD/\sqrt{I+\uLB_{(1)}^*}$ (by Proposition \ref{prop:Hvalued}(iii)), and $\sqrt{I \ot
\uAB^*}/\sqrt{I+\uLB_{(1)}^*}$ (by the second estimate in \eqref{eq:dualest}),
 we find
 \begin{align*}
  \big\| \sqrt{I+\uLB_{(1)}} F \big\|_p
 &= \sup_{\|G\|_q \leq 1}
  |\ip{\sqrt{I+\uLB_{(1)}} F, G}|
 \\&   \leq \sup_{\|G\|_q \leq 1}
  \| F \|_p \, \| 1/\sqrt{I+\uLB_{(1)}^*}G \|_q
  \\ & \qquad  +
 \|B\|\,
 \|\uuD  F \|_p\, \| \uuD/\sqrt{I+\uLB_{(1)}^*} G \|_q
 \\& \qquad    + \|\sqrt{I \ot \uAB} F\|_p \,
        \| \sqrt{I \ot \uAB^*}/\sqrt{I+\uLB_{(1)}^*} G \|_q
 \\&   \lesssim  \| F \|_p + \| \uuD F \|_p
        + \| \sqrt{I \ot \uAB} F\|_p.
 \end{align*}
The estimate
$$ \big\| \sqrt{I+\uLB_{(1)}\s} F \big\|_p \lesssim  \| F \|_p + \| \uuD F \|_p
        + \| \sqrt{I \ot \uAB\s} F\|_p
$$
is proved similarly and will not be needed.

It remains to prove the equality of domains.
Since $\D(\LB)
\ot \Dom(\uAB)$ is a core for $\D(\uLB_{(1)}),$ it is also a core for
$\D(\sqrt{\uLB_{(1)}}).$ Using this, the domain inclusion $
 \D(\sqrt{\uLB_{(1)}}) \subseteq \D(\uuD)
\cap \D(\sqrt{I \ot \uAB})$ follows, and the equivalence of norms
extends to all $F \in \D(\sqrt{\uLB_{(1)}}).$

Again by the equivalence of norms,
$\D(\sqrt{\uLB_{(1)}})$ is closed in $\D(\uuD) \cap \D(\sqrt{I \ot
\uAB}).$ It remains to prove that the inclusion is dense. This follows from
Lemma \ref{lem:invariant}, since for $F \in \D(\uuD) \cap \D(\sqrt{I
\ot \uAB})$ and $t
>0$ we have $\uQB_{(1)}(t) F \in \D(\sqrt{\uLB_{(1)}})$ and
$\uQB_{(1)}(t) F \to F$ in the norm of $\D(\uuD) \cap \D(\sqrt{I \ot
\uAB})$ as $t \downarrow 0.$
 \epf

Recall that $D$ denotes the Malliavin derivative. Since $A$ is a
closed operator, it follows from the results in \cite{GGvN03} that
the operator $\AB D$, initially defined on $\F C_{\rm b
}^1(E;\Dom(\AB))$, is closable as an operator from $L^p$ into
$L^p(E,\mu;H)$ for $1 < p < \infty.$ We denote its closure by
$D_{\AB}.$

We also consider the operator $\uD^2$ defined by
 \begin{align*}
 \D(\uD^2) := \{ f \in \D(\uD) : \uD f \in \D(\uuD)\}, \quad
 \uD^2 := \uuD \uD.
 \end{align*}
It is easy to check that this operator is closed from
$\D(\uD)$ into $L^p(E,\mu;\H^{\ot 2}).$

 \blem\label{lem:Psemigroupintersection}
Let $1<p<\infty.$ The semigroup $\PB$ restricts to a $C_0$-semigroup
on the space $\D(D_V^2) \cap \D(D_{\AB}).$
 \elem

 \bpf
An easy argument based on Theorem \ref{thm:uLsemigroup} shows
that $\PB(t) \D(\uD^2) \subseteq \D(\uD^2)$ and
 \begin{align*}
 \quad \uD^2 \PB(t) f :=
\uPB_{(2)}(t) \uD^2 f, \quad f \in \D(\uD^2).
 \end{align*}
Similarly, we have
 $\PB(t) \D(D_A) \subseteq \D(D_A)$ and
$$ D_A \PB(t) f = (\PB(t) \ot \SB(t)) D_A f, \quad f \in \D(D_A).$$
These identities easily imply the result.
 \epf

 \bpf[Proof of Theorem \ref{thm:DL}]
Using the fact that $\D(\LB) \subseteq \D(\uD),$ Proposition
\ref{prop:Lpformula}, the domain equality
$\D(\sqrt{\uL}) = \D(\uD\s B)$ (see (5) at the end of Section \ref{sec:Kato}),
Theorem
\ref{thm:sqrtuLBdomain}, the domain equality
$\Dom(\sqrt{\uAB}) = \Dom(V^*B)$ on $\ov{\Ran(V)}$ (see (6) at the end of
Section \ref{sec:Kato}), and the definition of
$D_{\AB},$ for $f \in \D(\LB)$ we obtain
 \begin{align*}
 \|f\|_p +  \|\LB f\|_p
    &  \eqsim  \|f\|_p +   \| \uD f \|_p + \|\LB f\|_p
  \\&   = \|f\|_p +  \| \uD f \|_p +  \| (\uD^* B) \uD f \|_p
  \\&  \eqsim \|f\|_p +   \| \uD f \|_p +
        \| \sqrt{\uLB} \uD f \|_p
  \\&  \eqsim \|f\|_p + \| \uD f \|_p +
           \| \uD^2 f \|_p + \| \sqrt{\uAB} \uD f \|_p
  \\&  \eqsim \|f\|_p + \| \uD f \|_p +
           \| \uD^2 f \|_p + \| (V^*B) \uD f \|_p
  \\&  \eqsim \|f\|_p +  \| \uD f \|_p +
           \| \uD^2 f \|_p + \| D_{\AB} f \|_p,
 \end{align*}
This proves the equivalence of norms and the domain inclusion
$\D(\LB) \subseteq \D(D_V^2) \cap \D(D_{\AB}).$ To obtain equality
of domains it remains to show that this inclusion is both closed and
dense. Closedness follows easily from the norm estimate and density
follows from Lemma \ref{lem:Psemigroupintersection} in the same way
as in Theorem \ref{thm:sqrtuLBdomain}.
 \epf

Note that Theorem \ref{thm:DL} is natural in view of the expression
 \begin{align*}
  L f(x) &= \uD^* B \uD f(x)\\ &= - \sum_{j,k=1}^n [B V h_j, V h_k]\partial_j \partial_k \varphi(\phi_{h_1}, \ldots, \phi_{h_n})
    + \sum_{j=1}^n \partial_j \varphi(\phi_{h_1}, \ldots,
    \phi_{h_n}) \cdot \phi_{A h_j},
 \end{align*}
which holds for all $f \in \F C_{\rm b}^\infty(E;\Dom(\AB))$ of the
form $f = \varphi(\phi_{h_1}, \ldots, \phi_{h_n}).$

\section{Proofs of Theorems \ref{thm:Hodge-main} and \ref{thm:R-bisectorial-main}}
\label{sec:Hodge}

The first part of Theorem \ref{thm:Hodge-main}
has already been proved in Proposition \ref{prop:rankeridentities}.
We begin with some preparations for the proof of the second part.

Let us denote by $\F P(E;\Dom(V))$ the vector space of all functions
of the form
$p = \varphi(\phi_{h_1},\dots,\phi_{h_n})$ with $h_j\in \Dom(V)$ for
$j=1,\dots,n$ and $\varphi:\R^n\to\R$ a polynomial in $n$ variables.
In the proof of the next proposition we need the following auxiliary result.

\begin{lemma}\label{lem:poldense}
For $1\le p<\infty$, $\F P(E;\Dom(V))$ is a core for $\D(\uD)$.
\end{lemma}
\begin{proof}
A simple approximation argument shows that $\F P(E;\Dom(V)) \subseteq
\D(\uD)$. Thus it suffices to approximate elements of $\F C_{\rm
b}^1(E;\Dom(V))$
 in the graph norm of $\D(\uD)$ with elements of $\F P(E;\Dom(V))$.
Let $f\in \F C_{\rm b}^1(E;\Dom(V))$ be of the form $f =
\varphi(\phi_{h_{1}},\dots,\phi_{h_{n}})$ with $h_j\in \Dom(V)$ for
$j=1,\dots,n$ and $\varphi \in C_{\rm b}^1(\R^n)$. By a Gram-Schmidt
argument we may assume that the elements $h_1, \dots, h_n$ are
orthonormal in $H$. Taking Borel versions of the functions $x\mapsto
\phi_{h_j}(x)$, the image measure of $\mu$ under the transformation
$x\mapsto (\phi_{h_{1}}(x),\dots.\phi_{h_{n}}(x))$ is the standard
Gaussian measure $\ga_n$ on $\R^n$.

This reduces the problem to finding polynomials $p_k$ in $n$
variables such that  $p_k\to \varphi $ in $L^p(\R^n,\ga_n)$ and
$\nabla p_k\to \nabla\varphi$ in $L^p(\R^n,\ga_n;\R^n)$. It is a
classical fact that such polynomials exist.
\end{proof}

In the remainder of this section we interpret $\uD\s B$ as a closed densely
defined operator from $\uL^p$ to $L^p$.

\bpf[Proof of Theorem \ref{thm:Hodge-main}, second part]
We shall prove separately that
\begin{align}
 \ov{\Ran_p(\uD)} + \Null_p(\uD\s B) & = \uL^p, \label{eq1}\\
\ov{\Ran_p(\uD)} \cap \Null_p(\uD\s B) & =  \{0\}. \label{eq2}
\end{align}
The proof of \eqref{eq1} is more or less standard. The idea behind
the proof of \eqref{eq2}  is to note that for $p=2$ the Hodge
decomposition is obtained as a special case of the Hodge
decomposition theorem of Axelsson, Keith, and M$^{\rm c}$Intosh
\cite{AKM}, and to use this fact together with the fact that the
$L^p$-norm and $L^2$-norm are equivalent on each summand in the
Wiener-It\^o decomposition.

We begin with the proof of \eqref{eq1}. By Theorem
\ref{thm:Kato}(1) the operator $R := \uD/\sqrt{\LB}$ is well
defined on $\Ran_p(\sqrt{\LB})$ and bounded. In view of the
decomposition $L^p = \ov{\Ran_p(\sqrt{\LB})} \oplus
\Null_p(\sqrt{\LB})$ we may extend $R$ to $L^p$ by putting
$R|_{\Null_p(\sqrt{\LB})} := 0$. A similar remark applies to
the operator $R_* := \uD/\sqrt{\LB^*}.$

For $F\in \uL^p$ we claim that $R R_*\s F\in
\ov{\Ran_p(\uD)}$, where $R_*\s:=(R_*)\s$. Indeed, there exists $f \in
\Null_p(\sqrt{\LB})$ and a sequence $f_n \in \D(\sqrt{\LB})$
such that $f + \sqrt{\LB}f_n \to R_*^* F$ in $L^p.$ Therefore
$ R R_*^* F = \lim_{n\to \infty} \uD f_n \in
\ov{\Ran_p(\uD)}.$

Now, for functions
 $\psi \in\D(\sqrt{\LB})$ and $\phi \in \Dom_q(\sqrt{\LB^*})$,
$$ \lb \uD\psi, B^*\uD\phi\rb
= \ip{\LB \psi, \phi} = \lb \sqrt{\LB}\psi,
\sqrt{\LB^*}\phi\rb.
$$
Furthermore, approximating a function $f\in L^p$ by a sequence
$(f_0 + \sqrt{\LB} f_n)_{n\ge 1}$ with $f_0 \in
\Null_p(\sqrt{\LB})$ and $f_n \in \D(\sqrt{\LB})$ we obtain
 \begin{align*}
 \ip{R f, B^*\uD\phi }
 & = \lim_{n\to \infty} \ip{\uD f_n, B^*\uD \phi}
 \\ & = \lim_{n\to \infty} \ip{\sqrt{\LB} f_n, \sqrt{\LB ^*} \phi}
 \\ & =  \ip{f - f_0, \sqrt{\LB ^*} \phi}
 \\ & =  \ip{f , \sqrt{\LB^*} \phi}.
 \end{align*}
Hence for the duality between $L^p$ and $L^q$  we obtain
$$
\lb F - R R_*^* B F, B^*\uD \phi\rb
   = \lb F, B^*\uD \phi\rb - \lb F,B^* R_*\sqrt{\LB ^*} \phi\rb
     =0.
$$
This shows that $F-R R_*\s B F\in \Null_p(\uD\s B)$. This
completes the proof of \eqref{eq1}.

We continue with the proof of \eqref{eq2}.
Assume that $G\in \D(\uD \s B)$ satisfies $\uD \s B G = 0$. Then for
all $f\in \Dq(\uD)$ we have $\lb B^* \uD f, G\rb =0$, where the
duality is between $\uL^q$ and $\uL^p$.

Let $I_{p,m}$ and $I_{q,m}$ denote the projections in $L^p$
and $L^q$ onto the $m$-th Wiener-It\^o chaoses. The ranges of
$I_{p,m}$ and $I_{q,m}$ are isomorphic by the equivalence of norms on the
Wiener-It\^o chaoses.
Note that $I_{p,m}\s =
I_{q,m}$. Then $I_{p,m}\otimes I$ and $I_{q,m}\otimes I$ are
bounded projections in $\uL^p$ and $\uL^q.$
Let $j_{p,m}$ denote the induced isomorphism of the
range of $I_{p,m}\otimes I$ onto the range of $I_{2,m}\otimes I$.

For cylindrical polynomials $f \in \F P(E;\Dom(V)) \cap
H^{(m)}$ (where $H^{(m)}$ is as in Section \ref{sec:examples})
we have the identity
 $B^*\uD f  = (I_{q, m-1}\ot I)B^*\uD f $
 and
 \beq  \bal \label{eq:chaosm}
 [ j_{p,m-1}(I_{p,m-1}\otimes I) G,B^*\uD f]
 & =   \lb (I_{p,m-1}\otimes I) G,B^*\uD f\rb
 \\ & = \lb G,(I_{q,m-1}\otimes I) B^*\uD f\rb
 \\ & = \lb G,B^*\uD f \rb
 \\ & = 0.
\eal
 \eeq
In the first term, the duality is the inner product of $\uL^2$.

On the other hand, if $f \in \F P(E;\Dom(V)) \cap H^{(n)}$ for
some $n\not=m$, then $j_{p,m-1}\s = j_{q,m-1}$ implies
 \beq\bal   \label{eq:chaosnotm}
\ &  [ j_{p,m-1}(I_{p,m-1}\otimes I) G,B^* \uD  f]
\\ & \qquad = \lb  (I_{p, m-1}\otimes I)G, B^* \uD f\rb
\\ & \qquad = \lb  (I_{p,m-1}\otimes I)G,
              ( I_{q,n-1} \otimes I) B^* \uD f \rb
\\ & \qquad = [j_{p,n-1}(I_{p,n-1}\otimes I)(I_{p,m-1}\otimes I)G,
        B^* \uD f]
\\ & \qquad = 0, \eal
 \eeq
since $\uD f$ is in the $(n-1)$-th chaos; in the last step we used
the $L^2$-orthogonality of the chaoses.

Since the cylindrical polynomials form a core for $\Dom(\uD)$
by Lemma \ref{lem:poldense} and $B$ is bounded
on $\H$, we conclude from \eqref{eq:chaosm}
and \eqref{eq:chaosnotm} that $j_{p,m-1}(I_{m-1}\otimes I) G$
annihilates $\Ran(B\s\uD)$ and therefore it belongs to
$\Null(\uD\s B)$.

Next we claim that if $G\in \overline{\Ran_p(\uD )}$, then
$j_{p,m-1}(I_{p,m-1}\otimes I)G\in \overline{\Ran(\uD)}$. Indeed,
from $G = \lim_{k\to\infty}\uD g_k$ in $\uL^p$ it follows that
\begin{align*}j_{p,m-1}(I_{p,m-1}\otimes I)G
  = \lim_{k\to
\infty} \uD j_{p,m}(I_{p,m}\otimes I)g_k \in \overline{\Ran(\uD )}.
 \end{align*}

Combining what we have proved, we see that if $G\in
\overline{\Ran_p(\uD )}\cap \Null_p(\uD \s B)$,
 then $j_{p,m-1}(I_{p,m-1}\otimes I)G\in \overline{\Ran(\uD)}\cap
\Null(\uD \s B)$. Hence, $j_{p,m-1}(I_{p,m-1}\otimes I)G=0$ by the
Hodge decomposition of $\uL^2$ \cite{AKM}. It follows that
$(I_{p,m-1}\otimes I)G= 0$ for all $m\ge 1$, and therefore $G=0$.
This concludes the proof of \eqref{eq2}. \epf

The next application is included for reasons of completeness.

\begin{corollary}
\label{cor:ranges} If the equivalent conditions of Theorem
\ref{thm:Kato} hold, then $$\ov{\Ran_p(\uD\s B)} =\ov{\Ran_p(\uD\s)}.$$
\end{corollary}
Note that by the second part of Theorem \ref{thm:Hodge-main} it is immaterial
whether we view $\uD\s B$ as an unbounded operator from
$\uL^p$ to $L^p$ or from $\ov{\Ran_p(\uD\s B)}$ to $L^p$.

\begin{proof}
By the first part of Theorem \ref{thm:Hodge-main} (first
applied to $B$ and then to $I$) we have the decompositions
$$ L^p = \Null_p(\uD) \oplus \ov{ \Ran_p(\uD\s B)} = \Null_p(\uD) \oplus
\ov{\Ran_p(\uD\s)},$$ where both  $\uD\s B$ and $\uD\s$ are
viewed as closed densely defined operators from $\ov{\Ran_p(\uD)}$ to
$L^p$. The corollary will follow if we check that $\ov{
\Ran_p(\uD\s B)}\subseteq \ov{ \Ran_p(\uD\s)}$. This inclusion
is trivial if we may interpret $\uD\s B$ and $\uD\s$ as
unbounded operators from $\uL^p$ to $L^p$. By the preceding
remark, we may indeed do so for $\uD\s B$. The proof will be
finished by checking that the conditions of Theorem
\ref{thm:Kato} also hold with $B$ replaced by $I,$ since then
we may do the same for $\uD^*$. But this follows from the fact
that $VV\s$, being self-adjoint on $\ov{\Ran(V)}$, admits a
bounded $H^\infty$-calculus on $\ov{\Ran(V)}.$
\end{proof}

We proceed with the proof of the second part of Theorem
\ref{thm:R-bisectorial-main}.
The first part has been proved in Section
\ref{sec:Kato}.

\begin{proof}[Proof of Theorem \ref{thm:R-bisectorial-main}, second part]
We use the notation $$X_1 := L^p \oplus
\overline{\Ran_p(\uD)} \quad \text{ and } \quad X_2 :=
\Null_p(\uD^*B).$$

Fix $t \in \R \setminus\{0\}.$ First we show that $it-\PiB$ is
injective on $L^p\oplus \uL^p.$ Theorem \ref{thm:Hodge-main} implies
the decomposition
 \begin{align}\label{eq:HodgePiB}
L^p\oplus \uL^p =X_1 \oplus X_2.
 \end{align}
Take $x = x^{(1)}+ x^{(2)} \in X_1 \oplus X_2$, and suppose
that $(it-\PiB)x = 0.$ Then $(it-\PiB)x^{(1)} = 0$ and $it
x^{(2)} = 0.$ Thus $x^{(1)} = x^{(2)} =0,$ since
$\PiB|_{X_1}$ in $X_1$ is bisectorial.

Next we show that $it-\PiB$ is surjective on $L^p\oplus
\uL^p$.  Let $ y^{(1)} \in X_1$ and $y^{(2)}\in
 X_2.$ The equation  $(it-\PiB)(x^{(1)} + x^{(2)}) =
y^{(1)} + y^{(2)}$ is solved by
$$x^{(1)} = (it-\PiB|_{X_1})^{-1}y^{(1)}
 \quad\text{and}\quad x^{(2)} =
(it)^{-1}y^{(2)}.$$ This implies that $it-\PiB$ is surjective.

Using \eqref{eq:HodgePiB} and the sectoriality of $\PiB$ on $X_1$ it
follows that
 \begin{align*}
\|x^{(1)}+x^{(2)}\|
   & \leq \|(it-\PiB|_{X_1})^{-1}\|  \|y^{(1)}\|
     +   |t|^{-1} \|y^{(2)}\|
 \\ &   \lesssim t^{-1} \big(\|y^{(1)}\| + \| y^{(2)}\|\big)
 \\ &  \lesssim  t^{-1} \|y^{(1)}+y^{(2)}\|,
\end{align*}
which is the desired resolvent estimate that shows
 that $\PiB$ is bisectorial on $X_1\oplus X_2$.

To show $R$-bisectoriality of $\PiB$ on $L^p\oplus \uL^p$ we
take $y_j = y_j^{(1)}+ y_j^{(2)} \in X_1 \oplus X_2.$ Let
$(r_j)_{j \geq 1}$ be a Rademacher sequence. Using the
$R$-bisectoriality of $\PiB|_{X_1}$ we obtain
 \begin{align*}
 \E \Big\| \sum_{j=1}^k r_j t_j (it_j-\PiB)^{-1} y_j \Big\|_p
  & \leq \E \Big\| \sum_{j=1}^k r_j t_j
  (it_j-\PiB|_{X_1})^{-1}
          y_j^{(1)} \Big\|_p
  \\ & \qquad \qquad   + \E \Big\| \sum_{j=1}^k r_j t_j
  (it_j-\PiB|_{ X_2 })^{-1}
          y_j^{(2)} \Big\|_p
  \\& \lesssim \E \Big\| \sum_{j=1}^N r_j  y_j^{(1)} \Big\|_p
       + \E \Big\| \sum_{j=1}^k r_j t_j (t_j^{-1}y_j^{(2)}) \Big\|_p
  \\& \lesssim  \E \Big\| \sum_{j=1}^k r_j  y_j \Big\|_p.
 \end{align*}
By an application of the Kahane-Khintchine inequalities we conclude
that $\{ t (it-\PiB)^{-1}\ : t \in \R\setminus\{0\} \}$ is
$R$-bounded on $L^p\oplus \uL^p$. This completes the proof.
\end{proof}

We finish by showing how the first part of Theorem
\ref{thm:R-bisectorial-main} can be used to prove the implication
(2)$\Rightarrow$(1) of Theorem \ref{thm:Kato}. We need the following
lemma, which is an extension of the corresponding Hilbert space
result, cf. \cite[Section (H)]{ADM}.

\begin{proposition}\label{prop:ADM} Let $1<p<\infty$ and suppose
$\calA$ is an $R$-bisectorial operator on a closed subspace
$U$ of $L^p(\mu;\HH)$, where $\HH$ is a Hilbert space. Then
$\calA^2$ is $R$-sectorial and for each $\om \in (0,\frac12
\pi)$ the following assertions are equivalent: \ben
\item[\rm(1)] $\calA$ admits a bounded $H^\infty(\S_\om)$-functional calculus;
\item[\rm(2)] $\calA^2$ admits a bounded $H^\infty(\S_{2\om}^+)$-functional
calculus. \een
\end{proposition}

\begin{proof}
We prove the implication $(2) \Rightarrow(1)$, the other
assertions being well known. Let $\widetilde\psi \in
H_0^\infty(\S_{2\om}^+)$ and define $\psi\in H_0^\infty(\S_\om)$ by
$\psi(z) := \widetilde\psi(z^2).$

Since $\calA^2$ has a bounded $H^\infty(\S_{2\om}^+)$-functional
calculus, by Proposition \ref{prop:Hinfty-sfe} $\calA^2$ satisfies a
square function estimate
$$ c\n f - P_{\Null(\calA^2)} f \n_p
\le \Big\n \Big(\int_0^\infty \n\widetilde\psi
(t\calA^2)f\n^2\,\frac{dt}{t}\Big)^{1/2}\Big\n_p \le
C\n f\n_p,$$ where $P_{\Null(\calA^2)}$ denotes the
projection on $\Null(\calA) = \Null(\calA^2)$ with range
$\ov{\Ran(\calA)} = \ov{\Ran(\calA^2)}.$

But, $\widetilde\psi(t\calA^2) = \psi(\sqrt{t}\calA)$, and therefore
by the substitution $\sqrt{t}=s$ the above estimate is equivalent to
$$ \frac{c}{\sqrt{2}}\n f - P_{\Null(\calA)} f \n_p
\le \Big\n \Big(\int_0^\infty \n \psi
(s\calA)f\n^2\,\frac{ds}{s}\Big)^{1/2}\Big\n_p \le
\frac{C}{\sqrt{2}}\n f\n_p.$$ By the bisectorial version of
Proposition \ref{prop:Hinfty-sfe}, this estimate implies that
$\calA$ has a bounded $H^\infty(\Sigma_\omega)$-functional calculus.
\end{proof}

\begin{proof}[Alternative proof of Theorem \ref{thm:Kato} $(2)\Rightarrow (1)$]
We adapt an argument of \cite{AKM}, where more details can be found.

Consider the function $\sgn \in H^\infty(\S_\om)$ given by
$\sgn(z) = \one_{\S_\om^+}(z) - \one_{\S_\om^-}(z)  =
z/\sqrt{z^2}.$ Since  $\PiB$ has a bounded $H^\infty$-calculus
on $L^p \oplus \ov{\Ran_p(\uD)}$ by Proposition \ref{prop:ADM}, the
operator $\sgn(\PiB)$ is bounded. By the results already proved, this implies that  $\D(\PiB)
= \D(\sqrt{\PiB^2})$ with
 \begin{align*} 
 \n \PiB x \n_p  \eqsim \big\n\sqrt{\PiB^2} x\big\n_p, \quad
x\in \D(\PiB) = \D(\sqrt{\PiB^2}).
 \end{align*}
Clearly, on $L^p \oplus \ov{\Ran_p(\uD)}$
 we have
$$ \sqrt{\PiB^2} = \bma \sqrt{\LB} & 0 \\ 0 & \sqrt{{\uLB}} \ema,$$
and by restricting to elements of the form $x=(f,0)$ with $f\in
\D(\uD )$ we obtain the desired result.
\end{proof}

\medskip

{\em Acknowledgment} -- Part of this work was done while the authors visited
the University of New South Wales (JM) and the Australian National
University (JvN). They thank Ben Goldys at UNSW and Alan M$^{\rm
c}$Intosh at ANU for their kind hospitality. The inspiration for
this paper came from many discussions with Alan on his
recent work on functional calculi for Hodge-Dirac operators.

\bibliographystyle{ams-pln}
\bibliography{riesz_transf}

\end{document}